\pdfoutput=1
\documentclass[11pt,namelimits,sumlimits,a4paper]{amsart}

\usepackage{amssymb,amsmath}
\usepackage[mathscr]{eucal}
\usepackage{slashed}

\usepackage[latin1]{inputenc}
\usepackage[T1]{fontenc}
\usepackage{amsfonts}
\usepackage{fancyhdr}
\usepackage{tikz}
\usepackage{amssymb,amsmath,amscd}
\usepackage{amsthm}
\usepackage{latexsym}
\usepackage{cite}
\usepackage{microtype}
\usepackage{comment}
\usepackage[mathscr]{eucal}
\usepackage{slashed}
\usepackage{times,psfrag}

\usepackage{mathtools}
\usepackage{multirow}
\usepackage{longtable}
\usepackage{bm}
\usepackage[all]{xypic}
\usepackage{float}
\usepackage{esint}
\usepackage{layout}
\usepackage{enumitem}
\usepackage{lmodern}
\usepackage{empheq}
\usepackage[pdftitle={Sp(2)-invariant expanders and shrinkers in Laplacian flow},hidelinks]{hyperref}

\numberwithin{equation}{section}
\newtheorem{theorem}[equation]{Theorem}
\newtheorem{lemma}[equation]{Lemma}
\newtheorem{prop}[equation]{Proposition}
\newtheorem{corollary}[equation]{Corollary}

\theoremstyle{definition}
\newtheorem{definition}[equation]{Definition}
\newtheorem{example}[equation]{Example}

\theoremstyle{remark}
\newtheorem{remark}[equation]{Remark}
\newtheorem*{remark*}{Remark}

\newtheorem{conjecture}[equation]{Conjecture}
\newcommand{\abs}[1]{\lvert#1\rvert}

\newcommand{\ie}{\emph{i.e.} }
\newcommand{\eg}{\emph{e.g.} }
\newcommand{\cf}{\emph{cf.} }

\newcommand{\beq}{\begin{equation}}
\newcommand{\eeq}{\end{equation}}
\newcommand{\bea}{\begin{eqnarray}}
\newcommand{\eea}{\end{eqnarray}}

\newcommand{\R}{\mathbb{R}}
\newcommand{\Rpos}{\mathbb{R}_{>0}}

\newcommand{\Z}{\mathbb{Z}}

\newcommand{\T}{\mathbb{T}}

\newcommand{\CP}{\mathbb{CP}}

\newcommand{\Sph}{\mathbb{S}}

\newcommand{\vol}{\operatorname{Vol}}

\newcommand{\half}{\tfrac12}

\newcommand{\pd}[2]{\frac{\partial #1}{\partial #2}}

\newcommand{\scfamsymb}{\mathcal{S}}
\newcommand{\scfam}[3]{\scfamsymb_{#1}^{#3}(#2)}
\newcommand{\spfam}[2]{\scfamsymb_{#1}(#2)}
\newcommand{\lmap}{\mathcal{L}}
\newcommand{\texorpdflmap}{\texorpdfstring{$\lmap$}{L}}

\newcommand{\iwa}{M}
\newcommand{\tu}[1]{\textup{#1}}

\newcommand{\gtwo}{\ensuremath{\textup{G}_2}}
\newcommand{\texorpdfgtwo}{\texorpdfstring{\ensuremath{\textup{G}_2}}{G\_2}}

\newcommand{\gtstr}{\gtwo--structure}
\newcommand{\texorpdfgtstr}{\texorpdfstring{\gtwo}{G\_2}--structure}

\newcommand{\gtmetric}{\gtwo--metric}
\newcommand{\gthol}{\gtwo--holonomy\ }

\newcommand{\hk}{hyperK\"ahler}

\newcommand{\wt}[1]{\widetilde #1}
\newcommand{\wh}[1]{\widehat #1}

\newcommand{\unitary}[1]{\textup{U$(#1)$}}

\newcommand{\sunitary}[1]{\textup{SU$(#1)$}}
\newcommand{\texorpdfsunitary}[1]{\texorpdfstring{\textup{SU$(#1)$}}{SU(#1)}}

\newcommand{\lorth}[1]{\ensuremath{\mathfrak{so}(#1)}}
\newcommand{\Sp}[1]{\textup{Sp$(#1)$}}
\newcommand{\lsp}[1]{\ensuremath{\mathfrak{sp}(#1)}}
\newcommand{\texorpdfSp}[1]{\texorpdfstring{\textup{Sp$(#1)$}}{Sp(#1)}}
\newcommand{\orth}[1]{\textup{O$(#1)$}}

\newcommand{{\isomgtc}}{\ensuremath{\sunitary{2}^3 \rtimes S_3}}

\DeclareMathOperator{\Div}{div}
\newcommand{\slead}{P_0}
\newcommand{\sother}{P_2}

\newcommand{\sis}{\mathscr{S}}

\newcounter{mtheorem}

\newtheoremstyle{mystyle}%
  {}%
  {}%
  {\itshape}%
  {}%
  {\bfseries}%
  {.}%
  { }%
  {}%

\theoremstyle{mystyle}
\newtheorem{mtheorem}[mtheorem]{Theorem}

\newtheorem{mprop}[mtheorem]{Proposition}
\newtheorem{mcor}[mtheorem]{Corollary}

\newcounter{mconjecture}
\theoremstyle{mystyle}
\newtheorem{mconj}[mconjecture]{Conjecture}
\begin{document}

\title[\Sp{2}-invariant Laplacian solitons]
{\Sp{2}-invariant expanders and shrinkers in Laplacian flow}
\author{
Mark Haskins
\and
Rowan Juneman
\and
Johannes Nordstr\"om}

\maketitle

\begin{abstract}
We show that the complete \Sp{2}-invariant expanding solitons for
Bryant's Laplacian flow on the anti-self-dual bundle of the 4-sphere form
a 1-parameter family, and
that they are all asymptotically conical (AC). %
We determine their asymptotic cones, and prove that this cone
determines the complete expander (up to scale). %
Neither the unique \Sp{2}-invariant  torsion-free
\gtwo-cone nor the asymptotic cone of the explicit AC
\Sp{2}-invariant shrinker from \cite{Haskins:Nordstrom:g2soliton1} occurs
as the asymptotic cone of a complete AC \Sp{2}-invariant expander.

We determine all possible end behaviours of \Sp{2}-invariant solitons,
identifying novel forward-complete end solutions for both expanders and shrinkers
with faster-than-Euclidean volume growth.
We conjecture that there exists a $1$-parameter family of complete \sunitary{3}-invariant expanders
on the anti-self-dual bundle of the complex projective plane $\CP^2$ with such asymptotic behaviour. 

We also conjecture that, in contrast to the \Sp{2}-invariant case, there exist complete \sunitary{3}-invariant AC
expanders with asymptotic cone matching that of the explicit AC
\sunitary{3}-invariant shrinker from \cite{Haskins:Nordstrom:g2soliton1}.
The latter conjecture suggests that Laplacian flow may naturally implement 
a type of surgery in which a $\CP^2$ shrinks to a conically singular point,
but after which the flow can be continued smoothly, expanding a topologically different $\CP^2$
from the singularity.

\end{abstract}

\section{Introduction and main results}

This paper continues the systematic study of cohomogeneity-one solitons in Laplacian flow initiated by two of the authors in~\cite{Haskins:Nordstrom:g2soliton1}.
The main contributions of this paper are twofold: 
\begin{enumerate}[left=0.1em]
\item we achieve a definitive understanding of all complete~\Sp{2}-invariant Laplacian expanders,  \ie Laplacian solitons with dilation constant $\lambda>0$, all of which are defined on $\Lambda^2_-\Sph^4$; 
\item
we characterise all possible end behaviours of~\Sp{2}-invariant Laplacian solitons for any value of the dilation constant $\lambda$. 
\end{enumerate}
Let us explain in rough terms what we mean by (ii). 
Cohomogeneity-one solitons are governed by a nonlinear system of ODEs that
determines how the precise shape of the principal orbits evolves.
The evolution of the volume of the principal orbits turns out to be strictly monotonic, and
we think of the direction of increasing volume as ``forward''.
We will refer to a forward-inextensible solution 
as an end solution, or simply a (soliton) end. 
We are interested in questions such as:
\begin{itemize}[leftmargin=*]
\item Is the end forward-complete or forward-incomplete (\ie does
the solution have finite or infinite forward lifetime)?
\item What is the asymptotic geometry of a forward-complete end?
\item Which end behaviours are stable, in the sense that the forward evolution
of any small perturbation of any initial data leading to a given end behaviour
exhibits the same end behaviour? 
\end{itemize}
Regarding (i), it turns out that all complete~\Sp{2}-invariant expanders
are asymptotically conical (AC); this motivates us to make a detailed
study of both AC expander and AC shrinker ends more generally.
We also consider the question of existence of AC expanders and AC shrinkers
whose asymptotic cones match, which is relevant for the possibility of flows
that continue smoothly after the formation of an isolated conical singularity.

Regarding (ii), we discover that~\Sp{2}-invariant solitons can also
have more exotic forward-complete end behaviours; while the detailed asymptotic
geometry of these exotic ends turns out to be quite different in the shrinker
and expander cases, a common feature is that both can be thought of as
occurring at the boundary of the space of AC soliton ends. 

We also determine the asymptotic behaviour of all forward-incomplete~\Sp{2}-invariant solitons 
and moreover show that this sort of end behaviour is always generic. 

\textheight 23.15cm

\subsection{Laplacian flow and its solitons}

Laplacian flow is a geometric flow of closed~\gtwo-structures on a $7$-manifold that evolves a closed positive $3$-form in the direction of its Hodge-Laplacian. 
Stationary points of the flow are precisely the torsion-free~\gtstr s. Studying the long-time existence and convergence of solutions to Laplacian flow provides a potential parabolic PDE approach to the construction of such torsion-free~\gtstr s.

A Laplacian soliton on a $7$-manifold $M$ is a triple $(\varphi,X,\lambda)$ consisting of a positive $3$-form $\varphi$, a vector field $X$ and a constant (the dilation constant) $\lambda \in \R$ satisfying the overdetermined nonlinear diffeomorphism-invariant system of PDEs 
\[
d \varphi =0, \quad \Delta_\varphi \varphi = \lambda \varphi + \mathcal{L}_X \varphi.
\]
Laplacian solitons are in one-to-one correspondence with self-similar solutions to Laplacian flow. 

\begin{remark}
\label{rmk:general_scaling}
Changing the length scale of a \gtstr{}~$\varphi$ by a factor $\mu > 0$ corresponds
to multiplying $\varphi$ by $\mu^3$. If $(\varphi,X)$ is a $\lambda$-soliton,
then $(\mu^3 \varphi, \mu^{-2}X)$ is a $\mu^{-2}\lambda$-soliton.
In particular, once we have fixed a value for $\lambda$ the Laplacian soliton PDE system is
not scale-invariant \emph{except} in the steady case $\lambda = 0$.
\end{remark}

No nontrivial compact Laplacian solitons are currently known and indeed compact shrinkers or  nontrivial compact steady solitons are impossible.  
The $3$-form $\varphi$ underlying any expander is necessarily exact, but currently, the authors are not aware 
of an exact~\gtstr~on \emph{any} compact $7$-manifold, even without imposing the soliton equations.
Therefore currently the main focus is the construction and classification 
of complete noncompact Laplacian solitons.

\subsubsection*{Cohomogeneity-one Laplacian solitons}
In the current absence of general analytic methods to produce complete solutions 
to the Laplacian soliton equations (a situation also faced with Ricci solitons, except recently in the K\"ahler setting), two of the authors began a study of cohomogeneity-one Laplacian 
solitons. In that setting the Laplacian soliton PDEs reduce to a system of nonlinear ODEs. 
While there is a long history of studying cohomogeneity-one solitons in Ricci flow and mean curvature flow, 
the study of such solitons in Laplacian flow is much more recent;  it has already revealed some important
differences between Laplacian flow and these other two flows. 

In~\cite{Haskins:Nordstrom:g2soliton1}, two of the authors studied local aspects of the theory of $G$-invariant Laplacian solitons for $G=\Sp{2}$ and $G=\sunitary{3}$ 
for any value of the dilation constant $\lambda$, including understanding which local solutions extend smoothly over a singular orbit. 
Moreover, we exhibited an explicit~\Sp{2}-invariant complete AC gradient shrinker (\ie $\lambda<0$) on $\Lambda^2_-\Sph^4$
asymptotic to a closed (but not torsion-free)~\Sp{2}-invariant cone,
see Example \ref{ex:explicit_shrinker} for details. 

For~\sunitary{3}-invariant steady (\ie $\lambda=0$) Laplacian solitons 
we classified all complete solitons and 
also determined all possible end behaviours. 
The complete steady solitons turn out to be AC, with the 
exception of an explicit solution with exponential volume growth and asymptotically constant negative scalar curvature.
(Neither type of asymptotic geometry occurs for steady Ricci solitons.)

In this paper we instead restrict attention to \Sp{2}-invariant solitons, while allowing $\lambda$ to be non-zero. Indeed, we will almost always restrict attention to non-steady solitons, since~\Sp{2}-invariant steady solitons reduce to a (much simpler) special case of the~\sunitary{3}-invariant steady solitons in~\cite{Haskins:Nordstrom:g2soliton1}.

\subsection{Complete expanders and their asymptotic cones}
\label{ss:intro:complete}

For~\Sp{2}-invariant expanders,  \ie solitons with $\lambda>0$, we obtain comprehensive results about the entire family of complete solutions.

\begin{mtheorem}
\label{mthm:expand:Sp2}
There exists (up to scaling) a $1$-parameter family of complete
$\Sp{2}$-invariant gradient Laplacian expanders on $\Lambda^2_-\Sph^4$.
The scale-invariant parameter is $q = \lambda \sqrt{\vol(\Sph^4)} \in \Rpos$, 
where $\vol(\Sph^4)$ is the volume of the zero section.
Furthermore
\begin{enumerate}[left=0.0em]
\item
 Any complete $\Sp{2}$-invariant Laplacian expander 
 belongs to this family (up to scaling).
\item
Every complete $\Sp{2}$-invariant Laplacian expander
is asymptotic to a unique closed $\Sp{2}$-invariant \gtwo-cone. 
\item
Every complete $\Sp{2}$-invariant Laplacian expander decays to its asymptotic cone at rate $-2$. 
\end{enumerate}
\end{mtheorem}

This result follows from combining Theorems \ref{thm:Sp2:smooth:closure}, \ref{thm:sc:expanders:complete} and
Proposition \ref{prop:reg}. The examples provided by Theorem \ref{mthm:expand:Sp2}, along with those constructed in
\cite{Haskins:Nordstrom:g2soliton1}, are the only currently known examples
of (nontrivial) AC Laplacian solitons.

\smallskip
Theorem~\ref{mthm:asymptotic:limit} below provides a complete description of the possible asymptotic cones of complete~$\Sp{2}$-invariant Laplacian expanders. 
To state it, we first need to set up some notation.

The complement of the zero section in $\Lambda^2_- \Sph^4$ is diffeomorphic
to $\Rpos \times \CP^3$. Given a pair of functions $x, y : \Rpos \to \R$ we can
define an \Sp{2}-invariant $3$-form on $\Rpos \times \CP^3$ by
\begin{equation}
\label{eq:phi}
\varphi = (x^2 \omega_1 + y^2 \omega_2) \wedge dt + xy^2 \alpha,
\end{equation}
where $t$ is the coordinate on $\Rpos$, $\omega_1, \omega_2$ form a basis for the invariant $2$-forms and
$\alpha$ is a closed invariant $3$-form on $\CP^3$. $\varphi$ is a positive 3-form
provided that $x$ and $y$ vanish nowhere (without loss of generality we take
them to be positive). The coordinate vector field $\pd{}{t}$ has unit
length with respect to the induced metric, while $x$ and $y$ can be interpreted
as the scale of the fibre and base, respectively, of the twistor fibration
$\Sph^2 \to \CP^3 \to \Sph^4$.
The 3-form is closed if and only if
\begin{equation}
\label{eq:closure}
\frac{d}{dt} (xy^2) = \half x^2 + y^2,
\end{equation}
and every closed, positive \Sp{2}-invariant 3-form normalising
$\pd{}{t}$ arises this way.

\begin{remark}
It is immediate from~\eqref{eq:closure} that the orbital volume is a strictly increasing function of the geodesic arclength parameter $t$. 
(That it is increasing rather than decreasing is a consequence of the
conventions we have chosen.)
\eqref{eq:closure} forces the orbital volume to vanish in finite backwards $t$. 

For the 3-form \eqref{eq:phi} to be the restriction of a smooth 3-form defined on
$\Lambda^2_- \Sph^4$, \ie to extend smoothly across the zero section $\Sph^4$,
not only must the orbital volume reach zero as $t \to 0$ (which can be arranged by
translation in $t$), but also the $\Sph^2$-fibre size $x \to 0$, while the $\Sph^4$-base size~$y$
has a non-zero limit. %
\end{remark}

The 3-form \eqref{eq:phi} is conical (with a vertex singularity at $t = 0$)
if $x$ and $y$ are linear, say $x = c_1t$ and $y = c_2 t$ (with $c_1, c_2>0$) and the closure
condition \eqref{eq:closure} then becomes
\begin{equation}
\label{eq:closed_cone}
6c_1c_2^2 = c_1^2 + 2c_2^2.
\end{equation}
Note that  equation~\eqref{eq:closed_cone} essentially amounts to a particular scale normalisation: for every $\ell > 0$, there
is a unique solution $(c_1, c_2)$ of it such that $\frac{c_2}{c_1} = \ell$. In fact, we can explicitly express
\begin{equation}
\label{eq:explicit_cone_solution}
c_1 = \tfrac{1}{6} \left(2 + \ell^{-2} \right), \quad c_2 =  \tfrac{1}{6} \ell \left(2 + \ell^{-2} \right).
\end{equation}
Hence we can parametrise closed \Sp{2}-invariant cones by $\ell>0$.
The closed cone with $\ell = 1$ (so $c_1= c_2=\tfrac{1}{2}$) is the unique 
$\Sp{2}$-invariant torsion-free \gtwo-cone.

For a pair of functions $(x,y)$ defined on a half-infinite
interval $(t_0,\infty)$, we say that the closed \Sp{2}-invariant \gtstr{}
in \eqref{eq:phi} is \emph{weakly AC} if
there exists $(c_1,c_2)$, a solution of~\eqref{eq:closed_cone}, so that 
\begin{equation}
\label{eq:weak_ac_def}
\frac{x}{t} \to c_1, \quad \frac{y}{t} \to c_2 \quad \text{as\ } t \to \infty.
\end{equation}

In terms of the functions $x$ and $y$, claiming as in Theorem~\ref{mthm:expand:Sp2}(iii) that a closed
\Sp{2}-invariant \gtstr{} on $\Lambda^2_-\Sph^4$ is asymptotic with rate $-2$
to the closed cone defined by $(c_1, c_2)$, means that (up to translation of $t$)
\[ x = c_1 t + O(t^{-1}),\quad  y = c_2 t + O(t^{-1}) . \]

\begin{definition}
\label{def:lmap}
Define $\lmap : \Rpos \to \Rpos$ as follows:  $\lmap(q)$ is the
parameter $\ell = \frac{c_2}{c_1}$ of the asymptotic cone of
any \Sp{2}-invariant smoothly-closing (and hence AC) expander with scale-invariant parameter
$q = \lambda \sqrt{\vol(\Sph^4)}$.
\end{definition}

This asymptotic limit map $\lmap$ is well-defined since the asymptotic cone of
an AC expander is patently invariant under rescaling. Our next main result
establishes the key properties of $\lmap$. 

\begin{mtheorem}
\label{mthm:asymptotic:limit}
$\lmap$ is a strictly increasing continuous bijection from $\Rpos$
to $(1,\infty)$.
In particular,
\begin{enumerate}[left=0.0em]
\item
Two complete $\Sp{2}$-invariant expanders with the same asymptotic cone agree (up to scaling). 
\item
Every $\Sp{2}$-invariant closed \gtwo-cone with $\ell = \frac{c_2}{c_1}>1$
arises as the asymptotic cone of a unique (up to scale) complete $\Sp{2}$-invariant expander.
\item
No closed $\Sp{2}$-invariant \gtwo-cone with
$\ell = \frac{c_2}{c_1} \le 1$ arises as the asymptotic cone of a
complete $\Sp{2}$-invariant expander. (In particular,  the asymptotic cone of a complete~\Sp{2}-invariant AC expander
is never torsion-free.)
\end{enumerate}
\end{mtheorem}

\subsection{Matching AC expanders with shrinkers and~\texorpdfgtwo~conifold transitions}
\label{ss:AC:matching}

AC expanders in a geometric flow give the simplest models for how the flow can smooth out a conical
singularity.  AC shrinkers, on the other hand,
give the simplest models for how an isolated singularity forms during the flow.
This raises questions  both of existence and of uniqueness of AC expanders with
asymptotic cone matching that of a given AC shrinker.

When no such matching AC expander exists, that raises the question of whether after singularity formation one should try to
continue to run the flow within some class of  singular spaces. 
Whenever there is such a matching AC expander,  it gives  rise to a natural
singular weak solution to the flow: shrinking self-similarly for all negative times, 
forming a conical singularity at time $t=0$, which is then resolved by the smooth
AC expander for positive times. This gives a possible way to continue the flow smoothly after the singularity has developed. 
The existence of multiple matching AC expanders gives the simplest mechanism for the failure of uniqueness of weak
solutions to the flow after a singularity formation. 
(There are some well-known instances of such non-uniqueness of AC expanders
with a given cone in both Ricci flow and mean curvature flow). 

We essentially know of only one AC shrinker with link $\CP^3$,
from \cite[Theorem D]{Haskins:Nordstrom:g2soliton1}.

\begin{example}
\label{ex:explicit_shrinker}
For any $b>0$ let $\lambda = - \frac{9}{4b^2}<0$ and define the following functions 
\[
x = t, \quad y^2 = b^2 + \tfrac{1}{4}t^2, %
\quad u = \frac{3t}{4b^2} + \frac{4t}{4b^2+t^2}, \quad\  t \ge 0.
\]
The closed \Sp{2}-invariant \gtstr~ thus defined by~\eqref{eq:phi}
together with the vector field $X=u \,\partial_t$
defines a complete AC Laplacian shrinker on $\Lambda^2_-\Sph^4$. Its asymptotic cone is the 
closed (but non-torsion-free) \Sp{2}-invariant \gtwo--cone $C_{\tu{shrink}}$ with parameter $\ell=\frac{c_2}{c_1}=\frac{1}{2}$.
\end{example}

As an immediate consequence of Theorem~\ref{mthm:asymptotic:limit} we obtain

\begin{mcor}
\label{thm:shrinker}
There is no complete~\Sp{2}-invariant expander (necessarily defined on $\Lambda^2_-\Sph^4$) whose asymptotic cone 
is $C_\tu{shrink}$ the asymptotic cone of the shrinker defined in Example~\ref{ex:explicit_shrinker}. \\[0.2em]
In particular, there is no~\Sp{2}-invariant ``flow-through solution''
to Laplacian flow
which consists of the AC shrinker for negative time, the cone $C_\tu{shrink}$ at time $0$ and a complete AC expander asymptotic to $C_\tu{shrink}$ 
for all positive times. 
\end{mcor}
AC shrinkers and expanders appear in a variety of geometric flows, \eg in mean curvature flow and Ricci flow.
However, it seems to be rare that there exists an AC shrinker with asymptotic cone~$C$ for which there is no known
AC expander that flows out of $C$.

\smallskip
There is a correspondence between complete AC~\Sp{2}-invariant solitons on $\Lambda^2_-\Sph^4$ and
complete AC solitons on $\Lambda^2_-\CP^2$ with isometry group $\sunitary{3} \times \Z_2$
(see \S\ref{subsec:su3} for details).
Via this correspondence Example \ref{ex:explicit_shrinker} also gives
rise to an $\sunitary{3} \times \Z_2$-invariant AC shrinker on
$\Lambda^2_- \CP^2$; Corollary \ref{thm:shrinker} implies that there is no 
$\sunitary{3} \times \Z_2$-invariant complete AC expander matching the cone of this shrinker. 
However, one can ask whether any $\sunitary{3}$-invariant complete AC expander
can match this shrinker cone, \ie can we find a matching complete AC expander when we drop the extra $\Z_2$-symmetry?
In \S\ref{subsec:su3cones} we explain some evidence for a positive answer. 

\begin{mconj}
\label{C:SU3:shrink:expand:matching}
There exists an \sunitary{3}-invariant (but not $\sunitary{3} \times \Z_2$-invariant) complete AC expander on  $\Lambda^2_- \CP^2$ 
that matches the explicit $\sunitary{3} \times \Z_2$-invariant AC shrinker on
$\Lambda^2_- \CP^2$.
\end{mconj}
There is a subtlety here though: 
in fact, there are three topologically different ways to identify the
complement of the zero section in  $\Lambda^2_-\CP^2$ with the cone on the flag
manifold $\sunitary{3}/\T^2$, distinguished by how $H_4(\CP^2)$ embeds in
$H_4(\sunitary{3}/\T^2)$.
If the AC shrinker is defined on a given one of these variants, then the matching AC expander exists on the other two variants. 
Flowing through the singularity by either of these two options would give rise
to a $G_2$ analogue of a conifold transition as discussed by
Atiyah--Witten~\cite[\S 2.3]{Atiyah:Witten:Mtheory}, collapsing a coassociative
$\CP^2$ and then gluing it back in a topologically different way,
\cf Remark~\ref{rmk:transition}.

\subsection{Soliton ends}

The basic strategy for obtaining complete \Sp{2}-invariant expanders as in
Theorem \ref{mthm:expand:Sp2} follows a pattern familiar from other
cohomogeneity-one problems.  First we understand ``smoothly-closing'' local
solutions, \ie solitons defined on a neighbourhood of the zero section
of~$\Lambda^2_- \Sph^4$ (it is by now well-known that such problems can be understood via a systematic framework, \eg as described by Eschenburg-Wang~\cite{eschenburg:wang}). 
Then by considering the value of such a smoothly-closing solution  at some small but positive $t$, one obtains a regular initial value problem 
for a smooth (in fact, usually real analytic) system of nonlinear ODEs. 

The main difficulty, and here no general systematic methods are known, 
is then to try to understand which of these smoothly-closing solutions have infinite forward lifetime and for any such solution 
to understand its asymptotic geometry in detail. We call this understanding the (eventual) forward-evolution of the soliton ODEs. 

More generally, we call any forward-inextensible solution of the soliton ODE system, a \emph{soliton end}. The lifetime of such a soliton end could be
finite or infinite, and we call it a \emph{forward-incomplete end} or a
\emph{forward-complete end} accordingly.

One of the main new contributions of this paper is to understand all possible soliton end behaviours, both forward-incomplete and forward-complete, %
in the~\Sp{2}-invariant setting; 
in particular, for any value of the dilation constant $\lambda$ we understand all possible soliton end behaviours \emph{without} assuming that it arises from a smoothly-closing soliton. 
Some of these arguments (\ie going beyond the smoothly-closing setting)  are, in fact, needed in the proofs of Theorems \ref{mthm:expand:Sp2} and \ref{mthm:asymptotic:limit}, but we have developed the analysis beyond what is needed for those applications. We have several motivations for doing so. 
\begin{enumerate}[left=0em]
\item
To gain insight into complete~\Sp{2}-invariant \emph{shrinkers} (whose classification remains open). 
\item
To better understand what behaviours to expect for complete~\sunitary{3}-invariant expanders (and shrinkers).
\item
To understand new asymptotic geometries that can arise for Laplacian shrinkers and expanders (and their potential significance). 
\end{enumerate}

The first of our results in the direction of understanding possible end behaviours concerns the AC end behaviour that arises in Theorem \ref{mthm:expand:Sp2}(iii): 
we show that such AC end behaviour is a consequence of just being able to bound the ratio $\frac{y}{x}$ on an end solution. In the rest of the paper this ratio $\frac{y}{x}$ will play a central role, and we will refer to it as the \emph{warping} of the solution. (The unique~\Sp{2}-invariant torsion-free~\gtwo-cone 
has warping $\frac{y}{x} \equiv 1$.) More generally, an important theme within this paper will be to understand the possible behaviours 
of the warping $\frac{y}{x}$ of  a general soliton and, in particular, its asymptotic behaviour on a soliton end. 

\begin{mtheorem}[Bounded warping implies AC with rate $-2$]
\label{mthm:AC}
Let $(x,y)$ be any~\Sp{2}-invariant non-steady soliton end for which
$\log{\frac{y}{x}}$ remains forward-bounded. Then
it is forward complete, and the warping $\frac{y}{x}$ has a limit $\ell \in (0,\infty)$
as $t \to \infty$.
Moreover $x$ and $y$ have the asymptotic expansions
\[
x = c_1 t + X_{-1} t^{-1} + o(t^{-1}), \quad y = c_2 t + Y_{-1} t^{-1} + o(t^{-1}),
\]
where $(c_1,c_2)$ is the unique closed cone with $\frac{c_2}{c_1}=\ell$ (recall~\eqref{eq:explicit_cone_solution} for explicit formulae)
and $X_{-1}$ and $Y_{-1}$ are explicit functions of $\lambda$ and $\ell$. 
In particular, the solution is $C^0$-asymptotic  with rate (at least) $-2$ to the
closed~\Sp{2}-invariant cone $x=c_1 t, \ y=c_2 t$.
\end{mtheorem}

Next we state our main result on the classification of all possible \Sp{2}-invariant non-steady soliton ends. It resembles the trichotomy for \sunitary{3}-invariant steady ends obtained by two of the authors in \cite[Theorem F]{Haskins:Nordstrom:g2soliton1}. %

\begin{mtheorem}
\label{mthm:both_tri}
Any~\Sp{2}-invariant non-steady soliton end satisfies exactly one of the following:
\begin{enumerate}[left=0em]
\item
It is forward complete with $\log{ \frac{y}{x}}$ forward-bounded. Then the end
is asymptotic with rate at least $-2$ to a unique closed~\Sp{2}-invariant~\gtwo-cone.
\item
It is forward complete with $|\log \frac{y}{x}| \to \infty$ as $t \to \infty$, and either
\begin{enumerate}
\item[a.]
for an expander with dilation constant $\lambda>0$, $t x \to \frac{3}{\lambda}$,
and $\frac{\log (xy^2)}{t^2} \to \frac{\lambda}{6}$ as $ t \to \infty$.  In particular, $x \to0$ while $y$ grows quadratic-exponentially in $t$; %
\item[b.] for a shrinker with dilation constant $\lambda<0$, $\frac{\log x}{t} \to \sqrt{\frac{-\lambda}{18}}$ and $\frac{x^2}{y^4} \to -\frac{9}{8\lambda}$ as $t \to \infty$.
In particular, both $x$ and $y$ grow exponentially in $t$ (but at different exponential rates). 
\end{enumerate}
\item
It has a finite maximal forward-existence time $t_*$. Then $\frac{y}{x} \to \infty$ and $x = O(\sqrt{t_*-t})$ and $y=O( (t_*-t)^{-1/4})$  as $t \to t_*$. 
\end{enumerate}
\end{mtheorem}

See Theorems \ref{thm:tri:expanders} and \ref{thm:tri:shrinkers} for more detailed versions of the statements for expander and shrinkers respectively. 

While the phrasing of Theorem~\ref{mthm:both_tri} makes the expander and shrinker cases look very similar, that impression turns out to be rather misleading. One feature they do share is that in both cases the forward-incomplete end behaviour (iii) is a stable end behaviour. 
For shrinkers, this is the only stable end behaviour;  that, in part, accounts for the extra difficulty in classifying complete shrinkers. 

However, for expanders, the AC end behaviour described in (i) is also a stable end behaviour---in the strong sense established in Theorem \ref{mthm:uniform}. 
Thus for expanders there are two types of stable end behaviour and the (quadratic-exponential volume growth) forward-complete ends in (ii.a) form a (\mbox{1-dimensional}) wall separating them. In contrast, for shrinkers, the AC ends in (i) form a \mbox{codimension-$1$} end behaviour, \ie given initial data
leading to an AC shrinker end only a \mbox{codimension-$1$} set of nearby initial data also gives rise to AC shrinker ends. 
The (exponential volume growth) forward-complete shrinker in (ii.b) is essentially unique; it appears as a boundary point of the  ($1$-dimensional) set of all AC shrinker ends.

\subsection{Structure of the paper}
We now describe the structure of the main arguments and of the paper itself.
While the statements of Theorems \ref{mthm:expand:Sp2} and \ref{mthm:asymptotic:limit} concern only
complete expanders, 
Theorems~\ref{mthm:AC} and~\ref{mthm:both_tri} instead 
involve understanding the
forward-evolution of general \Sp{2}-invariant solitons
(not only the ``smoothly-closing'' ones) and so superficially, they may not appear to be closely related. 
However, it turns out that our proofs of the main theorems about complete expanders rely on
results we establish about the forward-evolution of general \Sp{2}-invariant solitons.
\\[0.5em]
\noindent
\emph{Local~\Sp{2}-invariant solitons.}
Section~\ref{s:review} recalls the requisite facts about the theory of local
\Sp{2}-invariant solitons developed in~\cite{Haskins:Nordstrom:g2soliton1}.
For an \Sp{2}-invariant \gtstr{} $\varphi$ as in \eqref{eq:phi}, the codifferential (with respect to the metric induced by $\varphi$) is itself
\Sp{2}-invariant, so can be written as 
\begin{equation}
\label{eq:dstar}
d^*\varphi =  \tau_1\, \omega_1 + \tau_2\, \omega_2,
\end{equation}
for scalar functions $\tau_1$ and $\tau_2$ of $t$. 
If $\varphi$ is closed, then $d^* \varphi$ completely determines the torsion of $\varphi$,
and moreover the coefficient $\tau_1$ is determined algebraically by $\tau_2$
\eqref{eq:tau1:tau2:sp2}.
Then \eqref{eq:closure} and \eqref{eq:dstar} %
can be rearranged as equations determining $x'$ and $y'$ in terms of $x$,
$y$ and $\tau_2$, while the soliton condition corresponds to an equation for
$\tau_2'$.
Thus the Laplacian soliton equations reduce
to a first-order system of ODEs \eqref{eq:ODE:tau2'} for $(x,y,\tau_2)$.
The soliton vector field $X= u\, \partial_t$ can then be reconstructed algebraically from this data. 

\medskip
\noindent
\emph{Completeness of \Sp{2}-invariant solitons.}
In Section~\ref{sec:complete} we begin our study of complete~\Sp{2}-invariant expanders. 
The first step in this direction is to  understand solitons that
close smoothly on~$\Sph^4$:
that requires the parameter $x$ controlling the
size of the fibres of the fibration $\Sph^2 \to \CP^3 \to \Sph^4$ to approach
0 at the correct rate, while $y$, which recall sets the scale of the base,
must converge to a positive value $b$; 
geometrically,  $b^4$ is proportional to $\vol(\Sph^4)$, the volume of the zero-section of $\Lambda^2_-\Sph^4$. 
Usually we want to take into account the rescaling behaviour of solitons
and consider the classification of solitons up to scale. 
To that end, it is natural to introduce the scale-invariant quantity $q:= \lambda b^2$. 
Then Theorem \ref{thm:Sp2:smooth:closure} says that (up to scale) there is a $1$-parameter family of smoothly-closing solitons, 
parametrised by $q \in \R$, with $q > 0$ corresponding to expanders, $q < 0$ to shrinkers and $q=0$ 
to a unique steady (in fact, static) soliton. 

Understanding completeness in the forward direction turns out to boil down
to controlling evolution of the warping $\frac{y}{x}$
(which determines the homothety class of the principal orbits).
For smoothly-closing expanders, the necessary control can be obtained by a fairly elementary argument. Certain combinations of
inequalities that are satisfied near the special orbit $\Sph^4$ are self-perpetuating. Together these inequalities imply that  the warping $\frac{y}{x}$
is both strictly decreasing and  bounded below by $1$, and hence converges to some limit $\ell \ge 1$. 
That suffices (see Proposition \ref{mthm:warping}(i-ii)) for proving
Theorem \ref{thm:sc:expanders:complete}: all smoothly-closing
\Sp{2}-invariant expanders are weakly AC. This establishes all but
item (iii) in Theorem~\ref{mthm:expand:Sp2}.

While it is then clear that the asymptotic cone of any smoothly-closing
expander has to be a closed~\Sp{2}-invariant cone with parameter $\ell \ge 1$,
the question of precisely which closed~\Sp{2}-invariant cones do occur as the
asymptotic cone of some complete AC expander is significantly more subtle. 
We do not settle this question until \S\ref{s:asymptotic:cones}.

\medskip
\noindent
\emph{Controlling the forward-evolution of the warping $\frac{y}{x}$.}
The following result collects a number of important properties of the
warping~$\frac{y}{x}$, which will be used to 
understand the possible end behaviours of general expanders and shrinkers. 

\begin{mprop}
\label{mthm:warping}
Consider an %
\Sp{2}-invariant Laplacian soliton end defined by a pair of functions $(x,y)$
as in~\eqref{eq:phi}. Under forward evolution the warping $\frac{y}{x}$ has the following properties. 
\begin{enumerate}[left=0em]
\item Any forward-incomplete soliton has $\log \frac{y}{x}$ unbounded.
\item Any forward-complete soliton whose warping $\frac{y}{x}$ has a limit $\ell \in (0,\infty)$ as $t \to \infty$ is weakly AC.
\item  The warping $\frac{y}{x}$ of any soliton is eventually monotonic.
In particular, a soliton is weakly AC unless %
$\frac{y}{x} \to \infty$ or $\frac{y}{x} \to 0$.
\item For any expander %
$\frac{y}{x}$ is bounded below.
\item For any shrinker %
$\frac{y}{x}$ is bounded above unless $x \to 0$.
\end{enumerate}
\end{mprop}

\begin{remark*}
To clarify the meaning of some statements in the previous result, let us, for instance, expand on item (iv):
it means, as we evolve forward in $t$, then on any expander the warping $\frac{y}{x}$ remains bounded below
either as $t \to \infty$ or $t \to t_*$ (according to whether the solution has infinite lifetime or has a finite extinction time $t_*$).
\end{remark*}

Parts (i) and (ii) are proved in Section \ref{sec:complete}, (iii) and (iv) in
Section~\ref{sec:warping}, and (v) in Lemma \ref{lem:shrinker:x:unbounded}.
The proofs of the latter involve considering some monotonic quantities that
play important roles throughout the paper.
To begin with, for any~\Sp{2}-invariant soliton one can define the functions 
\[
\tilde \tau_1 := \tau_1 - u x^2, \quad \tilde \tau_2:= \tau_2 - uy^2
\]
where
$\tau_1$ and $\tau_2$ are the coefficients
in the torsion from \eqref{eq:dstar} and $u$ is the coefficient of the
soliton vector field $X=u \,\partial_t$. 
For brevity, we call $\tilde \tau_1$ and $\tilde \tau_2$ the
\emph{adjusted torsion} of the soliton. They show up naturally in
the derivation of the soliton ODEs because they are the components of the
invariant $2$-form $\tau -\iota_X \varphi$, \ie $\tau -\iota_X \varphi = \tilde \tau_1 \omega_1 + \tilde \tau_2 \omega_2$, 
where $\tau$ denotes the torsion $2$-form of the closed~\gtstr~$\varphi$ and $X$ is the soliton vector field. 
The soliton ODEs entail that the components of the adjusted torsion
satisfy
\begin{equation}
\label{eq:intro_adjusted}
 \tilde \tau_1' = \lambda x^2, \quad \tilde \tau_2' = \lambda y^2, \quad \text{and} \quad \tilde \tau_1 + 2\tilde \tau_2 = 2 \lambda g^3,
\end{equation}
and in particular they are monotonic. (Indeed, on steady solitons they are
constant---a major reason that case is so much easier to analyse.)

Further, if we define
\begin{equation}
\label{eq:intro_M}
M: = 3x + \tilde \tau_1
\end{equation}
then $\frac{M}{xy^2}$ turns out to be strictly decreasing (regardless of the
sign of $\lambda$), giving useful long-term control for the $\Sp{2}$-invariant
evolution equations.
(We are not aware of any geometric interpretation of $M$, or of any
analogous quantity in the~\sunitary{3}-invariant setting.)  

\smallskip
One of the uses of the adjusted torsion is to identify a (necessary and
sufficient) criterion guaranteeing that the forward-evolution of given
(regular) initial data leads to a weakly AC expander end. 
For expanders, it is immediate from \eqref{eq:intro_adjusted} that both
$\tilde \tau_1$ and $\tilde \tau_2$ are strictly increasing in $t$,
that at least one of them is positive,
and that on any forward-complete end on which $\log{\frac{y}{x}}$ is forward-bounded
(which is the case for any weakly AC end) eventually both $\tilde \tau_1$ and
$\tilde \tau_2$ must become positive. 
Conversely, Proposition \ref{prop:AC:end} characterises weakly AC expander ends as precisely those for which the increasing function 
$\tilde \tau_1$ is eventually positive. (Since it is easy to show that any
smoothly-closing expander has $\tilde \tau_1>0$ for all $t>0$ this yields a
different proof that such expanders are weakly AC.)

\begin{remark}
\label{rmk:intro_stab}
The condition that $\tilde \tau_1$ is eventually positive is clearly open. In particular, if some initial condition for $(x,y,\tau_2)$
leads to a weakly AC end, then so does any small perturbation of it. (This stability property also means that it is relatively easy to study AC expander ends numerically.)

No analogous condition so simply characterises weakly AC shrinker ends; 
indeed, such ends turn out to be nongeneric, \ie there exist arbitrarily small perturbations of any regular initial data leading to an AC shrinker end 
that exhibit different end behaviour. In particular, the necessary condition that the decreasing functions $\tilde \tau_1$ and $\tilde \tau_2$ are both eventually
negative is not sufficient to guarantee an AC shrinker end.\end{remark}

\noindent
\emph{Non-steady AC ends.}
Section~\ref{s:AC:geometry:Sp2:expanders} improves our understanding of AC ends of non-steady~\Sp{2}-invariant solitons in several directions.  
In particular, we prove Theorem \ref{mthm:AC}: 
any non-steady weakly AC end is, in fact, $C^0$-asymptotic with rate (at least) $-2$ 
to its asymptotic cone. 
Theorem~\ref{mthm:expand:Sp2}(iii) follows as a special case. 

In essence, the origin of the specific value $-2$ for the asymptotic rate can be traced back
to how rescaling a non-steady soliton (recall Remark \ref{rmk:general_scaling})
changes $\lambda$. Whichever way one presents the
soliton ODEs, the terms that involve a factor of $\lambda$ play a special role:
if one scales up all the variables \emph{without} changing $\lambda$, then the
terms involving $\lambda$ start to dominate.
When the soliton ODEs are written in terms of the variables $x$, $y$ and $\tau_2$,
$\lambda$ only appears in the ODE for $\tau_2$, which
takes the schematic form
\[ \tau_2' = a(x,y)\lambda S + b(x,y,\tau_2) , \]
where $a$ and $b$ are rational functions $a, b$ of degree 0 and 
\[
S: = y^2 - x^2 - \frac{3}{2} x \tau_2.
\]
On any non-steady end that is AC in a $C^1$ sense
this forces that $S$ must converge to a constant value. 
Proposition~\ref{prop:S:limit} establishes for any non-steady soliton that, in fact, weakly AC end behaviour
(\ie $C^0$-convergence of the warping $\frac{y}{x}$) already suffices
to imply convergence of $S$. 

Another reason why $S$ is distinguished is that
for any closed \Sp{2}-invariant \gtstr{} $\varphi$ defined by the pair $(x,y)$ it satisfies 
\[
-\frac{S}{xy^2}=\frac{d}{dt} \log{\frac{y}{x}}.
\] In particular, 
$S \equiv 0$ on any closed cone. Because $S$ has dimensions of length-squared,
then if~$\varphi$ is asymptotic to a closed
cone with rate $\nu$ we have $S = O(t^{\nu+2})$. It turns out to be possible to go in the other direction: namely, 
in Proposition \ref{prop:reg} we use the convergence
of $S$ on any non-steady weakly AC end just established to deduce that $\nu = -2$, thus completing the proof of Theorem~\ref{mthm:AC}.
(The argument that $S$ converges on a weakly AC end does not apply when $\lambda=0$. 
Indeed, for non-static AC steady solitons, instead $S$ turns out to  grow
linearly, and  $\nu = -1$ \cite[Theorem E]{Haskins:Nordstrom:g2soliton1}.)

In Section~\ref{ss:AC:expander:end:stability}, using some of the arguments used to prove
Theorem \ref{mthm:AC},  we establish an important strengthening 
of the stability of AC expander ends mentioned in Remark \ref{rmk:intro_stab}.
\begin{mtheorem}
\label{mthm:uniform}
Consider a forward-complete~\Sp{2}-invariant expander $(x,y, \tau_2)$ %
such that $\frac{y}{x}$
converges as $t \to \infty$, and any $t_0$ for which this solution is defined and non-singular,
in the sense that $x(t_0)$ and $y(t_0)$ are both positive.
Then any initial condition $\bar p$ sufficiently
close to $p:=(x(t_0), y(t_0), \tau_2(t_0))$ 
also gives rise to a forward-complete AC
expander
$(\bar x, \bar y, \bar \tau_2)$.
Moreover,
\[ \frac{\bar y}{\bar x} \to \frac{y}{x} \]
uniformly on $[t_0, \infty)$
as $\bar p \to p$.
In particular, if we let $\ell$ and $\bar{\ell}$ denote the respective asymptotic limits of these two solutions, then $\bar{\ell} \to \ell$ as $\bar p \to p$. 
\end{mtheorem}

\medskip
\noindent
\emph{Existence and uniqueness of complete expanders with prescribed asymptotic cone.}
In Section~\ref{s:asymptotic:cones} we switch our focus back to 
complete expanders, the goal now being to prove Theorem~\ref{mthm:asymptotic:limit}.
Some of the main results from Section~\ref{s:AC:geometry:Sp2:expanders}  play key roles in the proofs in this section.
 
Theorem~\ref{mthm:expand:Sp2} already established that the entire family of
smoothly-closing~\Sp{2}-invariant expanders,
parametrised (up to scale) by
$q=\lambda \sqrt{\vol{\Sph^4}} \in \Rpos$, consists of complete solitons each asymptotic
with rate $-2$ to a unique closed~\Sp{2}-invariant cone, which is determined by
$\ell = \lim_{t \to \infty} \frac{y}{x} \ge 1$. 
Theorem~\ref{mthm:uniform} implies that $\ell$
depends continuously on~$q>0$. In other words, the asymptotic limit map
$\lmap : q \mapsto \ell$ from Definition \ref{def:lmap} is continuous and
takes values in $[1,\infty)$. It remains to prove then
that, in fact, $\lmap$ gives a bijection between $\Rpos$ and $(1,\infty)$. 

To show that a smoothly-closing expander cannot be asymptotic to the torsion-free cone, 
\ie that $\lmap(q) >1$ for all $q>0$, the key ingredient is
the large-$t$ asymptotic expansion 
for $x$ and $y$ for a non-steady AC end from %
Proposition~\ref{prop:reg}.

The fact that $\lmap$ is non-decreasing is proven by
establishing a suitable comparison principle for expanders---see Proposition~\ref{prop:comp_g_1}.
Proving that $\lmap$ is \emph{strictly} increasing (established in Theorem~\ref{thm:D*:monotone}) turns out to be more
subtle; it relies on the details of the consequences of Theorem \ref{mthm:AC}, 
applied to pairs of expander ends asymptotic to the same cone.

That $\lmap(q) \to 1$ as
$q \to 0$, \ie that for $q \sim 0$  the asymptotic cone is close to the torsion-free cone, 
follows from making a comparison with
the Bryant--Salamon AC~\gtmetric~on $\Lambda^2_-\Sph^4$, considered as a (static)
steady soliton.

Finally, %
$\lmap(q) \to \infty$ as $q \to \infty$
is proved by comparing the values of the decreasing quantity defined in terms of
\eqref{eq:intro_M} at the singular orbit at $t=0$ and at the AC end when $t \to \infty$. 

\medskip
\noindent
\emph{Constructing end solutions.}
In Sections~\ref{subsec:end_sivp} and~\ref{ss:fc_non_ac} we outline how to
construct three different classes of expander ends and three different classes
of shrinker ends.
In all cases the basic approach is the same, namely we recast the construction of ends with prescribed 
asymptotic behaviour in terms of a singular initial value problem (which may be regular or irregular). 
Implicit in this approach is the identification of appropriate model ends, which are themselves approximate but not exact solutions of the soliton ODEs;  
the singular initial value problems arise
when we describe the perturbation needed to correct the model ends 
to actual soliton ends. 
For instance, for non-steady AC ends
there are obvious model ends, namely any~\Sp{2}-invariant closed cone. In other cases, identifying suitable model ends 
can turn out to be one of the most subtle parts of the problem. 

The forward-incomplete ends turn out to be the 
most robust type of end; they behave in essentially the same way for any value of the dilation constant $\lambda$. On the other hand, AC soliton ends exist for both shrinkers and expanders, but their qualitative nature depends 
on the sign of  $\lambda$, \eg AC expander ends are stable but AC shrinker ends are not. The origin of this difference is that the resulting singular initial value problem is irregular 
and that the sign of the dominant singular term is determined by $\lambda$. We explain the differences between AC expander and shrinker ends in more detail in Sections~\ref{subsec:rigid_vs_instab}
and \ref{subsec:end_sivp}.

In addition, there are exotic forward-complete but non-AC ends, and these ends behave completely differently in the expander and shrinker settings (\cf Theorem \ref{mthm:both_tri}(ii)). 
The reason for this is that the model ends themselves are different for shrinkers and expanders, as we explain further below. 
By Proposition~\ref{mthm:warping}, the warping $\frac{y}{x}$ on any forward-complete non-AC end  must either tend to $0$ 
or to $\infty$; in fact, the former happens only in the shrinker case and the latter only in the expander case. In either case, considering 
the relative sizes of various terms in the soliton ODE system leads one to consider $1$-parameter families of rescaled ODE systems determined by a parameter $\epsilon \ge 0$. Which of the rescalings is relevant differs between the two cases, reflecting the different asymptotic behaviour of the warping $\frac{y}{x}$ 
in the two cases. 

For any positive $\epsilon$ these rescaled ODE systems are equivalent to the~\Sp{2}-invariant soliton ODEs. The limit of the rescaled system as $\epsilon \to 0$ 
is well-defined, but the $\epsilon=0$ limit is no longer equivalent to the original ODE system. 
This rescaled limit ODE system turns out to be simpler than the original system because the defining rescalings 
become exact symmetries of the limit system. By considering ratios of quantities that are invariant 
under these rescalings one thereby gets a reduction from a $3$-dimensional system to a planar ODE system. 
The additional structure present in the limit systems allows us to find distinguished explicit solutions with the sort of end behaviour 
we seek to construct in the~\Sp{2}-invariant soliton ODE system. These explicit solutions differ significantly in the expander and shrinker cases
because they satisfy different limit ODEs. These exact solutions to the limit systems then serve as model ends 
that can be corrected by bounded terms to give genuine forward-complete non-AC~\Sp{2}-invariant ends.
This is analogous to how closed~\Sp{2}-invariant cones serve as model ends for both AC shrinker or expander ends; 
the main difference here is that the model ends themselves are now sensitive to the sign of the dilation constant.

\medskip
\noindent
\emph{End behaviour trichotomies.}
The expander part of the trichotomy stated in Theorem~\ref{mthm:both_tri}
is established in Theorem~\ref{thm:tri:expanders}. %
The decreasing quantity $M/g^3$ discussed earlier plays a key role in its proof:
it turns out that the three types of expander ends can be distinguished by
whether $M$ is positive and eventually increasing, remains positive but
eventually tends to $0$ or is eventually negative.

In the shrinker case, similar ideas lead to
Proposition~\ref{prop:incomplete_shrinkers},
which shows that any shrinker for which $x$ is eventually decreasing is forward-incomplete and its asymptotic behaviour 
as $t$ approaches the extinction time $t_*$ is well-controlled. 

The main result of Section~\ref{sec:shrinker-tri} is Theorem~\ref{thm:tri:shrinkers}, the shrinker part of the trichotomy stated in 
Theorem~\ref{mthm:both_tri}. The arguments in the shrinker case differ dramatically from those used in the expander case from the previous section. 
Here one of the two planar ODE systems mentioned above,  which we call the \emph{shrinker limit system} \eqref{eq:heisenberg:limit}, turns out to play a central role. 
Since the global dynamics of planar ODE systems are relatively tame, %
we are able to completely understand all possible eventual forward-time behaviours of solutions to this shrinker limit system.
The key step, achieved in Section~\ref{ss:shrinker:full}, is then to prove that~\Sp{2}-invariant shrinker ends with $x$ and $\frac{x}{y}$ sufficiently large 
are  well-approximated by solutions to the shrinker limit system.

\subsection{\texorpdfsunitary{3}-invariant expanders and shrinkers}
\label{ss:su3:intro}

The local theory of \Sp{2}-invariant solitons on $\Lambda^2_- \Sph^4$ and
\sunitary{3}-invariant solitons on $\Lambda^2_- \CP^2$ was developed
in parallel in \cite{Haskins:Nordstrom:g2soliton1}---indeed, the equations
governing the former coincide with those of a subfamily of the latter that possess
an additional $\Z_2$ symmetry.
The results in the present paper in the \Sp{2}-invariant case therefore
naturally raise the question to what extent they can be generalised
to the \sunitary{3}-invariant case.
We outline our expectations in the final (more speculative) section of the
paper, %
where, in particular, we formulate three precise conjectures. 

\begin{mconj}
There exist only finitely many~\sunitary{3}-invariant complete AC shrinkers. 
\end{mconj}

This is a (partial) generalisation of our conjecture that
Example \ref{ex:explicit_shrinker} is the unique~\Sp{2}-invariant complete AC
shrinker. 
The heuristic justification is that we can prove that both the smoothly-closing solutions and AC end solutions are 2-dimensional subsets of a 4-dimensional space of local solutions.

\begin{mconj}
\label{mconj:transition} 
There exists a $1$-parameter family of complete~\sunitary{3}-invariant
expanders defined on $\Lambda^2_-\CP^2$ all with quadratic-exponential volume
growth.
\end{mconj}

Among the 2-parameter family of smoothly-closing ~\sunitary{3}-invariant
expanders, considering a
1-parameter family with an extra $\Z_2$-invariance shows that there exist
some solutions that are complete AC (because then the ODEs reduce to those for the
\Sp{2}-invariant solitons in Theorem~\ref{mthm:expand:Sp2}).
On the other hand, numerical simulation suggests that there is a non-empty open subset of 
smoothly-closing solutions with a forward-incomplete end analogous to
those in Theorem \ref{mthm:both_tri}(iii). Conjecture \ref{conj:transition}
describes in more detail our expectation that the transition between
these two types of end behaviour happens along a curve
of complete solutions with end behaviour analogous to that in
Theorem \ref{mthm:both_tri}(ii.a).

\smallskip
The forward-complete non-AC ends have an approximate $\Z_2$ symmetry (that does
not extend to a diffeomorphism of $\Lambda^2_- \CP^2$). That is one of the
motivations for the following prediction (analogous to Theorem
\ref{mthm:asymptotic:limit} from the \Sp{2}-invariant setting).

\begin{mconj}
In the 2-dimensional set of cones with \sunitary{3}-invariant closed \gtstr s,
the boundary of the subset that arises as the asymptotic cones of
\sunitary{3}-invariant expanders on $\Lambda^2_- \CP^2$ consists of
cones with an additional $\Z_2$ symmetry that does not extend to a
diffeomorphism of~$\Lambda^2_- \CP^2$.
\end{mconj}

See Conjecture \ref{conjecture:SU3:hit:cones} for a more precise
statement, and an explanation why that would in particular
imply Conjecture~\ref{C:SU3:shrink:expand:matching}.

\subsection{Further questions}
\label{ss:furtherQ}

The results we have proven in this paper suggest a number of further natural questions
distinct from those about~\sunitary{3}-invariant solitons just described. We conclude the introduction by describing some of 
the most obvious ones. \\[0.5em]
\textbf{Question 1.} 
Do the explicit AC shrinkers on $\Lambda^2_-\Sph^4$ and on $\Lambda^2_-\CP^2$ actually arise as finite-time 
singularity models of Laplacian flow (for non-solitonic data)?\\[0.25em]
Recent numerical simulations strongly suggest that for~\Sp{2}-invariant Laplacian flow the explicit AC shrinker on $\Lambda^2_-\Sph^4$ 
indeed appears as a finite-time singularity model and moreover that forming such a singularity is a relatively stable property in the~\Sp{2}-invariant setting. 

\smallskip

Assuming then that singularities modelled on the explicit AC shrinker on $\Lambda^2_-\Sph^4$ actually occur 
it is natural to ask what are the possible forward evolutions of Laplacian flow that emerge from the~\Sp{2}-invariant shrinker cone. 
The nicest situation would be if there were a smooth solution of the flow that emerges from the~\Sp{2}-invariant shrinker cone. 

We already know from the work of this paper that there is no complete~\Sp{2}-invariant AC expander asymptotic to the shrinker cone. 
This suggests the following question.  \\[0.5em]
\textbf{Question 2.} Must a smooth Laplacian expander asymptotic to an~\Sp{2}-invariant closed cone 
itself be~\Sp{2}-invariant? \\[0.3em]
For gradient AC Laplacian \emph{shrinkers} the inheritance of symmetries of the asymptotic cone has recently been proven 
in complete generality by Haskins--Khan--Payne~\cite{HKP}. 
If such a symmetry inheritance result were also true for continuous symmetries of smooth expanders then, together with the results of this paper, it would imply that there cannot be any complete Laplacian expander 
asymptotic to %
the asymptotic cone of the AC shrinker.
As a word of warning, we should point out though 
that the authors can show that discrete symmetries of the asymptotic cone definitely need not be inherited by AC expander ends (unlike for AC shrinker ends). 

We note that  some related symmetry-inheritance results have been established in the context of codimension $1$ Mean Curvature Flow (MCF).
For instance, it has recently been proven that smooth solutions to MCF that are asymptotic to 
smooth rotationally-invariant double cones are also rotationally invariant. In particular, 
this implies that smooth self-expanders asymptotic to such cones are rotationally invariant~\cite{L:Chen:double:cone1}.
However, the method of proof, motivated by Alexandrov reflection, has no clear analogue in our intrinsic setting.

If indeed there is no smooth AC expander asymptotic to the~\Sp{2}-invariant shrinker cone one  should consider the question of the existence of non-solitonic forward evolutions 
of the shrinker. If there were some smooth forward-evolution under Laplacian flow flowing out of the shrinker cone
then, by parabolic rescaling, it should be possible to extract a smooth AC immortal solution that flows out of the cone C
but which is itself not an expander. 
This raises the following\\[0.5em]
\textbf{Question 3.} Is there a smooth immortal solution to Laplacian flow that flows out of the~\Sp{2}-invariant shrinker cone 
that is not an expander?
\\[0.2em]
One should also consider the possibility of flowing 
closed singular~\gtwo-structures when there is no instantaneously smooth forward evolution, \eg
is it possible to develop a theory of Laplacian flow maintaining isolated conical singularities?

The problem of studying singular solutions to Lagrangian MCF with isolated conical singularities modelled on 
(stable) special Lagrangian cones has been considered. The key difference here is that our shrinker cone, unlike a special Lagrangian cone, 
is not itself static under the flow;  so the problem has a potentially very different character, because the non-static nature of the cone itself 
should significantly influence the forward dynamics. It would of course also be possible to study Laplacian flow 
with isolated singularities modelled on (stable) torsion-free~\gtwo-cones; this would be much more analogous 
to the Lagrangian MCF flow problem just mentioned. While this seems an interesting question in itself it would not shed any light on the pressing issue of what to 
do in Laplacian flow if one encounters a finite-time singularity modelled on the explicit~\Sp{2}-invariant AC shrinker.

\medskip
\noindent
\emph{Acknowledgements.}
MH and JN would like to thank the Simons Foundation for its support of their
research under the Simons Collaboration on Special Holonomy in Geometry,
Analysis and Physics (grant \#488620 and \#488631).
This material is based in part upon work supported by the National Science Foundation
under Grant No. DMS-1928930, while MH and JN were in residence at the
Simons Laufer Mathematical Sciences Research Institute in Berkeley, California, during the
fall semester of 2024. RJ has been supported by a University
Research Studentship at the University of Bath and by Undergraduate Research
Bursary 19-20-82 from the London Mathematical Society.

\section{The \texorpdfgtwo{}-soliton equations}
\label{s:review}

This section is a review of the basics about Laplacian solitons 
and of the requisite facts about~\Sp{2}-invariant Laplacian solitons developed in~\cite{Haskins:Nordstrom:g2soliton1}. 
In particular, we explain how~\Sp{2}-invariant solitons on the product
of $\CP^3$ with an interval $I$ can be described in terms of a first-order ODE
system.

\subsection{Singularity formation and resolution in Laplacian flow and Laplacian solitons}

A \gtstr{} on a smooth 7-manifold $M$ can be defined in terms of a smooth
3-form $\varphi \in \Omega^3(M)$ such that
$\varphi \wedge (v \lrcorner \varphi)^2 \not= 0$ for every nonzero tangent vector $v$.
Such a ``positive'' 3-form $\varphi$ determines an orientation and a Riemannian metric.
The \gtstr{} is torsion-free if $d\varphi = d^* \varphi = 0$, in which case the
holonomy group of the metric is contained in $G_2$, see \eg Joyce~\cite[\S10]{Joyce:holonomy:book}.

A $1$-parameter family of closed~\gtstr s $\varphi(t)$ evolves according to the Laplacian flow if it satisfies
\begin{equation*}
\label{eq:LF}
\frac{\partial \varphi}{\partial t} = \Delta_{\varphi} \varphi
\end{equation*}
where $\Delta_\varphi$ is the Hodge Laplacian on $3$-forms determined by the evolving metric $g_{\varphi(t)}$. 
Clearly any torsion-free~\gtstr~gives rise to a fixed point of Laplacian flow and on a compact manifold integration by parts shows 
that these are the only fixed points. 

The next simplest class of solutions to 
Laplacian flow are its self-similar solutions, \ie where $\varphi$ evolves by a combination of scaling and diffeomorphisms. 
These are in one-to-one correspondence with Laplacian solitons. 
We call a triple $(\varphi,X,\lambda)$ consisting of a \mbox{\gtstr~}$\varphi$, 
a vector field $X$ and a constant $\lambda \in \R$  a \emph{Laplacian soliton} if the triple satisfies the 
following nonlinear system of  PDEs
\begin{equation}
\label{eq:LSE}
\left\{
  \begin{aligned}
    d \varphi &= 0,\\
\Delta_\varphi \varphi&= \lambda \varphi + \mathcal{L}_X \varphi.
      \end{aligned}\right.
\end{equation}
The constant $\lambda$ will be referred to as the dilation constant of the soliton. 
If we wish to emphasise the role of the dilation constant we will also sometimes write a $\lambda$-soliton. 

The Laplacian soliton equations form a diffeomorphism-invariant nonlinear system of overdetermined PDEs
and so a priori even local existence of solitons is unclear. 
The local generality of their solutions was studied by Bryant~\cite{Bryant:LS:local:generality}
using the methods of overdetermined PDEs.  Bryant's work 
gives a complete answer to the local theory when the vector field $X$ is generic. (However, in the special case of gradient Laplacian solitons, 
\ie when $X=\nabla f$ for some function $f$, even the local generality of solutions remains unclear; 
this is due to the existence of additional identities satisfied by gradient Laplacian solitons.)
As already mentioned in the introduction, for compact Laplacian solitons we know the following:
there are no shrinking solitons and all steady solitons are static. Moreover, 
it is unknown whether or not compact Laplacian expanders exist.
Therefore we focus on \emph{complete noncompact} Laplacian solitons. Such complete noncompact solitons 
can arise naturally when studying singularity formation and singularity resolution in Laplacian flow.

As in other geometric flows, the study of ancient solutions to Laplacian flow, \ie   solutions to the flow defined 
on a semi-infinite time interval extending back to $t=-\infty$, is fundamental to understanding finite-time singularities. 
They arise when taking parabolic blowups 
in space-time at finite-time singularities of the flow. 
Our current understanding of ancient solutions in Laplacian flow is very rudimentary, somewhat like the situation 
for ancient solutions to Ricci flow in high dimensions when one does not impose any kind of positivity of curvature. 
Shrinking and steady solitons provide the simplest classes of ancient solutions and
typically one has to understand such solitons before one can hope to  understand more general ancient solutions. 

Expanders on the other hand, are the simplest class of immortal solutions to the flow, \ie solutions defined on a semi-infinite time interval
extending to $t=+\infty$. One way in which expanders (or more general immortal solutions) can arise is as parabolic blowups for a solution to the flow that is smooth for $t>0$, 
but that emerges out of some singular initial data at $t=0$.  In the simplest instance, the singular initial data might be a cone with an isolated singularity. 
In this way, understanding expanders (especially AC ones) can be important in understanding how finite-time singularities that might appear in a geometric flow, 
could (or could not) be smoothed out again by the flow.

\subsection{\texorpdfSp{2}-invariant Laplacian solitons}

In this paper we study complete~\Sp{2}-invariant expanders and the end behaviours of~\Sp{2}-invariant expanders and shrinkers, 
building on the previous work of two of the authors on cohomogeneity-one Laplacian solitons~\cite{Haskins:Nordstrom:g2soliton1}.
We now recall the basic set-up from~\cite{Haskins:Nordstrom:g2soliton1} and restate the results we will need from that paper.

The $\Sp{2}$ action we consider has principal isotropy group $K=\unitary{1} \times \Sp{1}$ so that any principal orbit is diffeomorphic to $\CP^3$; any complete closed~\Sp{2}-invariant~\gtstr~must be noncompact, with
a unique singular orbit whose isotropy subgroup must be $\Sp{1} \times \Sp{1}$, \ie the singular orbit is a round $\Sph^4$.  Any complete closed $\Sp{2}$-invariant~\gtstr~must therefore be defined on $\Lambda^2_-\Sph^4$. The Bryant--Salamon AC~\gthol metric on $\Lambda^2_-\Sph^4$~\cite{Bryant:Salamon}  gives the
best-known instance of such a complete closed $\Sp{2}$-invariant~\gtstr.

For an interval $I$, consider closed \Sp{2}-invariant \gtstr s $\varphi$ on
$I \times \CP^3$, such that the coordinate vector field $\pd{}{t}$ has unit
length with respect to the induced metric.
Up to the action of a natural discrete symmetry group $\Z_2 \times \Z_2$, any
such $\varphi$ can be written as
\[
\varphi = %
(x^2 \omega_1 + y^2 \omega_2)\wedge dt + xy^2\, \alpha,
\]
where $(x,y): I \to \R^2$ is a positive pair satisfying the first-order ODE 
\[
(xy^2)' =  \tfrac{1}{2}x^2 + y^2,
\]
 $\omega_1, \omega_2$ is a certain basis for the invariant $2$-forms
on $\CP^3$, and $\alpha = 2\, d \omega_1 = d \omega_2$. 
 Associated with the pair $\omega_1, \omega_2$ is a pair of invariant bilinear forms $g_1$ and $g_2$; $g_1$ may be identified with an~\Sp{1}-invariant metric on $\Sph^2$, which is isometric to the round metric with constant curvature $1$, while 
$g_2$ can be identified with an~\Sp{2}-invariant metric 
on $\Sph^4$ which is the round metric with constant sectional curvature $\frac{1}{2}$. 

\begin{remark}
\label{rmk:vol:growth}
It is convenient to define the positive function $g$ by $g^3=xy^2$. 
The ODE expressing the closure of $\varphi$ is  then equivalent to 
\begin{equation}
\label{eq:g:dot}
g' = \frac{x^2+2y^2}{6g^2} = \frac{1}{6} \left(\frac{x}{y}\right)^{4/3} + \frac{1}{3} \left( \frac{y}{x} \right)^{2/3}.
\end{equation}
Hence given any bound on $\log{\frac{y}{x}}$ (\ie given lower and upper bounds on $\frac{y}{x}$) we will have a corresponding bound on $g'$. 
Therefore on any finite interval we will also have a bound on $g$ itself. 
Moreover, the power means inequality implies that 
\[
g' \ge \tfrac{1}{2}
\]
 with equality if and only if $x=y$.
In particular, any complete closed~\Sp{2}-invariant~\gtstr~has at least Euclidean volume growth, \ie large geodesic balls of radius $r$ grow at least like $r^7$.

The previous inequality also implies that no closed~\Sp{2}-invariant~\gtstr~can have an infinite backward existence time, because backwards $g$ decreases 
at least linearly. Thus a complete closed~\Sp{2}-invariant~\gtstr~ must be noncompact, have a unique singular orbit, a single end and be defined on $\Lambda^2_-\Sph^4$.
\end{remark}
The torsion $\tau$ of a closed~\gtstr~is the unique $2$-form of type $14$ determined by $d(*\varphi) = \tau \wedge \varphi$. 
In our setting $\tau$ is also~\Sp{2}-invariant and hence can be written as
$\tau = \tau_1 \omega_1 + \tau_2 \omega_2$ for functions $\tau_1$ and $\tau_2$.  These functions are given  in terms of $(x,y)$ and their first derivatives by
\begin{equation}
\label{eq:torsion:xy}
\tau_1 = (x^2)' - 2x + \frac{x^3}{y^2}, \quad \tau_2 = (y^2)' - x.
\end{equation}
The type $14$ condition on the torsion $2$-form $\tau$ in our setting is equivalent to 
\begin{equation}
\label{eq:tau1:tau2:sp2}
\frac{1}{x^2}\tau_1  +\frac{2}{y^2} \tau_2=0.
\end{equation}
Obviously this implies that $\tau_1$ and $\tau_2$ always have opposite
signs. Using~\eqref{eq:tau1:tau2:sp2} we can also replace $\tau_1$
in~\eqref{eq:torsion:xy} with $-\tfrac{2x^2}{y^2}\tau_2$.

The quantity
\begin{equation}
\label{eq:S:def}
S = y^2- x^2 - \frac{3}{2} x \tau_2,
\end{equation}
turns out to play an important role in our later analysis. 
The torsion equations~\eqref{eq:torsion:xy} are also equivalent to 
\begin{equation}
\label{eq:xydot:S}
x' = \frac{(x^2+2y^2)+4S}{6y^2}, \qquad y' = \frac{(x^2+2y^2)-2S}{6xy}.
\end{equation}
One reason for the importance of $S$ is that it appears in the derivative of the ratio $\frac{y}{x}$. 

\begin{lemma}
\label{lem:log_xy}
For any closed \Sp{2}-invariant \gtstr{} $(x,y)$ we have
\begin{equation}
\label{eq:d:dt:log:yx}
 \frac{d}{dt} \log \frac{y}{x} = -\frac{S}{g^3}.
 \end{equation}
\end{lemma}

\begin{proof}
This is immediate from \eqref{eq:xydot:S} which relies only on $d\varphi = 0$.
\end{proof}

Using the structure equations $\alpha = 2 d \omega_1 = d\omega_2$ one readily
verifies that the Laplacian soliton equations \eqref{eq:LSE}
for a closed \Sp{2}-invariant \gtstr{} $\varphi$ and vector field
$X = u \, \partial_t$ reduce to the following mixed-order system of ODEs
\begin{subequations}
\label{eq:ODEs:Sp2}
\begin{align}
\label{eq:ODE:closure}
(xy^2)' &= \tfrac{1}{2}x^2 + y^2,\\
\label{eq:ODE:tau1}
(\tau_1 - ux^2)' &= \lambda x^2,\\
\label{eq:ODE:tau2}
(\tau_2 - uy^2)' & = \lambda y^2,\\
\label{eq:ODE:conserve}
\tau_1+2 \tau_2 &= u(x^2+2y^2) + 2\lambda xy^2,
\end{align}
\end{subequations}
where the components $\tau_1$ and $\tau_2$ of the torsion $2$-form $\tau = \tau_1 \omega_1 + \tau_2 \omega_2$ of the closed~\gtstr~$\varphi$ are 
given by \eqref{eq:torsion:xy}.

\begin{remark}
\label{rmk:scaling}
The scaling symmetry from Remark \ref{rmk:general_scaling} can be expressed
as the transformation
\begin{equation}
\label{eq:scaling}
(t,x,y, u, \lambda, \tau_2, S, \varphi, g_\varphi) \mapsto (\mu t, \mu x, \mu y,  \mu^{-1}u, \mu^{-2} \lambda,  \mu\tau_2, \mu^2S, \mu^3 \varphi, \mu^2 g_\varphi),
\end{equation}
which is indeed a symmetry of the ODE system \eqref{eq:ODEs:Sp2}.
\end{remark}

\subsection{First-order reformulation of the soliton ODE}

The mixed-order ODE system~\eqref{eq:ODEs:Sp2} for the quintuple $(x,y,\tau_1,\tau_2,u)$ turns out to be equivalent to a first-order system for the triple $(x,y,\tau_2)$.

\begin{prop}[\!\!{\cite[Proposition 5.24]{Haskins:Nordstrom:g2soliton1}}]
\label{prop:1st:order:tau2}
If $(x, y,  \tau_1, \tau_2, u)$ is any solution to the \Sp{2}-invariant soliton system 
\eqref{eq:ODEs:Sp2} then $(x,y,\tau_2)$ satisfies the rational
first-order ODE system
\begin{equation}
\label{eq:ODE:tau2'}
(x^2)' = 2x - \frac{x^2}{y^2} (x+ 2\tau_2), \qquad  (y^2)' = x + \tau_2, \qquad \tau_2' = \frac{4R_1 S}{3x(x^2+2y^2)},
\end{equation}
where 
\[R_1:=\lambda xy^2-3\tau_2,  \quad \ S:=y^2-x^2-\tfrac{3}{2}x\tau_2.
\]
Conversely, if $(x,y, \tau_2)$ is a solution to the first-order system \eqref{eq:ODE:tau2'} with $x,y>0$ 
and if we define $\tau_1$ and $u$ in terms of $(x,y, \tau_2)$ by
\begin{equation}
\label{eq:tau1:u:tau2}
\tau_1 := -\frac{2x^2}{y^2}\tau_2, \qquad u := \frac{2(y^2-x^2)\tau_2 -2 \lambda xy^4}{y^2(x^2+2y^2)},
\end{equation}
then $(x,y,\tau_1,\tau_2,u)$ solves the \Sp{2}-invariant soliton system \eqref{eq:ODEs:Sp2}.
\end{prop}
Since $S$ is determined algebraically from $x$, $y$ and $\tau_2$ one can also derive an equivalent first-order ODE system
for the variables $(x,y,S)$ rather than $(x,y,\tau_2)$. Equations~\eqref{eq:xydot:S} above already
give the ODEs for $x'$ and $y'$ in terms of $x$, $y$ and $S$. 
The equation satisfied by $S'$ in terms of $x$, $y$ and $S$ plays a key role later in the paper, 
so we record that equation here.
\begin{lemma}
For any~\Sp{2}-invariant soliton $(x,y,S)$ the quantity
$S$ satisfies the equation
\begin{equation}
\label{eq:Sdot}
S' =\left( \frac{2xy^2}{x^2+2y^2}\right) (\slead - \sother S),
\end{equation}
where $\slead=\slead(\frac{y}{x})$, and $\slead$ is the function
\begin{equation}
\label{eq:alpha}
\slead(l) := \frac{1}{12}(l^2-1)(2 + l^{-2})^2,
\end{equation}
and 
\[
\sother := \lambda + \frac{ x^4+12x^2y^2-4y^4}{4x^2y^4} + \frac{(4y^2-x^2)S}{3x^2y^4}.
\]
\end{lemma}
\begin{proof}
From~\eqref{eq:xydot:S} and~\eqref{eq:ODE:tau2'}, 
$S'$ can be written in terms of $x$, $y$ and $S$ itself
as
\begin{align*} S' &= 2yy' - \left(2x + \tfrac{3}{2}\tau_2\right)x' - \tfrac{3}{2}x\tau_2' \\
&= \left(\frac{x^2 + 2y^2 - 2S}{3x}\right) + \left(\frac{S-x^2-y^2}{x}\right) \left( \frac{x^2+2y^2 + 4S}{6y^2}\right)
- \frac{2R_1S}{x^2 + 2y^2} \\
& = \frac{2xy^2}{x^2+2y^2} \left( -\frac{(x^2+2y^2)^2 (x^2-y^2)}{12x^2y^4} - \lambda S - \frac{ (x^4+12x^2y^2-4y^4 )S}{4x^2y^4} - \frac{(4y^2-x^2)S^2}{3x^2y^4} \right).
\end{align*}
The first coefficient in the parentheses above is equal to $\slead(\tfrac{y}{x})$ and 
the remaining ones give $-\sother S$. 
\end{proof}

\begin{remark}
\label{rmk:gaussian}
The simplest solutions to the soliton ODEs are the so-called \emph{Gaussian solitons} over the (unique) torsion-free~\Sp{2}-invariant closed~\gtwo-cone. 
These exist for any value of the dilation constant $\lambda$ and are characterised by $x = y = \frac{1}{2}t$, $\tau_1=\tau_2 \equiv 0$ 
and $u = -\frac{2\lambda}{3} t$. 

Clearly, these solutions satisfy $S \equiv 0$. Conversely, if $S \equiv 0$, then obviously $S' \equiv 0$
and hence $\slead \equiv 0$ by~\eqref{eq:Sdot}, which in turn implies that $x = y = \frac{1}{2}t$ and $\tau_1=\tau_2 \equiv 0$. 
In fact, if at some point $S$ and $S'$ both vanish then $y=x$ and $\tau_2=0$ at that point. 
Then by uniqueness of solutions to the first-order ODE system~\eqref{eq:ODE:tau2'} the soliton must be a Gaussian soliton. 
\end{remark}

\subsection{Adjusted torsion}
\label{ss:tilde_tau}

For later purposes it is convenient to define the quantities
\begin{equation}
\label{eq:tilde:tau}
 \tilde{\tau}_1:=\tau_1-ux^2 \quad \text{and} \quad  \tilde{\tau}_2:=\tau_2-uy^2,
\end{equation}
where $u$ is the coefficient of the soliton vector field $X=u \, \partial_t$.
We call these quantities the adjusted torsion of the soliton. 
Then~\eqref{eq:ODE:tau1} and~\eqref{eq:ODE:tau2} simply read
\begin{equation}
\label{eq:tt'}
\tilde \tau_1' = \lambda x^2, \quad \tilde \tau_2'= \lambda y^2,
\end{equation}
which leads to monotonicity of the adjusted torsion, and sometimes their signs
(and that of $u$) being preserved.
Note also that \eqref{eq:ODE:conserve} is equivalent to
\begin{equation}
\label{eq:ODE:conserve2}
\tilde \tau_1 + 2\tilde \tau_2 = 2\lambda xy^2.
\end{equation}
In particular, for any expander at least one of the $\tilde \tau_i$ must be positive at any instant
(and then by the lemma below it remains positive thereafter).

\begin{lemma}
\label{lem:tau:tilde:positive} For any~\Sp{2}-invariant expander 
the functions $\tilde \tau_1$ and $\tilde \tau_2$ are both increasing functions of $t$.
In particular,  the conditions $\tilde \tau_1 > 0$ and $\tilde \tau_2 > 0$ are (individually) preserved.
\end{lemma}

\begin{proof}
This follows immediately from $\tilde\tau_1' = \lambda x^2$ and $\tilde \tau_2' = \lambda y^2$.
\end{proof}

\begin{lemma}
\label{lemma:u:sign} For any~\Sp{2}-invariant soliton, 
if $\tilde \tau_1$ and $\tilde \tau_2$ have the same sign then $u$ has the opposite sign.
\end{lemma}

\begin{proof}
Applying the type $14$ condition~\eqref{eq:tau1:tau2:sp2} on $\tau$ to the definition of the $\tilde \tau_i$ implies that
\begin{equation}
\label{eq:u:tilde:tau}
u = - \frac{1}{3} \left( \frac{\tilde \tau_1}{x^2} + \frac{2 \tilde \tau_2}{y^2} \right).
\vspace{-1em}
\end{equation}
\end{proof}

\begin{remark}
\label{rmk:R1:R2}
For later purposes it is useful to have alternative expressions for the
quantities $\tilde{\tau}_1$ and~$\tilde{\tau}_2$.
From~\eqref{eq:u:tilde:tau} we have
\[ 
-3uy^2 = 2 \tilde \tau_2 +\tilde \tau_1 \frac{y^2}{x^2}.
\]
Therefore
\begin{equation}
\label{eq:tau2:tildetau1}
3 \tau_2 + \frac{y^2}{x^2} \tilde \tau_1 =3 \tau_2   - 3uy^2 -2 \tilde \tau_2 = 3 \tilde \tau_2 - 2 \tilde \tau_2 = \tilde \tau_2.
\end{equation}
From~\eqref{eq:tau1:tau2:sp2} and~\eqref{eq:ODE:conserve}  we have
\begin{equation}
\label{eq:u:tau1:tau2}
-ux^2 = \frac{2\lambda x^3y^2 - (x^2-y^2) \tau_1}{x^2+2y^2}, \qquad -uy^2 = 2\left(\frac{\lambda xy^4 - (y^2-x^2) \tau_2}{x^2+2y^2}\right)
\end{equation}
and therefore
\begin{equation}
\label{eq:tilde:tau1:tau2}
\tilde{\tau}_1 = \tau_1 - ux^2 = \frac{y^2(3\tau_1+2\lambda x^3)}{x^2+2y^2} = \frac{2x^2 R_1}{x^2+2y^2}, \quad \tilde{\tau}_2 = \tau_2 - uy^2 = \frac{3x^2\tau_2+2\lambda xy^4}{x^2+2y^2} = \frac{x^2R_2}{x^2+2y^2},
\end{equation}
where 
\[
R_1:= \lambda xy^2-3 \tau_2, \qquad R_2:= \frac{2\lambda y^4}{x}+3 \tau_2.\]
Note that $R_1$ already appeared on the right-hand side of the ODE for $\tau_2'$ in the system~\eqref{eq:ODE:tau2'}.
\end{remark}

One of the uses of the adjusted torsion is that the torsion itself can be
recovered from it together with the warping. Write $h=\frac{y}{x}$. 
The relation
\begin{equation}
\label{eq:tau2:tilde:taui}
3 \tau_2 = - h^2 \tilde \tau_1 + \tilde \tau_2
\end{equation}
holds by the definition of the $\tilde \tau_i$ and the type $14$ condition. 
Differentiating~\eqref{eq:tau2:tilde:taui} and using~\eqref{eq:ODEs:Sp2}
implies 
\begin{equation}
\label{eq:tau2':alt}
 \tau_2'
= -\frac{1}{3} \tilde \tau_1  (h^2)'.
\end{equation}
Now write $k=\frac{x}{y}$. The equation satisfied by $\tau_1'$ can be written
as
\[
\tau_1' = - \frac{2}{3} \tilde \tau_2  (k^2)'.
\]
Indeed, since $\tau_1 = -2 k^2 \tau_2$, equation~\eqref{eq:tau2:tilde:taui} is
equivalent to $\tau_1 = - \frac{2}{3} (k^2 \tilde \tau_2 - \tilde \tau_1)$
and differentiating this yields the claimed expression for $\tau_1'$.

\section{Complete expanders}
\label{sec:complete}

It follows immediately from the %
first-order reformulation of the soliton ODE system in
Proposition \ref{prop:1st:order:tau2} for that any $\lambda \neq 0$ there is a
$2$-parameter family of distinct local \Sp{2}-invariant 
$\lambda$-solitons and that (on the open dense set of principal orbits) these solutions are real-analytic functions of $t$.

The generic member of this $2$-parameter family of local \Sp{2}-invariant
non-steady Laplacian solitons does not however extend to a complete Laplacian
soliton on $\Lambda^2_-\Sph^4$. A first necessary condition for such a complete
extension is that the solution must extend smoothly over
the unique singular orbit $\Sph^4$, so we begin by reviewing the results
from \cite[\S6]{Haskins:Nordstrom:g2soliton1} on such solutions.
Understanding which solutions have a complete end is in general more difficult. 
In this section we establish some
rudimentary control that will suffice to prove that every smoothly-closing
\Sp{2}-invariant expander is, in fact, complete with an end that is AC in at least a weak
sense.

\subsection{Solitons extending smoothly over the singular orbit}
\label{subsec:smoothly_closing}

Recall that the functions $x$ and $y$ defining the \Sp{2}-invariant
\gtstr{} in \eqref{eq:phi} control the scale of the fibre and base,
respectively, of the twistor fibration $\CP^3 \to \Sph^4$.
Extending smoothly over $\Sph^4$ requires that at a boundary point of the
interval on which $x$ and $y$ are defined, $x \to 0$ at a certain rate,
while $y$ has a positive limit~$b$.
(The volume of the special orbit $\Sph^4$ is then proportional to $b^4$.)
In~\cite[Theorem B]{Haskins:Nordstrom:g2soliton1},  two of the authors proved that among the 
$2$-parameter family of local \Sp{2}-invariant non-steady Laplacian solitons there is a $1$-parameter family of smoothly-closing solitons.

\begin{theorem}
\label{thm:Sp2:smooth:closure}
Fix any $\lambda \in \R$ and $b>0$. Then there exists a unique local (that is,
defined on some neighbourhood of the singular orbit $\Sph^4$) \Sp{2}-invariant
$\lambda$-soliton $\spfam{\lambda}{b}$ which closes smoothly on 
the singular orbit $\Sph^4 \subset \Lambda^2_-\Sph^4$ and which satisfies
\[
x= t + t^3 \hat{x}, \quad y = b + t^2 \hat{y}, \quad \tau_1 = t^3 \hat{\tau}_1, \quad \tau_2 = t \hat{\tau}_2,
\]
where the functions $\hat{x}, \hat{y}$ and $\hat{\tau}_i$ satisfy the initial conditions
\[
\hat{x}(0) = -\frac{1}{6b^2} - \frac{2\lambda}{27}, \quad
\hat{y}(0) = \frac{\left( 2\lambda b^2+9 \right)} {36b}, \quad
\hat{\tau}_1(0) = -\frac{4\lambda}{9}, \quad 
\hat{\tau}_2(0) = \frac{2}{9} \lambda b^2.
\]
The solution $\spfam{\lambda}{b}$ is real analytic on $[0,\epsilon) \times \CP^3$ for some $\epsilon >0$ (depending on $b$ and $\lambda$).
Moreover, up to a natural $\Z_2 \times \Z_2$-action, any \Sp{2}-invariant \gtwo-soliton which closes smoothly on the singular orbit $\Sph^4 \subset \Lambda^2_-\Sph^4$ belongs to this $1$-parameter family of solutions.
\end{theorem}

\begin{remark}
\label{rmk:scaling_sc}
Rescaling $\spfam{\lambda}{b}$ by a length factor $\mu > 0$, as in
Remark \ref{rmk:scaling}, yields a smoothly-closing $\mu^{-2}\lambda$-soliton
with initial condition $y(0) = \mu b$.
By the uniqueness, that must be precisely $\spfam{\mu^{-2}\lambda}{\mu b}$.
Hence for $\lambda \neq 0$ we can think of the solutions from Theorem \ref{thm:Sp2:smooth:closure}
as a 1-parameter family up to scale, with scale-invariant parameter
$\lambda b^2$. Positive values of this scale-invariant parameter correspond to expanders and negative values to shrinkers. 

For $\lambda=0$ there is instead a unique smoothly-closing steady soliton up to scale. The Bryant--Salamon AC torsion-free~\gtstr s on $\Lambda^2_-\Sph^4$ 
with vanishing soliton vector field $X$ clearly give smoothly-closing~\Sp{2}-invariant steady solitons. Hence by uniqueness the smoothly-closing steady solitons
from Theorem~\ref{thm:Sp2:smooth:closure} must in fact be the Bryant-Salamon (static) solutions (which recall are unique up to scale). In particular, its asymptotic cone is the cone over the unique~\Sp{2}-invariant nearly K\"ahler structure on~$\CP^3$, which has $x=y=\frac{1}{2}t$. 
\end{remark}

Since $x$, $y$, $u$, $\tau_1$ and $\tau_2$ are real analytic on $[0,\epsilon)$ for some $\epsilon>0$, one can ask what are
their power series expansions at $t=0$. 
The proof of Theorem~\ref{thm:Sp2:smooth:closure}  also shows that the functions $x$, $\tau_1$, $\tau_2$ and $u$ are all odd functions of $t$, while $y$ is an even function of $t$.
A term-by-term calculation 
gives the following expressions for the leading coefficients
\begin{equation}
\label{eq:sp2:expansion}
\begin{aligned}
x &= t - \frac{t^3}{54b^2} \left(4\lambda b^2+9\right) + \cdots &
y &= b + \frac{t^2}{36b} \left( 2\lambda b^2+9 \right)  + \cdots \\
\tau_1 & = -\frac{4\lambda}{9} t^3 + \frac{2\lambda t^5}{405b^2} \left(26b^2\lambda + 81  \right) + \cdots &
\tau_2 &= \frac{2\lambda b^2}{9} t  - \frac{2\lambda t^3}{1215} \left( 4b^2\lambda+9  \right) + \cdots\\
u & = - \frac{7\lambda }{9}t + \frac{4\lambda t^3}{1215b^2} \left(13b^2 \lambda + 63 \right) + \cdots 
\end{aligned}
\end{equation}
Computer-assisted symbolic calculation allows us to determine many further
higher-order coefficients.  This is helpful when doing numerical simulations of
the ODE system, because rather than trying to numerically solve this ODE system
close to the singular point that occurs at $t=0$, we can evaluate these
polynomial approximations to $x$, $y$ and $\tau_2$ at some small positive $t$
and use those as our (approximate) nonsingular initial data for the ODE system. 

\begin{remark*}\hfill
\begin{enumerate}[left=0em]
\item Note that $\lambda$ appears as a factor in all the coefficients of $\tau_1$, $\tau_2$ and $u$ in~\eqref{eq:sp2:expansion}.
This is consistent with any smoothly-closing~\Sp{2}-invariant steady soliton being static with vector field $X=0$.  

\item
The expansions~\eqref{eq:sp2:expansion} are compatible with $x$, $b^{-1}y$, $\tau_1$
and $b^{-2}\tau_2$ having well-defined limits as $b \to \infty$. 
We refer the reader to Remark~\ref{rmk:ell_sharp} for a discussion of a closely-related issue. 
\end{enumerate}
\end{remark*}

\begin{remark}
\label{rmk:init_signs}
We can also use the expansions~\eqref{eq:sp2:expansion} to determine the initial signs
of various quantities. In particular, for any smoothly-closing \Sp{2}-invariant expander $\spfam{\lambda}{b}$ 
we find that the four quantities
\[
y-x, \quad \tilde\tau_2 - \tilde\tau_1, \quad S \quad \text{and} \quad \tau_2,
\]
are all positive 
for $t>0$ sufficiently small. 
Later in the section we prove that all four conditions above persist under forward-evolution of the expander ODEs. 
\end{remark}

Given Theorem~\ref{thm:Sp2:smooth:closure}, to find all complete
\Sp{2}-invariant solitons it remains to understand which of the
smoothly-closing ones with $\lambda \not= 0$ are also forward complete.
For any complete~\Sp{2}-invariant soliton one would then also like to understand its asymptotic geometry.

We consider the $\lambda > 0$ case in detail beginning in the next section, where
we use the inequalities from Remark
\ref{rmk:init_signs} to prove that all smoothly-closing expanders
$\spfam{\lambda}{b}$ are forward complete.

For $\lambda<0$, Example \ref{ex:explicit_shrinker} is an explicit complete
AC shrinker, asymptotic to the closed (but non-torsion-free) \Sp{2}-invariant
\gtwo--cone defined by $x=t$, $y= \frac{1}{2}t$.
As we will discuss in Section~\ref{sec:ends}, AC end behaviour is non-generic
for shrinkers. In fact, we conjecture that this is the
unique complete \Sp{2}-invariant AC shrinker up to scale.

\begin{remark}
\label{rmk:explicit_shrinker} 
Example \ref{ex:explicit_shrinker} has $\lambda = -\frac{9}{4b^2}$, $x = t$
and $y = \sqrt{b^2 + \frac{t^2}{4}}$, so up to scale we can identify it with the
smoothly-closing solution $\spfam{-1}{\frac32}$.
(It is already clear from the series expansions
\eqref{eq:sp2:expansion} that the series for $x$ and $\tau_2$ simplify
significantly when $4\lambda b^2+9=0$.)
\end{remark}

\subsection{Forward completeness and weak asymptotic conicality}
\label{s:Sp2:expanders:completeness}

In combination, the first two parts of Theorem \ref{mthm:warping} mean that
an \Sp{2}-invariant soliton whose warping $\frac{y}{x}$ converges to
a non-zero limit forward in time must be weakly AC.

The basic incompleteness criterion in Theorem \ref{mthm:warping}(i)
is, in fact, a special case of a
more general finite extinction time criterion for~\sunitary{3}-invariant
solitons (which holds regardless of the value of the dilation constant).
More specifically, it follows by setting $f_2=f_3$
in~\cite[Proposition~5.16]{Haskins:Nordstrom:g2soliton1}.

\begin{prop}
\label{prop:lifetime}
If an \Sp{2}-invariant soliton has finite lifetime, then $\log \frac{y}{x}$
is unbounded.
\end{prop}

In other words, one of $\frac{y}{x}$ and $\frac{x}{y}$ must be unbounded above.
We can in fact replace the condition on $\frac{y}{x}$ by $\frac{1}{x}$. 

\begin{corollary}
\label{cor:lifetime}
If an \Sp{2}-invariant soliton has finite lifetime, then $\frac{x}{y}$ or
$\frac{1}{x}$ is unbounded above. 

In particular, $x$ cannot be bounded both above and away from 0.
\end{corollary}

\begin{proof}
If $\frac{1}{x}$ and $\frac{x}{y}$ are both bounded above, then so is
\[ \frac{d}{dt} \log (xy^2) = \frac{1}{2x}\left(\frac{x^2}{y^2} + 2\right). \]
Thus $xy^2$ is bounded above in finite time, and hence so is
$\frac{y}{x} = (xy^2)\frac{1}{x^3}\frac{x}{y}$.

If $x$ is bounded above, then certainly $\frac{x}{y}$ is bounded above because
$\frac{x^2}{y^2} = \frac{x^3}{g^3}$ and 
$g^3=xy^2$ is increasing.
\end{proof}

\begin{remark}
In \S \ref{ss:fc_non_ac} we discuss the existence of
forward-complete expanders and shrinkers for which $\frac{y}{x}$ and $\frac{x}{y}$, respectively, are unbounded above as $t \to \infty$. 
Their asymptotic behaviour %
differs dramatically from the AC behaviour exhibited
by solutions for which $\log{\frac{y}{x}}$ is forward-bounded. 
\end{remark}

An obvious corollary is that if we can establish that $\log \frac{y}{x}$ remains forward-bounded, then its lifetime is infinite and so the solution is forward complete. 
In this case we can ask what the asymptotic geometry of the solution is as $t \to \infty$. 

We now prove Theorem \ref{mthm:warping}(ii), namely,  that %
if we further assume that the bounded ratio
$\log \frac{y}{x}$ actually converges to a limit, then the \gtstr{} is
weakly AC in the sense of \eqref{eq:weak_ac_def}.

\begin{lemma}
\label{lem:yx:bounded}
For any closed \Sp{2}-invariant~\gtstr~$\varphi$ such that $\frac{y}{x}$ has a 
limit $\ell \in (0,\infty)$ as $t \to \infty$, then we have 
\[
\frac{x}{t} \to c_1\, \quad \frac{y}{t} \to c_2
\]
for the unique (positive) solution $(c_1, c_2)$ to the closed cone equation
$c_1^2 + 2c_2^2 = 6c_1c_2^2$ such that $\frac{c_2}{c_1} = \ell$. (Recall that $(c_1,c_2)$ was given explicitly in terms of $\ell$ in~\eqref{eq:explicit_cone_solution}) 
In other words, $\varphi$ is weakly asymptotic to the unique closed~\Sp{2}-invariant~\gtwo-cone determined by $\ell = \lim_{t \to \infty} \frac{y}{x}>0$.
\end{lemma}
\begin{proof}
Given any parametrised $\Sp{2}$-invariant cone, %
\ie 
\[
x=c_1 t,\  y=c_2 t \quad \text{for\ } c=(c_1,c_2) \text{\ with\ } c_1>0, \ c_2>0
\]
the condition that the $3$-form $\varphi$ determined by $x$ and $y$ be closed is that
\begin{equation}
\label{eq:closed:c1c2}
6 c_1 c_2^2 = c_1^2 + 2c_2^2.
\end{equation}
When restricted to the positive quadrant this equation defines a smooth noncompact connected curve.  Rearranging to solve for $c_2^2$ yields
\[c_2^2 = \frac{c_1^2}{2(3c_1-1)} = \frac{c_1}{6} + \frac{1}{18} + \frac{1}{18(3c_1-1)}.
\]
It is immediate from this that any closed cone $(c_1,c_2)$ satisfies $c_1 > \tfrac{1}{3}$ and that every homothety class of cones is represented uniquely as claimed. 

If $\frac{y}{x} \to \ell$ then from~\eqref{eq:g:dot} it follows that 
\[
g' \to \frac{1}{6\ell^{4/3}} \left( 1 + 2 \ell^2\right).
\]
Let $\ell$ be the unique positive solution of~\eqref{eq:closed:c1c2} with $c_2/c_1 = \ell$. 
Then 
\[
\frac{1}{6\ell^{4/3}} \left( 1 + 2 \ell^2\right) = \frac{1}{6} \left(1+2 \frac{c_2^2}{c_1^2} \right)
\left(\frac{c_1}{c_2}\right)^{4/3} = \frac{c_2^2}{c_1} \left(\frac{c_1}{c_2}\right)^{4/3} = (c_1c_2^2)^{1/3},
\]
where we used~\eqref{eq:closed:c1c2} to obtain the second equality. 
Hence $\frac{g}{t} \to (c_1 c_2^2)^{1/3}$.
Then 
\[
\frac{x^3}{t^3} = \left(\frac{x^2}{y^2}\right) \frac{g^3}{t^3} \to \frac{1}{\ell^2} c_1 c_2^2 = \frac{c_1^2}{c_2^2} c_1 c_2^2 = c_1^3
\]
 and hence $\frac{x}{t} \to c_1$.  Therefore also $\frac{y}{t} \to c_2$ as claimed. 
\end{proof}
Note that this result used only the closure of $\varphi$ and, in particular, we did not require $\varphi$ to satisfy the soliton ODEs.
Later,  we will prove that for a non-steady soliton, if we can prove that $\log{\frac{y}{x}}$ remains bounded 
then, in fact, $\frac{y}{x}$ automatically converges to a limit and hence by the previous result it is weakly AC.

\begin{remark*}
A direct computation using~\eqref{eq:torsion:xy} shows that the components of the torsion $2$-form $\tau = \tau_1 \omega_1 + \tau_2 \omega_2$ of the closed cone parametrised by $(c_1,c_2)$ are
\begin{equation}
\label{eq:torsion:Sp2:cone}
\tau_1 = 4c_1(2c_1-1)t, \quad \tau_2 = \frac{c_1(1-2c_1)}{(3c_1-1)}t.
\end{equation}
The cone $c_1=c_2=\tfrac{1}{2}$ is therefore the unique closed~\Sp{2}-invariant~\gtwo-cone which is torsion free.
This is just  the~\gtwo--cone over the standard $\Sp{2}$-invariant nearly K\"ahler structure on $\CP^3$. 
The set of closed cones for which $\ell = \frac{c_2}{c_1}>1$,  the set of closed cones with $c_1 \in (\tfrac{1}{3},\tfrac{1}{2})$ and the set of closed cones on which $\tau_2>0$ for all $t>0$ therefore
coincide. 

The line $c_1=c_2$ intersects the curve $\gamma := \{ (c_1, c_2) : c_1^2 + 2c_2^2 = 6c_1c_2^2\} \subset \R^2$
transversally at $\frac{1}{2}(1,1)$, the torsion-free cone.
Removing this point
separates $\gamma$ into two noncompact open connected curves: $\gamma^+$ on which $c_2-c_1>0$ and $\gamma^-$ on which $c_2-c_1<0$. 
To see the asymptotic behaviour of the curve $\gamma^+$ (far from the torsion-free cone)
suppose that $c_1 = \frac{1}{3} + \epsilon$, for $\epsilon>0$ sufficiently small.  Then we find that $c_2 \approx (54 \epsilon)^{-1/2}$. 
If we instead suppose that $c_1 \gg 0$ then
$c_2 \approx (\frac{1}{6}c_1)^{1/2}$, describing the 
asymptotic behaviour of $\gamma^-$.
\end{remark*}

\subsection{Complete \texorpdfSp{2}-invariant AC expanders}
\label{subsec:complete_expanders}

We will now prove that when $\lambda > 0$, then the results of the previous
subsection apply to any of the smoothly-closing expanders $\spfam{\lambda}{b}$
of Section~\ref{subsec:smoothly_closing}, and establish that they are all weakly AC.
This proves Theorem~\ref{mthm:expand:Sp2} with the exception of part (iii)
that will be established in Section~\ref{s:AC:geometry:Sp2:expanders}. 

\begin{prop}
\label{prop:elem_AC}
For an \Sp{2}-invariant expander, if $S(t_0)$ and $\tau_2(t_0)$ have the same sign at some instant $t_0$, then that remains true for all $t > t_0$, $\log \frac{y}{x}$ has the same sign too, and $|\log \frac{y}{x}|$ is decreasing.
In particular, %
the solution is forward complete and weakly AC
in the sense of \eqref{eq:weak_ac_def}. %
\end{prop}

\begin{proof}
Suppose $S(t_0)$ and $\tau_2(t_0)$ are both positive.
We first show that this condition is preserved forward in time.

Suppose that this fails first at $t_1 > t_0$. We cannot have $S=0$ and $\tau_2=0$ simultaneously at $t_1$
unless the solution is a Gaussian soliton, and Gaussian solitons do not satisfy our hypotheses.
So first suppose that 
$S(t_1)=0$,  $S'(t_1)\le 0$ and $\tau_2(t_1)>0$. 
It follows from~\eqref{eq:Sdot} that at any instant when $S$ vanishes its derivative satisfies
\begin{equation}
\label{S:dot:zero}
S' = \frac{\tau_2}{4y^2}\left(x^2+2y^2\right),
\end{equation}
which in our case is positive because $\tau_2(t_1)>0$.  This gives a contradiction. 
Now suppose instead that $S(t_1)>0$ and $\tau_2(t_1)=0$ with $\tau_2'(t_1)\le 0$. 
But at any point where $\tau_2=0$, \eqref{eq:ODE:tau2'} implies that
\[
\tau_2' = \frac{4 \lambda Sy^2}{3(x^2+2y^2)}.
\]
Hence $\tau_2'(t_1)$ must be positive, which again is a contradiction.

Thus $S$ and $\tau_2$ remain positive.
Now $\log{\frac{y}{x}}$ is positive because
$S>0$ and $\tau_2>0$ imply that $y^2 - x^2 > \frac{3}{2} x \tau_2 >0$.
Because $S$ is positive, $\frac{y}{x}$ (and hence $|\log\frac{y}{x}|$) is decreasing by Lemma \ref{lem:log_xy}
and since it bounded below by $1$, it converges. 

Thus the solution is forward complete by Proposition \ref{prop:lifetime}
and satisfies the hypotheses of Lemma~\ref{lem:yx:bounded}.

If we instead start from $S(t_0)<0$ and $\tau_2(t_0)<0$ then all inequalities in the proof are reversed.
\end{proof}

\begin{theorem}
\label{thm:sc:expanders:complete}
For any smoothly-closing expander $\spfam{\lambda}{b}$
\begin{enumerate}[left=0em]
\item 
The following quantities are positive for all~$t>0$:
\[
\tilde \tau_1, \ \tilde \tau_2,\  S,\  \tau_2,\ y^2-x^2,\  \tilde \tau_2 -\tilde \tau_1 \  \text{and}\  {-}u.
\] 
\item
Its forward lifetime is infinite, so the expander is metrically complete. 
The limit as $t \to \infty$ of the warping $\frac{y}{x}$ exists and is equal to some $\ell \in [1,\infty)$. 
Then we have $\frac{x}{t} \to c_1$, $\frac{y}{t} \to c_2$ for the
unique (positive) solution $(c_1, c_2)$ to the closed cone equation
$c_1^2 + 2c_2^2 = 6c_1c_2^2$ such that $\frac{c_2}{c_1} = \ell$
(given by \eqref{eq:explicit_cone_solution}). 
\end{enumerate}
In other words, $\spfam{\lambda}{b}$ %
is $C^0$-asymptotic to a unique closed cone with  $\ell=\frac{c_2}{c_1} \ge 1$. 
\end{theorem}

\begin{proof}
The initial conditions of Theorem~\ref{thm:Sp2:smooth:closure} imply that $\tilde \tau_1(0)=\tilde \tau_2(0)=\tau_2(0) = 0$, and
$S(0) = b^2>0$.
Lemma \ref{lem:tau:tilde:positive} implies that $\tilde \tau_1$ and $\tilde \tau_2$ are positive for all $t > 0$, and hence by~\eqref{eq:tilde:tau1:tau2}, so is
$R_1$. 

Also, $S(t)$ is positive for small $t$.
Hence by \eqref{eq:ODE:tau2'}, $\tau_2'(t)$ is positive at least for small $t$,
and hence so too is $\tau_2(t)$. 
Proposition \ref{prop:elem_AC} then implies positivity
of $S$ and $\tau_2$ and $y^2 - x^2$ for all $t>0$, and the claims in (ii).

That $\tilde \tau_2 > \tilde \tau_1$ follows from and $y > x$ and \eqref{eq:tt'}, and Lemma \ref{lemma:u:sign} implies $u < 0$, completing the proof of (i).\end{proof}

\begin{remark}
\label{rmk:cone_constraint}
The conclusion that $\ell \ge 1$ immediately implies that 
there is a nonempty open subset 
of the $1$-dimensional space of~\Sp{2}-invariant closed~\gtwo-cones that can never appear as the asymptotic cone of any\emph{ complete} 
AC~\Sp{2}-invariant expander. 
In particular, 
this already proves that no smoothly-closing~\Sp{2}-invariant expander is asymptotic to the asymptotic cone 
of the explicit AC shrinker described in Example \ref{ex:explicit_shrinker}, \ie Corollary \ref{thm:shrinker} does not rely
on the full strength of Theorem \ref{mthm:asymptotic:limit}.
\end{remark}

\section{Monotonicity and bounds on warping}
\label{sec:warping}

We keep pursuing the idea from the previous section that we can understand the
end behaviour of an \Sp{2}-invariant soliton by controlling its warping
$\frac{y}{x}$, but now in a more general context than just smoothly-closing 
expanders. We prove that the warping $\frac{y}{x}$ on any \Sp{2}-invariant
soliton is eventually monotone. It must thus eventually tend to $0$,
to a finite positive number or to $+\infty$, and we already noted in
Lemma \ref{lem:yx:bounded} that the second case corresponds to weakly AC ends.

In Section~\ref{ss:weak:ac:expand} 
we use this to establish necessary and sufficient conditions for
an~\Sp{2}-invariant expander to be weakly AC in terms of the monotone
quantity $\tilde \tau_1$ (the adjusted torsion from \S\ref{ss:tilde_tau}).
In particular, this implies that for expanders 
having a weakly AC end is a stable property. (The AC end condition is not stable for shrinkers though). 

\subsection{Monotone and eventually monotone quantities}
\label{ss:monotone}

We first introduce a quantity involving $x$, $y$ and $\tilde \tau_1$ that is strictly decreasing for all~\Sp{2}-invariant solitons, and serves as a key
technical tool. The most immediate application is that its sign controls
the nature of critical points of $x$, which helps to control the eventual
behaviour of $x$, but it will also play a crucial role in
\S\ref{ss:large:b:asymptotics} and in~\S\ref{ss:trichotomy}.

\begin{lemma}
\label{lem:monotone}
For any~\Sp{2}-invariant soliton define the quantity
\begin{equation}
M:= 3x + \tilde \tau_1 .
\end{equation}
Then the quantity $M/g^3$ satisfies the equation
\begin{equation}
\label{eq:monotonic}
\frac{d}{dt} M g^{-3} = -\frac{3x}{y^4}.
\end{equation}
Hence on any~\Sp{2}-invariant soliton $M g^{-3}$ is a strictly decreasing function of $t$.
\end{lemma}

\begin{proof}
Recall from~\eqref{eq:tilde:tau1:tau2} that
\[
\tilde \tau_1 = \frac{2x^2R_1}{2y^2 + x^2} = \frac{2x^2(\lambda xy^2 - 3 \tau_2)}{2y^2+x^2}.
\]
Rewriting the equation for $\tilde \tau_1$ in the form
\[ \frac{(2y^2 + x^2)\tilde \tau_1 + 6x^2\tau_2}{xy^2} =
2\lambda x^2, \]
and using the ODE for $x'$ from~\eqref{eq:ODE:tau2'} shows that $M=3x+ \tilde \tau_1$ satisfies the differential equation
\begin{equation}
\label{eq:M'}
M' = 
3x'+
\lambda x^2 =
\frac{6xy^2 - 6x^2\tau_2 - 3x^3 + (2y^2+x^2)\tilde\tau_1 + 6x^2\tau_2}
{2xy^2} =
\frac{(y^2 + \half x^2)}{xy^2} M- \frac{3x^2}{y^2}.
\end{equation}
The result now follows because $g^3=xy^2$ satisfies $({g^3})' = \frac{1}{2}x^2+y^2$.
\end{proof}

The geometric significance of $M$ is still somewhat mysterious. However, it
proves useful to express the second derivative of $x$ in terms of $M$.

\begin{lemma}
For any~\Sp{2}-invariant soliton the second derivative of $x$ satisfies
\label{lem:x''}
\begin{equation}
\label{eq:x''}
 3x'' = \frac{(g^3)'M}{xg^3} + \left(\left(\frac{1}{y^2} - \frac{1}{x^2}\right) M - \frac{9x}{y^2} - 2\lambda x \right) x'. 
\end{equation}
\end{lemma}

\begin{proof}
We can rewrite the first part of~\eqref{eq:xydot:S} as
\[ S = \frac{-(x^2 + 2y^2) + 6y^2x'}{4}, \]
and combine it with the second part of~\eqref{eq:xydot:S}  to obtain
\[ (g^3)'' = \left(\frac{x^2}{2} + y^2\right)'
= xx' + \frac{\frac{3}{2}(x^2 + 2y^2) - 3y^2x'}{3x}
= \frac{1}{2x}(x^2 + 2y^2) + \left(x - \frac{y^2}{x}\right)x' =
\frac{(g^3)'}{x} + \left(x - \frac{y^2}{x}\right)x'. \]
Then
\begin{align*}
3x'' &= (M' -\lambda x^2)' 
= \left((g^3)'\frac{M}{g^3} - \frac{3x^2}{y^2}\right)' - 2\lambda x x' \\
&= \frac{M}{g^3} (g^3)''- \frac{3x(g^3)'}{y^4} - \frac{6Sx}{y^4} - 2\lambda x x' \\
&= \left(\frac{(g^3)'}{x} +
\left(x - \frac{y^2}{x}\right)x'\right) \frac{M}{g^3} -
\frac{3x}{y^4}\left((g^3)' + 2S\right) - 2\lambda xx' \\
&= \frac{(g^3)'M}{xg^3} + \left(\frac{1}{y^2} - \frac{1}{x^2}\right)M x' - \frac{3x}{y^4} 3 y^2 x' -2\lambda x x'
\end{align*}
gives the claimed right-hand side after collecting all the $x'$ terms.
\end{proof}
Specialising the previous result to critical points of $x$ gives us the following important result. 

\begin{corollary}
\label{cor:x_monotone}
At any critical point of $x$, its second derivative $x''$ has the same sign as $M$.
In particular, $x$ has at most three critical points, so it is eventually monotone.
(In fact, for expanders and steady solitons $x$ has at most one critical point and this must be a nondegenerate local minimum.)
\end{corollary}

\begin{proof}
If $x' = 0$ then Lemma \ref{lem:x''} yields
\[
x'' = \frac{M}{g^3} \frac{(g^3)'}{3x}
\]
and other than $M$ all factors on the right-hand side of this equation are positive.
Also notice that it follows from~\eqref{eq:M'} that when $x'=0$
\[
\frac{M}{3g^3} \frac{(g^3)'}{x} = \frac{x}{3y^2}(3+\lambda y^2).
\]
Hence if $\lambda \ge 0$ then, in fact, $x'' > 0$ at any critical point of $x$ and therefore
$x$ has at most one critical point which can only be a nondegenerate local minimum. 
(In particular, when $\lambda\ge 0$, if $x$ has a critical point then $x$ is eventually increasing; so if $x$ is eventually decreasing then it must, in fact, be globally decreasing). 
In the shrinker case this is no longer true, but since $\frac{M}{g^3}$ is strictly decreasing, $x$ can have
\begin{itemize}[left=0em]
\item
at most one nondegenerate local minimum of $x$ and this can only occur when $\frac{M}{g^3}$ is positive
\item
at most one degenerate critical point of $x$ and this can only occur when $\frac{M}{g^3}=0$
\item
at most one nondegenerate local maximum of $x$ and this can only occur when
$\frac{M}{g^3}$ is negative. \qedhere 
\end{itemize}
\end{proof}
 
The following result improves further on the eventual monotonicity of $x$ just established.
\begin{prop}
\label{prop:x_limit}
For any~\Sp{2}-invariant soliton, either 
\begin{enumerate}[left=0em]
\item $x \to 0$ or
\item $x \to \infty$ or
\item $\lambda = 0$, the solution is forward complete %
and $x \to -\frac{\tilde\tau_1}{3}>0$ as $t \to \infty$
\end{enumerate}
as $t$ approaches the end of the lifetime 
of the solution (which can be either finite or infinite).
\end{prop}

\begin{proof}
Since $x$ is eventually monotone by Corollary~\ref{cor:x_monotone}, if $x$ is not bounded away from 0 then (i) holds, and if $x$ is unbounded above then (ii) holds.

So now suppose $x$ is bounded both above and away from 0. It remains to prove that then (iii) holds. 

First note that the solution must be forward-complete by Corollary \ref{cor:lifetime}. 
Since $x$ is eventually monotone and bounded, $x$ has a limit $x_\infty > 0$, and $x' \to 0$ as $t \to \infty$. 
Because $x$ is bounded above, $xy^2$ grows exponentially in $t$ and hence so does $y^2$. 
Boundedness of $x$ also implies that  $\tilde \tau_1$ grows at most linearly in $t$ and hence so does
$M=3x +\tilde \tau_1$. 

Now consider the large-$t$ behaviour of the expression for $3x''$ from~\eqref{eq:x''}. 
The first term behaves like
\[
\frac{(g^3)'M}{xg^3} = M \left( \frac{1}{2y^2} + \frac{1}{x^2} \right) = \frac{M}{x^2} + \text {an exponentially decaying term}.
\]
The coefficient of $x'$ behaves like 
\[
\left( \frac{1}{y^2}- \frac{1}{x^2} \right) M - \frac{9x}{y^2} - 2 \lambda x = - \frac{M}{x^2} - 2\lambda x + \text {an exponentially decaying term}.
\]
Thus for $t$ sufficiently large
\[
3 x'' \sim \frac{M}{x^2} -  \left(\frac{M}{x^2}+2\lambda x\right) x' \sim \frac{M}{x_\infty^2},
\]
since $x$ and $\frac{1}{x}$ are bounded and $x' \to 0$. 
Given that $x' \to 0$ as $t \to \infty$ and that $M$ changes sign at most once, this is only possible if $M \to 0$. But if $\lambda \not= 0$, then
 $\tilde \tau_1 \sim \lambda x_\infty^2\, t$ (since $\tilde \tau_1' = \lambda x^2$) and hence $M=3x + \tilde \tau_1$ grows linearly in $t$. So the only possibility is that $\lambda = 0$. Then $\tilde \tau_1$ is constant, so~$M \to 0$ simply means $x \to -\frac{\tilde\tau_1}{3}$ as claimed.
 \end{proof}

\begin{remark}
As a special case of \cite[Corollary 8.14]{Haskins:Nordstrom:g2soliton1}, there does, in fact, exist a forward-complete \Sp{2}-invariant steady end with exponential volume growth and
$x \to -\frac{\tilde\tau_1}{3}$.
\end{remark}

\subsection{Eventual monotonicity of the warping}
\label{ss:warping:monotone}

We now prove Proposition \ref{mthm:warping}(iii).

\begin{theorem}
\label{thm:y:x:eventually:mono}
For any \Sp{2}-invariant soliton, the warping~$\frac{y}{x}$ is eventually monotonic.
Moreover, except on a Gaussian soliton (which has $\frac{y}{x} \equiv 1$), the warping~$\frac{y}{x}$ is eventually strictly monotone.
\end{theorem}

An easy consequence of the previous result is the following characterisation of weakly AC solitons. 
\begin{corollary}
\label{cor:weak:AC}
An~\Sp{2}-invariant soliton is weakly AC if and only if $\log{\frac{y}{x}}$ is bounded throughout its lifetime.
\end{corollary}
\begin{proof}
One direction is immediate from the definition \eqref{eq:weak_ac_def} of
being weakly AC.
If $\log{\frac{y}{x}}$ is bounded then by Proposition~\ref{prop:lifetime} 
the solution has infinite lifetime. By Theorem~\ref{thm:y:x:eventually:mono}
$\log{\frac{y}{x}}$ is eventually monotone and hence because it is bounded it has a limit as $t \to \infty$. 
Therefore by Lemma~\ref{lem:yx:bounded} the solution is weakly AC. 
\end{proof}

While the steady case of Theorem \ref{thm:y:x:eventually:mono} can be viewed as
a consequence of the trichotomy for steady~\sunitary{3}-invariant ends
in \cite[Theorem F]{Haskins:Nordstrom:g2soliton1},
let us give a brief self-contained proof here.

\begin{lemma}
For any \Sp{2}-invariant steady soliton (other than the torsion-free cone),
the warping~$\frac{y}{x}$ has at most one critical point.
\end{lemma}

\begin{proof}
Recall that if $\lambda = 0$ then $\tilde\tau_2$ is conserved
by \eqref{eq:tt'}. Thus the sign of $\tau_2$ is constant
by \eqref{eq:tilde:tau1:tau2}.
Therefore \eqref{S:dot:zero} implies that $S$ has at most one zero
(unless $S \equiv 0$, which only happens for the static torsion-free cone
solution) and those correspond to the critical points of~$\frac{y}{x}$ by
Lemma \ref{lem:log_xy}.
\end{proof}

The proofs of Theorem \ref{thm:y:x:eventually:mono} in the expander and
shrinker cases require more work, and are somewhat different from each other.
First we make an elementary observation about the nature of critical points
of $\frac{y}{x}$. 

\begin{lemma}
\label{lem:y:x:critical}
For any~\Sp{2}-invariant soliton, except the Gaussian solitons over the torsion-free cone described in Remark~\ref{rmk:gaussian}, the following hold:
\begin{enumerate}[left=0em]
\item
All critical points of $\frac{y}{x}$ are nondegenerate.
\item  
A critical point of $\frac{y}{x}$ is a local maximum (minimum) if and only if $\frac{y}{x}>1$ (respectively $\frac{y}{x}<1$).
\end{enumerate}
\end{lemma}

\begin{proof}
First observe that $\frac{y}{x} \equiv 1$ on the Gaussian solitons and that these are the only solitons 
for which $\frac{y}{x}$ is constant. Moreover, at any critical point of $\frac{y}{x}$, $S$ vanishes by~\eqref{eq:d:dt:log:yx},
and so if  $y-x$ also vanishes then so does $\tau_2$. 
By uniqueness of solutions to~\eqref{eq:ODE:tau2'}
this implies the solution is a Gaussian soliton.

So now suppose that the soliton is not a Gaussian soliton. 
At any critical point of $\frac{y}{x}$,  differentiating~\eqref{eq:d:dt:log:yx} implies
\[
\left(\log{\frac{y}{x}}\right)'' = -\frac{S'}{g^3} = - \frac{2}{x^2+2y^2} \slead(\tfrac{y}{x})  = - \left(\frac{x^2+2y^2}{6g^6}\right) (y^2-x^2),
\]
where we have used~\eqref{eq:Sdot} and that $S=0$ at a critical point of $\frac{y}{x}$. 
By the previous observation, $y-x$ cannot vanish at a critical point of $\frac{y}{x}$, and hence 
any critical point is nondegenerate.
Notice that the right-hand side of the previous equation has the opposite sign to $\log{\frac{y}{x}}$. Hence
a critical point of $\frac{y}{x}$ is a local maximum if and only if $\frac{y}{x}>1$
(respectively a local minimum if and only if $\frac{y}{x}<1$).
\end{proof}

\begin{lemma}
\label{lem:osc}
For any~\Sp{2}-invariant $\lambda$-soliton, 
given any $\ell>0$, if $x \to \infty$ and there are infinitely many points where $y = \ell x$, 
then the decreasing function $\frac{M}{g^3}$ has limit $\frac{2\lambda}{1 + 2\ell^2}$.
\end{lemma}

\begin{proof}
Using~\eqref{eq:tilde:tau1:tau2} we have
\begin{equation}
\label{eq:MS}
M = 3x+ \tilde \tau_1 = 3x + \frac{2x^2R_1}{x^2 + 2y^2}
= 3x + \frac{2x^2(\lambda xy^2 - 3 \tau_2)}{x^2 + 2y^2} = \frac{2\lambda x^3y^2 + x(7x^2 + 2y^2 + 4S)}{x^2+2y^2}
\end{equation}
and so at any point where $y = \ell x$
\[ \frac{M}{g^3} = \frac{2\lambda}{1+2\ell^2} + \left(\frac{7 + 2\ell^2}{(1 + 2 \ell^2)\ell^2 x^2} + \frac{4S}{\ell^2(1+2\ell^2)x^4} \right) . \]
The first term is constant, so the second term is decreasing. Because  we assumed $x \to \infty$, while $S$ alternates signs between visits to $y = \ell x$, the limit of the second term must be 0.
\end{proof}

\begin{prop}
\label{prop:y:x:eventually:mono}
For any \Sp{2}-invariant non-steady soliton, unless $\frac{y}{x} \to 1$ at the
end then $S$ changes sign at most finitely often forwards in time, and so
$\frac{y}{x}$ is eventually monotonic.
\end{prop}

\begin{proof}
By Proposition \ref{prop:x_limit} either $x \to 0$ or $x \to \infty$ monotonically as we approach $t_*$, the end of the lifetime (which could be finite or infinite). 
The result in the former case is immediate, since then eventually $x'<0$ and this forces that $S<0$ eventually. 
So we may assume $x \to \infty$ and $x$ is eventually increasing. 

On any soliton other than a Gaussian soliton, Lemma~\ref{lem:y:x:critical}(i) tells us that at a local maximum of~$\frac{y}{x}$ we have $\frac{y}{x} > 1$ and at a local minimum $\frac{y}{x} < 1$. 
Hence between successive critical points of $\frac{y}{x}$ there is always a (unique) point where $\frac{y}{x}=1$.

Suppose for a contradiction that $\frac{y}{x}$ is not eventually monotone.  Then $\frac{y}{x}$ has infinitely many critical points as $t \to t_*$ and
thus there must also be infinitely many points $t_i \to t_*$  where $\frac{y}{x} = 1$.

If $\frac{y}{x}$ failed to converge to 1, then  there must also be some $\ell \not=1$ such that $\frac{y}{x} = \ell$ holds infinitely often. But that is impossible because
we have assumed $x \to \infty$ and thus Lemma \ref{lem:osc} implies that the decreasing function $\frac{M}{g^3}$ would have two different limits
(because $\frac{y}{x}$ assumes both the values $1$ and $\ell \neq 1$ infinitely often).
\end{proof}

By Proposition~\ref{mthm:warping}(i) and (ii)  any soliton where $\frac{y}{x} \to 1$ is forward complete with a weakly AC end asymptotic to the torsion-free cone. 
We can complete the proof of Theorem \ref{thm:y:x:eventually:mono} in the shrinker case  as follows. 
By Proposition \ref{prop:reg} below, any weakly AC end is, in fact,  AC in a $C^0$ rate $-2$ sense. That allows us to apply the uniqueness result 
for gradient AC shrinkers proven by Haskins--Khan--Payne~\cite{HKP} (see Theorem \ref{thm:HKP} for a precise statement) to deduce that the shrinker is actually the Gaussian. (In fact, one can also give a self-contained proof that for any~\Sp{2}-invariant closed cone~$C$, 
there is a unique~\Sp{2}-invariant shrinker end asymptotic to $C$. Taking $C$ to be the torsion-free cone again implies that the unique AC shrinker end asymptotic to $C$ is $C$ itself.)

In the expander case, %
dealing with the situation where $\frac{y}{x} \to 1$ is actually no easier than proving the following stronger version of 
Theorem \ref{thm:y:x:eventually:mono} from scratch.

\begin{prop}
\label{cor:yx:crit}
On any (non-Gaussian) expander the warping $\frac{y}{x}$ has at most one critical point.
Moreover,  if the warping $\frac{y}{x}$ is not globally monotone then it must converge to a limit $\ell \in (0,\infty)$
and hence the solution is weakly AC. 

In particular, $\frac{y}{x}$ is strictly monotone on any expander on which
$\log{\frac{y}{x}}$ is unbounded (which includes all forward-incomplete
expanders).
\end{prop}

\begin{proof}
Suppose that $\frac{y}{x}$ has a local maximum at $t=t_0$. Then $\frac{y}{x}(t_0)>1$, $S(t_0)=0$ 
and hence also $\tau_2(t_0)>0$; 
moreover, $S'(t_0) = \left(\frac{2g^3}{x^2+2y^2}\right) \slead(t_0)>0$, so 
$S$ must change sign from negative to positive as we move past $t_0$.
Therefore for $t=t_0+ \epsilon$ with $\epsilon>0$ sufficiently small we have $S>0$, $\tau_2>0$. 
Then by Proposition~\ref{prop:elem_AC} these conditions persist (so, in particular, there are no further critical points of $\frac{y}{x}$) and the solution is weakly AC.

The proof in the case that $\frac{y}{x}$ has a local minimum is entirely analogous. %
\end{proof}

\subsection{Criterion for weak asymptotic conicality of expanders}
\label{ss:weak:ac:expand}

In this subsection we give necessary and
sufficient conditions for an \Sp{2}-invariant expander to be weakly AC; we show, in particular, that such initial conditions form an
open region in the phase space. In other words, for expanders having an AC end is a stable  or generic property. 

Recall from Section~\ref{ss:tilde_tau} that on expanders the adjusted torsion
$\tilde \tau_1$ and $\tilde \tau_2$ are increasing functions of $t$, at least one of which must be positive
at any instant (and thereafter it must remain positive). 
Clearly, if an expander is weakly AC then $\log{\frac{y}{x}}$ remains bounded and eventually $\tilde \tau_1$ and $\tilde \tau_2$ are both positive.
The next result gives a converse: if both $\tilde \tau_1$ and $\tilde \tau_2$ ever
become positive, then $\log{\frac{y}{x}}$ remains bounded and
hence by the eventual monotonicity of $\frac{y}{x}$ it converges to some limiting value,
and therefore by Lemma~\ref{lem:yx:bounded} it is weakly AC. 

\begin{prop}
\label{prop:xy:bounds}For any~\Sp{2}-invariant expander the following hold:
\begin{enumerate}
\item If $\tilde \tau_1 > 0$ then $\frac{y}{x}$ is bounded above for all future time. 
\item If $\tilde \tau_2 > 0$ then $\frac{x}{y}$ is bounded above for all future time. 
\end{enumerate}
\end{prop}

\begin{proof}
For (i), write $h = \frac{y}{x}$. 
Recall from \eqref{eq:tau2':alt} that
\[ \tau_2' = -\frac{1}{3} \tilde \tau_1  (h^2)'. \]
By Corollary~\ref{cor:yx:crit},  eventually $h$ is monotone.  If $h$ is eventually decreasing then certainly it is bounded above. 
So it suffices to consider the case that $h$ is eventually increasing (so $S$ is eventually negative). 
If $\tilde\tau_1(t_0) = A > 0$, then the fact that $\tilde\tau_1$ is increasing implies that for any $t_1$ such that
$h(t) \leq h(t_1)$ for all $t \in [t_0, t_1]$
\[ \tau_2(t_1) < \frac{A}{3}\left(h(t_0)^2 - h(t_1)^2\right) + \tau_2(t_0) . \]
Hence given any initial values $h(t_0)$ and $\tau_2(t_0)$ there exists an $M > 1$
such that $h(t_1) \ge M>1$ implies $\tau_2(t_1)<0$. Therefore if $h$ is unbounded above then 
there exists $t_1$ so that $\tau_2(t_1)<0$. Since $h(t_1)>1$ and $\tau_2(t_1)<0$, $S(t_1)$ must be positive, which is a contradiction.

For (ii), write $k=\frac{x}{y}$.
Recall that
\[
\tau_1' = - \frac{2}{3} \tilde \tau_2  (k^2)'.
\]
Corollary~\ref{cor:yx:crit} implies that $k$ is eventually monotone and clearly it suffices to consider the case that $k$ is eventually increasing. 
Then we can apply the obvious analogue of the argument from  part~(i) to conclude that $k=\frac{x}{y}$ is bounded above for all future times. 
\end{proof}

We now improve on Proposition~\ref{prop:xy:bounds} by showing
that, in fact, for expanders $\frac{x}{y}$ is always forward-bounded above.
In turn,  this allows us to prove that if $\tilde \tau_1$ is eventually positive
then the solution is weakly AC (implying that
eventually $\tilde \tau_2$ becomes positive too).

\begin{prop}
\label{prop:x:over:y:bded}
On any expander $y$ and $\frac{y}{x}$ are bounded away from $0$ for all future $t$.
\end{prop}

\begin{proof}
Recall from~\eqref{eq:monotonic} that 
\[
\frac{M}{xy^2} = \frac{3x+ \tilde \tau_1}{g^3} = \frac{3}{y^2} + \frac{\tilde \tau_1}{g^3}
\]
is decreasing in $t$ and  hence it is certainly bounded above.
Because $\tilde \tau_1$ is increasing while
$g^3$ is positive and bounded away from 0, the term $\frac{\tilde\tau_1}{g^3}$
is bounded below. Therefore $\frac{1}{y^2}$ is bounded above.

Now suppose $\frac{x}{y}$ were unbounded above. 
Then by the eventual monotonicity of $\frac{x}{y}$, it is eventually increasing. Hence $S$ is eventually positive, which
implies $x$ is eventually increasing and so $x \to \infty$ by Proposition~\ref{prop:x_limit}. 
By Proposition~\ref{prop:xy:bounds}, $\frac{x}{y}$ can be unbounded above only if $\tilde \tau_2$ remains negative, and so
$\tilde\tau_1$ must remain positive by~\eqref{eq:ODE:conserve2}.
Since $\tilde \tau_2$ is negative and increasing it is bounded. 

The lower bound for $y$ easily rules out the possibility of infinite lifetime: 
because $\tilde \tau_2' = \lambda y^2>C>0$, then eventually $\tilde \tau_2$ would become positive. 
It remains to rule out the possibility that the solution has finite lifetime.
Because $\tilde \tau_1$ is  positive and $x$ is eventually increasing, (eventually) we have the bound 
\[
\frac{d}{dt} \left(\frac{\tilde \tau_1}{x^2}\right) = \lambda - \frac{\tilde \tau_1 (x^2)'}{x^4} < \lambda.
\]
Therefore $\frac{\tilde \tau_1}{x^2}$ remains bounded above on  any finite time interval. 
Eventually $x>1$ holds and then $(\log{x})'$ satisfies the bound
\[ (\log x)' = \frac{1}{x} - \frac{x}{2y^2} + \frac{\tilde\tau_1}{3x^2} -\frac{\tilde\tau_2}{3y^2} < 1 + \frac{\tilde\tau_1}{3x^2} -\frac{\tilde\tau_2}{3y^2}.\]
Hence on any finite interval $(\log{x})'$ is bounded above: because  $\frac{\tilde \tau_1}{x^2}$ is bounded above, $y$ is bounded away from $0$ and $\tilde \tau_2$ is bounded.
So $x$ is bounded above on any finite interval, contradicting that~$x \to \infty$.
\end{proof}

Two easy consequences of the previous proposition
are the following results. 

\begin{corollary}
\label{cor:y:x:inc}
On any non-AC~\Sp{2}-invariant expander, the warping $\frac{y}{x}$ is strictly increasing throughout its lifetime 
with $\frac{y}{x} \to \infty$ at the end of its lifetime (which could be finite or infinite). 
\end{corollary}
\begin{proof}
Corollary~\ref{cor:yx:crit} implies that the solution is weakly AC, unless $\frac{y}{x}$ is monotone throughout. 
By Proposition~\ref{prop:x:over:y:bded}, $\frac{y}{x}$ is bounded away from $0$ and so if $\frac{y}{x}$ were decreasing then 
$\frac{y}{x}$ would converge to some limit $\ell>0$ as $t \to \infty$. Then by Lemma~\ref{lem:yx:bounded}
the solution would be AC. Hence $\frac{y}{x}$ must be strictly increasing and unbounded above.
\end{proof}

\begin{lemma}
\label{lem:extinction_x}
$x \to 0$ and $\frac{x}{y} \to 0$ at the extinction time of any forward-incomplete expander.
In fact, $x$ is decreasing and $S$ is negative throughout its lifetime. 
\end{lemma}

\begin{proof}
By Corollary~\ref{cor:y:x:inc} $S$ is negative throughout the lifetime of any forward-incomplete expander
and $\frac{x}{y} \to 0$ at its extinction time. 
Since, by Corollary~\ref{cor:x_monotone}, $x$ is eventually monotone, to prove that~$x \to 0$ it suffices to prove that $x$ cannot be bounded away from $0$. 
Together with the upper bound for $\frac{x}{y}$ from
Proposition \ref{prop:x:over:y:bded}, bounding $x$ away from 0 would imply
infinite lifetime by Corollary~\ref{cor:lifetime}.
Finally, by the critical point analysis for $x$ (Corollary~\ref{cor:x_monotone} again), any expander for which $x$ is eventually decreasing must, in fact, have $x$ decreasing throughout its lifetime.
\end{proof}

We can now state 
our most general characterisation of which expanders have weakly AC ends. 
\begin{prop}
\label{prop:AC:end}
For an \Sp{2}-invariant expander, the following are equivalent.
\begin{enumerate}[left=0em]
\item $\tilde \tau_1$ ever becomes positive
\item the solution is forward complete, and $\frac{y}{x}$ has a limit $\ell \in (0,\infty)$ as $t \to \infty$
\item $\frac{y}{x}$ is not strictly increasing and unbounded
\end{enumerate}
In this situation the expander is weakly asymptotic to the unique
closed~\Sp{2}-invariant~\gtwo-cone
with  $c_2/c_1=\ell$. 
In particular, $\frac{x}{t} \to c_1$ and $\frac{y}{t} \to c_2$, where $c_1 = \frac{1}{6} \left(2 + \ell^{-2} \right)$ and $c_2 =  \frac{\ell}{6} \left(2 + \ell^{-2} \right)$.  Eventually $\tilde \tau_2$ is also positive.
\end{prop}
\begin{remark*}
By (iii) we mean that either $\frac{y}{x}$ is bounded above or $\frac{y}{x}$ is not strictly increasing. 
\end{remark*}

\begin{proof}
(i) implies $\frac{y}{x}$ is bounded above by
Proposition \ref{prop:xy:bounds}, and hence (iii) holds. 

By Corollary \ref{cor:y:x:inc}, an expander is forward complete with
$\frac{y}{x}$ converging as $t \to \infty$ unless $\frac{y}{x}$ is strictly increasing and unbounded.
In other words,  (iii) implies (ii).

Since $xy^2$ is increasing, (ii) certainly implies that $x$ is bounded away
from 0.
Then $\tilde\tau_1' = \lambda x^2$ implies that
$\tilde \tau_1$ must eventually become positive, so that (i) holds. 

(ii) also implies weak conicality by Lemma \ref{lem:yx:bounded},
and $y$ being bounded away from $0$ implies that $\tilde \tau_2$ eventually becomes positive.
\end{proof}

\begin{remark}
Proposition \ref{prop:AC:end} implies that set of initial conditions that
lead to weakly AC expanders is open in the phase space.
We discuss in \S\ref{sec:ends} that the boundary of this open
region corresponds to expanders that are forward complete but with
quadratic-exponential volume growth; for such solutions the increasing function $\tilde \tau_1$ remains negative
but tends to $0$ as $t \to \infty$.
\end{remark}

\section{Asymptotically conical non-steady ends}
\label{s:AC:geometry:Sp2:expanders}

Proposition \ref{prop:AC:end} implies that having a weakly AC end 
is a stable end behaviour for expanders (this is not true for shrinkers; see \S\ref{subsec:rigid_vs_instab}
for further discussion of the contrast in stability between AC shrinkers and AC
expanders). 
We also established previously 
that all smoothly-closing expanders have weakly AC ends. 
In this section, we strengthen these results in two directions.

The first main result of this section (proven in ~\S\,\ref{subsec:conv_S}--\ref{ss:rate-2:AC}) holds %
for both  expanders and shrinkers and amounts to a regularity result for non-steady weakly AC ends. 
It shows that any~\Sp{2}-invariant non-steady soliton with $\log \frac{y}{x}$ bounded
is AC not only in the weak sense of \eqref{eq:weak_ac_def} %
but also
in a stronger ``$C^0$ rate $-2$'' sense. (In contrast,
the AC steady solitons that two of the authors studied
in~\cite{Haskins:Nordstrom:g2soliton1} generically have rate $-1$.)
A key ingredient here is to prove that 
the quantity $S$ remains bounded and, in fact, converges as $t \to \infty$ to a value that
is determined by the asymptotic cone of the solution. 

The other main result in this section (proven in~\S\,\ref{ss:AC:expander:end:stability}) establishes 
a continuity result for AC expander ends. Namely, 
the condition
for an \Sp{2}-invariant expander to be weakly AC is not only stable, but moreover
the asymptotic cone depends continuously on the initial conditions.
This result applies, in particular, to show that the asymptotic cone 
of a smoothly-closing expander depends continuously on the parameter $\lambda b^2$. 

Finally we point out that
both of the main results of this section rely on a technical result, Lemma~\ref{lem:Sdot}, proven in~\S\,\ref{subsec:conv_S}.

\subsection{Convergence of \texorpdfstring{$S$}{S}}
\label{subsec:conv_S} 

For a general \Sp{2}-invariant conical \gtstr{}, the quantity $S$
grows like $t^2$ (as one would also expect from the description of the scaling
behaviour in Remark \ref{rmk:scaling}, which shows $S$ to
have ``scaling weight 2'').
However, for a \emph{closed} cone
Lemma \ref{lem:log_xy} implies that $S \equiv 0$.

Therefore, for a closed \gtstr{} that is asymptotically conical with
rate $\nu$, in the sense that the difference between $x, y$ and the cone
is $O(t^{1+\nu})$, the growth rate of $S$ is controlled by the error terms,
so that $S = O(t^{2+\nu})$.

For non-steady AC soliton ends we will turn this argument around: first we prove that $S$ in fact converges to a finite limit as $t \to \infty$, and then leverage that fact to prove that the AC rate is $-2$. 
(In contrast, on AC \emph{steady} soliton ends,
which have rate $-1$, $S$ is instead asymptotically linear.)

The main result of this section is the following:
\begin{prop}
\label{prop:S:limit}
Suppose that a forward-complete \Sp{2}-invariant non-steady soliton is weakly
asymptotic to the unique closed cone $(c_1,c_2)$ with $\frac{c_2}{c_1} = \ell$,
\ie $\frac{x}{t} \to c_1$, $\frac{y}{t} \to c_2$; equivalently $\frac{y}{x} \to \ell$ by above. Then
the quantity $S$ converges as $t \to \infty$ to the finite value 
\begin{equation}
\label{eq:s*_from_c}
S^*:= 
\frac{(c_1^2+2c_2^2)^2(c_2^2-c_1^2)}{12\lambda (c_1c_2^2)^2}
= \frac{3}{\lambda}(c_2^2-c_1^2) = \frac{1}{12\lambda} (\ell^2-1)(2+\ell^{-2})^2.
\end{equation}
In particular, for any smoothly-closing expander $\spfam{\lambda}{b}$,
$S \to S^*$ as $t \to \infty$ with $S^* \ge 0$, with equality if and only if it is asymptotic to the torsion-free cone $\ell=1$. 
\end{prop}

As the reader will see, our proof of Proposition~\ref{prop:S:limit} is sensitive to the sign of $\lambda$. 
Before proving this result we isolate a general observation, Lemma~\ref{lem:Sdot}, about the behaviour of $|S-S^*|$ 
on non-steady solitons defined on sufficiently long $t$-intervals.
This result will be used in our proof of Proposition~\ref{prop:S:limit} 
and again in our later study of the stability of AC expander ends.

Note that \eqref{eq:s*_from_c} can be thought of as a continuous function
of~$\ell = \frac{c_2}{c_1}$; indeed, we can also write it
as $S^* = \frac{1}{\lambda}\slead\left(\frac{c_2}{c_1}\right)$ where
\begin{equation*}
\slead(l) := \frac{1}{12}(l^2-1)(2 + l^{-2})^2 .
\end{equation*}
Since $\frac{d\slead}{dl} > 0$ for all $l >0$ and $\slead(l) \to - \infty$ as $l \to 0$ and $\slead(l) \to +\infty$ as $l \to \infty$, 
the map $\slead : (0,\infty) \to \R$ is a smooth
bijection. 
For any soliton $(x,y)$ we can consider $\slead(\frac{y}{x})$, and saying that $\slead(\frac{y}{x})$ is close to
$\slead(\frac{c_2}{c_1})$ is just another way to say that $(x,y)$ is close to
$(c_1, c_2)$ up to scale.

\begin{lemma}
\label{lem:Sdot}
For any constants $S^* \in \R$, $\lambda \not= 0$, $B > 0$ and $C > 0$
there is $g_0 > 0$ (depending on $\lambda$, $B$ and $C$) such that at any point of any $\lambda$-soliton
where 
\[ g > g_0, \quad
\left|\frac{\slead\left(\tfrac{y}{x}\right)}{\lambda} - S^*\right|< \frac{1}{5}C \quad\textrm{ and } \quad C \le  |S - S^*| \le B g^2, \]
hold then one has
\[ \frac{1}{\lambda}\frac{d}{dt} |S-S^*| <  -\frac{C g^3}{x^2 +2y^2} < 0 . 
\]
\end{lemma}

\begin{proof}
First we recall from~\eqref{eq:Sdot} that the ODE satisfied by $S'$ is
\begin{equation*}
S' =\left( \frac{2xy^2}{x^2+2y^2}\right) (\slead - \sother S),
\end{equation*}
where $\slead=\slead(\frac{y}{x})$ is as above and 
\[
\sother = \lambda + \frac{ x^4+12x^2y^2-4y^4}{4x^2y^4} + \frac{(4y^2-x^2)S}{3x^2y^4}.
\]
The precise expression for $\sother$ is not so important to our argument. What does matter
is that while $\slead(\frac{y}{x})$ and (patently) $\lambda$ are independent
of the scale of $(x,y)$, the non-$\lambda$ terms in $\sother$ have scaling
weight $-2$, and our hypotheses imply that $x$, $y$ and $S$ really do have
bounds corresponding to their scaling weight:
$\slead(\frac{y}{x})$ bounded implies $\frac{y}{x}$ and $\frac{x}{y}$ are both bounded, so $x$ and $y$ are both $O(g)$, while $|S-S^*| < Bg^2$ implies
$S = O(g^2)$.
Thus $\sother = \lambda + O(g^{-2})$.

It is then clear that for any $B>0$ and $C>0$  it is possible to choose $g_0$ sufficiently large so that both
\[
\frac{\sother}{\lambda}> \frac{5}{6} \quad \text{and} \quad \left|\frac{\slead}{\sother} - \frac{\slead}{\lambda}\right| < \frac{C}{5}
\]
hold whenever $g>g_0$, $|\lambda^{-1}\slead(\frac{y}{x}) - S^*| < \frac{C}{5}$
and $|S-S^*| < Bg^2$.
If moreover, $|S-S^*|  \ge C$
then rewriting \eqref{eq:Sdot} as
\[ \frac{1}{\lambda}S' =
\frac{2g^3}{x^2 + 2y^2} \frac{\sother}{\lambda}
\left(\slead\left(\tfrac{1}{\sother} - \tfrac{1}{\lambda}\right)
+ \left(\tfrac{\slead}{\lambda} - S^*\right) + (S^* - S)\right) \]
shows that $\frac{1}{\lambda}S'$ has the opposite sign to $S-S^*$, and has absolute value greater than
\[ \left(\frac{2g^3}{x^2 + 2y^2}\right) \frac{5}{6}\left(-\frac{C}{5} -\frac{C}{5} + C\right)
= \frac{Cg^3}{x^2 + 2y^2}. \]
This is equivalent to the desired conclusion.
\end{proof}

We will use Lemma \ref{lem:Sdot}
to show that the following dichotomy holds:
when $S$ is not already close to $S^*= \tfrac{1}{\lambda}\slead \left(\tfrac{c_2}{c_1} \right)$ while $S/g^2$ is bounded, then
in the expander case $S$ is attracted to $S^*$, while in the shrinker case $S$ is repelled away from $S^*$.
Either way, we will be able to use that information to deduce that $S \to S^*$, provided that asymptotically $\sis=S/g^2$ is
not too big.

The following weak constraint on the asymptotic behaviour of $\sis$ suffices to make the
argument work.

\begin{lemma}
\label{lem:sis_accumulate}
On any weakly AC \Sp{2}-invariant closed \Sp{2}-invariant \gtstr,
$\sis:= S/g^2$ cannot be bounded away from 0.
\end{lemma}

\begin{proof}

Because $\frac{y}{x}$ converges as $t \to \infty$, so does
\[
\int_0^t \left(\frac{x}{y}\right)'  ds = \int_0^t \frac{S}{y^3} ds,
\]
where the last equality follows from~\eqref{eq:d:dt:log:yx}.
If $\sis$ were bounded away from zero, then since $g$ grows linearly in $t$, and
\[
\frac{S}{y^3} = \frac{x}{y} \frac{\sis}{g},
\] 
the above integral would be infinite, which gives a contradiction. 
\end{proof}

\begin{proof}[Proof of Proposition \ref{prop:S:limit}]
Our goal is to prove that $S \to S^*$ as $t \to \infty$. 
If we fix any $C > 0$,
continuity of $\slead$ and the hypothesis that
$\frac{x}{t} \to c_1$, $\frac{y}{t} \to c_2$ ensure that
\[
\left|S^* - \lambda^{-1} \slead(\tfrac{y}{x})\right| =
\left| \lambda^{-1} \left(\slead(\tfrac{c_2}{c_1}) - \slead(\tfrac{y}{x})\right)\right| < \tfrac{1}{5}C
\]
holds for all $t$ sufficiently large.
Therefore for any $B > 0$, Lemma \ref{lem:Sdot}
implies that we can choose $t_0$ sufficiently large so that
\begin{equation}
\label{eq:S:dot:bound}
 \frac{1}{\lambda}\frac{d}{dt}|S-S^*| < -\frac{Cg^3}{x^2 + 2y^2} \quad 
\text{holds whenever $t > t_0$ and $C \le |S-S^*| \le Bg^2$}.
\end{equation}
The cases where $\lambda>0$ and $\lambda<0$ must now be treated differently. 

Consider first the expander case. In this case we can choose any $B>0$ we like. 
Then because $\lambda>0$, and $\frac{g^3}{x^2+2y^2}$ grows linearly,
\eqref{eq:S:dot:bound} certainly implies that there is
a negative constant $K$ such that
\begin{equation}
\label{eq:S:dot:bd:expand}
 \frac{d}{dt}|S-S^*| < K < 0 \quad
\text{holds whenever $t > t_0$ and $C \le |S-S^*| \le Bg^2.$}
\end{equation}
Lemma \ref{lem:sis_accumulate} implies that for any $B>0$ there exists $t>t_0$ so that $|S-S^*| \le Bg^2$ holds at $t$.
If $|S-S^*|$ is not already less than $C$ at $t$, then~\eqref{eq:S:dot:bd:expand}
implies that past $t$,  $|S-S^*|$ decreases until it reaches $C$ at some
finite time $t_1> t> t_0$. Moreover, since $\frac{d}{dt}|S-S^*|<0$ still holds at $t_1$, then $|S-S^*|$
becomes less than $C$ immediately after $t_1$. Once $|S-S^*|$ decreases below $C$ it must it remain so for all $t > t_1$: 
if $\tilde t >t_1$ were the first instant at which $|S-S^*|=C$ then 
$\tfrac{d}{dt} \abs{S-S^*}$ would have to be nonnegative at $\tilde t$, which contradicts~\eqref{eq:S:dot:bd:expand}.
Since the constant $C>0$ was arbitrary this implies that $S \to S^*$.
Now consider the $\lambda < 0$ case.
Since $\frac{x^2 + 2y^2}{g^2}>0$ converges as $t \to \infty$, it is certainly bounded above. 
Hence there is a $B > 0$ and $t_1 > t_0$ such that 
\[
\frac{-3\lambda C g^4}{(x^2 + 2y^2)^2} >B \quad \text{for all\ } t > t_1.
\]
Now consider the nonnegative quantity
\[
\wt \sis := g^{-2}\left|S(t) -S^*\right|.
\]
Then~\eqref{eq:S:dot:bound} shows that if there exists $t>t_1$ for which
\begin{equation}
\label{eq:shrinker:bd}
\wt \sis \le B \quad \textrm{ and }
\quad \left|S -S^*\right| \ge  C \quad \text{both hold at $t$}
\end{equation}
then, not only is $\frac{d}{dt}\abs{S-S^*}$ positive at $t$, but so is $\wt \sis'$, because
\begin{equation}
\label{eq:sis:deriv:bd}
\begin{aligned} 
\frac{d\wt \sis}{dt} \; &= 
-\frac{x^2 + 2y^2}{3g^5} \abs{S(t)-S^*} + \frac{1}{g^2} \frac{d}{dt} \abs{S(t)-S^*} \\ &> \;
-\frac{x^2 + 2y^2}{3g^3} \wt \sis
\, + \, \frac{1}{3g^3} \frac{-3\lambda C g^4}{x^2 + 2y^2} \; > \;
\frac{x^2 + 2y^2}{3g^3}
 (B - \wt \sis)
\end{aligned}
\end{equation}
(the final inequality holds because of how we chose $B$ above). 
By making $t_1$ larger, we can in addition ensure that $\wt \sis \geq B$ implies
$|S-S^*| \geq C$ for $t > t_1$.
Now suppose for a contradiction that there exists $t_2 > t_1$
such that $|S(t_2) - S^*| \geq C$.
If $\wt \sis(t_3) = B$ at some $t_3 \geq t_2$ then $\wt \sis(t) > B$ 
for all $t>t_3$:  by~\eqref{eq:sis:deriv:bd} the derivative of $\wt \sis$ is
positive at $t_3$, as well as at the first instant $t > t_3$ at which
$\wt \sis(t) = B$, which would be a contradiction.
On the other hand, if $\wt \sis(t) < B$ for all $t \geq t_2$
then~\eqref{eq:sis:deriv:bd} implies that $\wt \sis$ is increasing.
In either scenario, $\wt \sis$ (and hence $\sis$) is bounded away from 0 as $t \to \infty$,
in contradiction to Lemma \ref{lem:sis_accumulate}.
Hence $|S-S^*| < C$ for all $t$ sufficiently large. 
Since $C$ was arbitrary this proves that~$S \to S^*$.
\end{proof}

In Corollary~\ref{cor:weak:AC} we deduced that weakly AC behaviour of a soliton is equivalent to the boundedness of $\log{\frac{y}{x}}$ 
by using the eventual monotonicity of the warping $\frac{y}{x}$ proven in Theorem~\ref{thm:y:x:eventually:mono}. 
We finish this subsection by giving a different proof of Corollary~\ref{cor:weak:AC} in the non-steady case 
by using the methods developed to prove Proposition~\ref{prop:S:limit}.
One advantage of this argument is that it
admits a generalisation to the~\sunitary{3}-invariant case (because Lemma~\ref{lem:Sdot} does), while the proof of the eventual monotonicity of the warping $\frac{y}{x}$ uses the decreasing quantity $M/g^3$ 
in a fundamental way and we have not yet found a substitute for this quantity in the~\sunitary{3}-invariant case. 
A further advantage is that it avoids having to appeal to the AC shrinker end uniqueness result~\cite{HKP}.

\begin{lemma}
\label{lem:weak_AC}
Any~\Sp{2}-invariant non-steady soliton for which $\log{\frac{y}{x}}$ remains bounded is weakly AC. 
In particular, the hypotheses of Proposition~\ref{prop:S:limit} can be weakened to just assume that $\log{\frac{y}{x}}$ remains bounded.
\end{lemma}
\begin{proof}
By the incompleteness criterion Proposition \ref{prop:lifetime},
boundedness of  $\log{\frac{y}{x}}$
implies that the solution has  infinite forward lifetime. 
It suffices to prove  that $S$ remains bounded as $t \to \infty$, assuming only the boundedness of $\log{\frac{y}{x}}$:
because then $g'$ is bounded, so $g$ grows linearly and then a bound on $S$ implies that the right-hand side of~\eqref{eq:d:dt:log:yx}
is integrable and hence $\log{\frac{y}{x}}$ has a finite limit as $t \to \infty$. Then Lemma~\ref{lem:yx:bounded}
implies that the soliton is weakly AC. 

To see that $S$ must remain bounded we argue as follows. The boundedness of $\log \frac{y}{x}$ implies 
that~$\slead( \frac{y}{x})$ is bounded as $t \to \infty$. Hence, there exists some $C>0$ so that
\[
\abs{\lambda^{-1} \slead (\tfrac{y}{x})} < \tfrac{1}{5}C
\]
holds for all~$t$ sufficiently large. The arguments from the proof of Proposition~\ref{prop:S:limit} (in both the expander and shrinker cases) can still be applied, but now with $S^*$ chosen to be $0$
(\ie appealing to Lemma~\ref{lem:Sdot} with $S^*=0$)
and only for this particular value of $C$, and they show that $\abs{S} = \abs{S-S^*} < C$ for all~$t$ sufficiently large. 
The key points to observe are that the boundedness of $\log{\frac{y}{x}}$ implies: 
(i)~bounds on~$g'$ and hence $g^3/(x^2+2y^2)$ still grows linearly with $t$ and $g^2/(x^2+2y^2)$ is bounded above;  
(ii)~that~$\sis$ still cannot be bounded away from $0$.
\end{proof}

\subsection{\texorpdfstring{$C^0$ rate $-2$}{C0 rate -2} convergence of non-steady AC ends}
\label{ss:rate-2:AC}
The convergence of $\frac{y}{x}$ and $S$ already established also ensure, thanks to~\eqref{eq:ODE:tau2'},
that the limits of $x', y'$ and $\tau_2'$ as $t \to \infty$ all exist and are determined uniquely by $\ell= \lim_{t \to \infty}\frac{y}{x}$. 
One might therefore seek to go further and establish that any non-steady AC end
has a well-defined smooth asymptotic expansion at $t = \infty$.
We outline in~\S\ref{subsec:end_sivp} how this can be done by
solving a singular initial value problem in terms of $u = t^{-2}$.
However, while the coefficients of a (unique) formal power series in $u$ can be
identified from \eqref{eq:irreg_ode}, it
requires a nontrivial investment to prove that such a formal power series solution really corresponds to the derivatives of a genuine smooth solution to the problem (because the singularity at $u=0$ turns out to be irregular).

Instead we focus here on pinning down the ``next'' coefficients of $x$ and $y$,
which turns out to be crucial (and sufficient) for the purposes of the current
paper (in particular, for the proof of Theorem~\ref{mthm:asymptotic:limit}
in \S\ref{s:asymptotic:cones}).
The basic reasoning is as follows:
because the differential equation satisfied by~$\frac{x}{y}$ is so simple, we can
make use of the convergence of $S$ %
to control the ``next'' coefficients by applying the following simple lemma
several times.

\begin{lemma}
\label{lem:mu_ode3}
Let $\nu > 1$, and $a \in \R$ and suppose that $t^\nu \frac{df}{dt} \to a$ as $t \to \infty$.
Then $f$ has a limit $f_\infty$, and
\[ t^{\nu-1}(f-f_\infty) \to \frac{a}{1-\nu} . \]
\end{lemma}

\begin{proof}
That $f$ has a limit is simply a consequence of the integral
$\int_{t_0}^\infty f'(s)ds$ being finite.
Now letting $g(t) = t^\nu f'(t)$ we can write
\[ f_\infty - f = \int_t^\infty s^{-\nu}g(s)ds . \]
Changing variables by $s = tu$ we get
\[ t^{\nu-1}(f-f_\infty) = - \int_1^\infty u^{-\nu}g(tu)du \to
-a \int_1^\infty u^{-\nu}du = \frac{a}{1-\nu} . \qedhere \]
\end{proof}

\begin{prop}
\label{prop:reg}
Let $(x,y,S)$ be any forward-complete \Sp{2}-invariant soliton
with nonzero dilation constant $\lambda$ 
such that  $\frac{x}{t} \to c_1$ and $\frac{y}{t} \to c_2$. Then
\begin{equation}
\label{eq:xy_coeff}
\frac{x}{y} =
\frac{c_1}{c_2}\left(1 - \frac{S^*}{2c_1c_2^2}t^{-2} + o(t^{-2})\right) ,
\end{equation}
where $S^*$ is the limiting value of $S$ given in~\eqref{eq:s*_from_c}.
Moreover, up to making a suitable translation $t \mapsto t + k$, we have
\begin{equation} 
\label{eq:x:y:larget:exp}
x = c_1t + \xi S^* t^{-1} + o(t^{-1}),
\qquad y = c_2t + \zeta S^* t^{-1} + o(t^{-1})
\end{equation}
where $\xi$ and $\zeta$ are determined uniquely by $(c_1,c_2)$ via
\[ \xi = \frac{c_1^2 - 4c_2^2}{18 c_1c_2^4},
\quad
\zeta = \frac{5c_1^2 - 2c_2^2}{36 c_1^2c_2^3} . \]
\end{prop}
By Theorem~\ref{thm:sc:expanders:complete}, Proposition~\ref{prop:reg} applies in particular to any smoothly-closing expander, thus showing
Theorem~\ref{mthm:expand:Sp2}(iii) and so completing the proof of Theorem~\ref{mthm:expand:Sp2}.
Meanwhile Theorem \ref{mthm:AC} is proved by combining Lemma \ref{lem:weak_AC}
and Propositions \ref{prop:S:limit} and \ref{prop:reg}.

\begin{remark}
\label{rmk:precise_rate}
If $(c_1,c_2)$ is the unique closed cone with $\ell = \frac{c_2}{c_1}$ then using the formulae for $\xi$ and $\zeta$ given in Proposition~\ref{prop:reg}, 
together with~\eqref{eq:explicit_cone_solution} and~\eqref{eq:s*_from_c}
we can express the coefficients of $t^{-1}$ in $x$ and $y$ respectively in terms of $\ell$ as
\[
\xi S^*= \frac{(\ell^2-1)(1-4 \ell^2)}{\lambda \ell^2 (1+2 \ell^2)}, \qquad \zeta S^* = - \frac{(\ell^2-1)(2 \ell^2-5)}{2 \lambda \ell (1+2 \ell^2)}.
\]
In particular, the coefficient of $t^{-1}$ in $x$ vanishes if and only if $\ell =1$ or $\ell = \frac{1}{2}$. The former corresponds to the torsion-free cone, while 
the latter corresponds to the asymptotic cone of the explicit AC shrinker. (This is consistent with the fact that the explicit AC shrinker satisfies $x \equiv \frac{1}{2}t$). 
Note also that unless $\ell=1$ then the coefficient of $t^{-1}$ in at least one of $x$ or $y$ is nonzero. Thus the asymptotic rate is \emph{precisely} $-2$ when $\ell \not= 1$. 

\begin{samepage} 
Similarly, the coefficient of $t^{-2}$ in $\frac{x}{y}$ can be written as
\[
-\frac{c_1}{c_2} \frac{S^*}{2c_1 c_2^2} = -\frac{9(\ell^2-1)}{\lambda \ell (1+2\ell^2)}
\]
which is nonzero except when $\ell =1$.
\end{samepage} 
\end{remark}

\begin{proof}
Recall that
\[ \frac{d}{dt}\frac{x}{y} = \frac{S}{y^3} . \]
Our hypotheses and the convergence $S \to S^*$ assured by
Proposition \ref{prop:S:limit} imply that
\[\frac{S}{y^3} =
\frac{S^*}{c_2^3}t^{-3} + o(t^{-3}) , \]
so applying Lemma~\ref{lem:mu_ode3} with $\nu = 3$ gives \eqref{eq:xy_coeff}.
Recall from \eqref{eq:xydot:S} that $x'$ and $y'$ satisfy
\[ x' = \frac{1}{3} + \frac{x^2}{6y^2} + \frac{2S}{3y^2} ,
\qquad y' = \frac{y}{3x} + \frac{x}{6y} - \frac{S}{3xy} . \]
Hence using~\eqref{eq:xy_coeff} we find
\begin{align*}
x' 
&= \frac{1}{3} + \frac{1}{6}\left(
\frac{c_1^2}{c_2^2} - \frac{c_1S^*}{c_2^4} t^{-2}+ o(t^{-2})\right)
+ \frac{2S^*}{3c_2^2}t^{-2} + o(t^{-2}) 
= c_1 - \xi S^* t^{-2} + o(t^{-2}) \end{align*}
where
\[
 \xi = \frac{c_1 - 4c_2^2}{6c_2^4} = \frac{3c_1^2 - 12c_1c_2^2}{18c_1c_2^4}
=\frac{c_1^2 - 4c_2^2}{18c_1c_2^4}.
\]
Similarly for $y'$ we find
\begin{align*}
y' 
&= %
\frac{c_2}{3c_1}\left(1 + \frac{S^*}{2c_1c_2^2}t^{-2}\right)
+ \frac{c_1}{6c_2}
- \frac{S^*}{12c_2^3}t^{-2} - \frac{S^*}{3c_1c_2}t^{-2} + o(t^{-2}) 
= c_2 - \zeta S^* t^{-2} + o(t^{-2})
\end{align*}
where
\[  \zeta = \frac{-2c_2^2 + c_1^2 + 4c_1c_2^2}{12c_1^2c_2^3}
= \frac{5c_1^2 - 2c_2^2}{36c_1^2c_2^3} .\]
Now applying Lemma \ref{lem:mu_ode3} again, this time to $x-c_1 t$ and $y-c_2 t$ with $\nu = 2$, gives
\[x = c_1t + a + \xi S^* t^{-1} + o(t^{-1}), \quad 
y = c_2t + b + \zeta S^* t^{-1} + o(t^{-1})\]
for some constants $a, b$. But now
\[
c_2 x - c_1 y = c_2 y \left(\frac{x}{y} - \frac{c_1}{c_2}\right) = O(t^{-1}),
\]
where the last estimate uses again~~\eqref{eq:xy_coeff}.
This implies that $a = c_1 k$, $b = c_2 k$ for some $k$.
After a suitable translation of $t$ we can thus assume that $a = b = 0$.
\end{proof}

\subsection{Stability of AC expander ends}
\label{ss:AC:expander:end:stability}

Recall that Proposition~\ref{prop:AC:end} implies that  
having an AC end is a stable end behaviour for expanders. 
We can also use the
reasoning from~\S\,\ref{subsec:conv_S} to prove the
stability of AC expander ends and also importantly to prove further that the
asymptotic cone depends continuously on initial conditions.

Our starting point is to use Lemma \ref{lem:Sdot} to deduce that
AC behaviour of an expander end is guaranteed  if, given bounds on $\log \frac{y}{x}$ and on $\sis$,  
the scale $g$ is sufficiently large.

\begin{prop}
\label{prop:stable}
Given $N > 0$ and $\epsilon > 0$, there is $g_0 > 0$ (depending on $\lambda$, $N$ and $\epsilon$) 
such that for any $\lambda$-expander end solution satisfying
\begin{equation}
\label{eq:squeeze}
\left|\log\frac{y(t_1)}{x(t_1)}\right| < N, 
\quad |S(t_1)| < \epsilon g(t_1)^2 \quad \textrm{and\  } g(t_1) > g_0
\end{equation}
for some $t_1$, then
$\left|\log\frac{y(t)}{x(t)} - \log\frac{y(t_1)}{x(t_1)}\right| < \epsilon$ holds
for all $t > t_1$. Moreover,  $\frac{y}{x}$ converges to a limit $\ell$
and hence the expander end is weakly asymptotically conical.
\end{prop}

\begin{proof}
Given $N>0$ and $\epsilon>0$, choose $C>0$ such that $|\frac{\slead(l)}{\lambda}| < \frac{1}{5}C$ holds whenever
$|\log l| < N + \epsilon$.
Then applying Lemma \ref{lem:Sdot}
with  $S^*=0$,  $B=\epsilon$ and $C$ as above and using the fact that $\lambda>0$ 
 yields the existence of 
$g_0>0$ sufficiently large such that
\begin{subequations}
\label{eq:abs:S:'}
\begin{gather}
 \frac{d}{dt}|S| < 0 
\intertext{holds whenever}
 g > g_0, \quad \left|\log \frac{y}{x}\right| < N + \epsilon
\quad\textrm{ and } \quad C < |S| < \epsilon g^2.
 \end{gather}
 \end{subequations} 

By enlarging $g_0$ if necessary, we can also ensure that $C < \epsilon g_0^2$.
Starting from $t_1$ such that \eqref{eq:squeeze} holds, then~\eqref{eq:abs:S:'} implies that
$|S|$ cannot increase past $\epsilon g(t_1)^2$ as long as the bound
$\left|\log\frac{y(t)}{x(t)}\right| < N + \epsilon$ is maintained.
On the other hand, if $|S(t)| < \epsilon g(t_1)^2$ holds  for $t \in [t_1, t_2]$
then integrating
\[ \frac{d}{dt} \log \frac{y}{x} = -\frac{S}{g^3} \]
(from Lemma \ref{lem:log_xy}) and using $g' \geq \half$ gives
\[ \left|\log \frac{y(t_2)}{x(t_2)} - \log \frac{y(t_1)}{x(t_1)}\right|
< 2\epsilon g(t_1)^2 \int_{g(t_1)}^{g(t_2)} \frac{dg}{g^3} < \epsilon. \]
Thus $\left|\log\frac{y(t)}{x(t)} - \log \frac{y(t_1)}{x(t_1)}\right| < \epsilon$,
and hence also $\left|\log \frac{y}{x}\right| < N + \epsilon$, remains true for
at least as long as $\abs{S(t)} < \epsilon g(t_1)^2$.
Thus both estimates remain true throughout the lifetime of the solution,
and therefore Proposition \ref{prop:lifetime} implies that the lifetime is infinite.

With $S$ uniformly bounded, integrating the derivative of
$\log \frac{y}{x}$ now implies that $\tfrac{y}{x}$ converges to a limit $\ell$.
\end{proof}

The most important consequence of the previous result is the following 
stability result for arbitrary AC expander ends. For the $1$-parameter family of smoothly-closing expanders 
this theorem implies that their asymptotic cones depend continuously
on the parameter $\lambda b^2$.

\begin{theorem}
\label{thm:uniform}
Consider a forward-complete $\lambda$-expander end  $(x,y, \tau_2)$ 
such that $\frac{y}{x}$
converges as~$t \to \infty$.  Choose any $t_0$ for which this solution is defined and non-singular,
in the sense that $x(t_0)$ and $y(t_0)$ are positive.

Then any initial condition $\bar p$ sufficiently
close to $p=(x(t_0), y(t_0), \tau_2(t_0))$ and $\bar \lambda$ sufficiently close
to $\lambda$ also gives rise to a forward-complete $\bar \lambda$-expander end
$(\bar x, \bar y, \bar \tau_2)$ with $\frac{\bar y}{\bar x}$ convergent as $t \to \infty$. 
Moreover
\[ \frac{\bar y}{\bar x} \to \frac{y}{x} \]
uniformly on $[t_0, \infty)$
as $(\bar p, \bar\lambda) \to (p, \lambda)$.

In particular, if we let $(c_1, c_2)$ and $(\bar c_1, \bar c_2)$ denote their
respective asymptotic cones, 
then $(\bar c_1, \bar c_2) \to (c_1, c_2)$
as $(\bar p, \bar\lambda) \to (p, \lambda)$.
\end{theorem}

\begin{proof}
Proposition \ref{prop:S:limit} implies that $S \to S^*$ where $\lambda S^* = \slead(\ell)$ and $\ell = \lim_{t \to \infty} \frac{y}{x}$. 
Hence given any $\epsilon > 0$, by picking $N$ and $t_1$ large enough we can ensure that
\eqref{eq:squeeze} holds at $t=t_1$ for the forward-complete $\lambda$-expander $(x,y,\tau_2)$.
Standard ODE theory guarantees
uniform convergence $(\bar x, \bar y, \bar \tau_2) \to (x,y, \tau_2)$
on the compact domain $[t_0, t_1]$
as $(\bar p, \bar\lambda) \to (p, \lambda)$.
In particular, choosing $(\bar p, \bar \lambda)$
close enough to $(p, \lambda)$ ensures that
$|\log \frac{\bar y}{\bar x} - \log\frac{y}{x}| < \epsilon$ on $[t_0, t_1]$,
and that $(\bar x, \bar y, \bar \tau_2)$ also satisfies \eqref{eq:squeeze} at $t=t_1$. 
Then Proposition \ref{prop:stable} implies that
$\left|\log \frac{y}{x}(t) - \log \frac{y(t_1)}{x(t_1)} \right|$ and
$\left|\log \frac{\bar y}{\bar x}(t) - \log \frac{\bar y(t_1)}{\bar x(t_1)}\right|$
are both bounded by $ \epsilon$ for all $t > t_1$. Hence
\[ \left|\log \frac{\bar y}{\bar x}(t) - \log \frac{y}{x}(t)\right| < 3\epsilon \]
for all $t > t_1$. Since $\epsilon$ was arbitrary this proves the claimed 
uniform convergence.
\end{proof}

\begin{remark}
\label{rmk:lending}
In addition to Theorem \ref{thm:uniform} giving a stronger notion of stability
than what followed from the more elementary arguments in
Proposition \ref{prop:AC:end}, another advantage of the reasoning here
is that (along with the proofs of
Proposition \ref{prop:S:limit} and Lemma \ref{lem:weak_AC}) it
lends itself better to generalisation to the
context of \sunitary{3}-invariant solitons on $I \times \sunitary{3}/T^2$.
\end{remark}

\section{Asymptotic cones of smoothly-closing expanders}
\label{s:asymptotic:cones}

So far, we have established that every member of the $1$-parameter family of
smoothly-closing \Sp{2}-invariant expanders $\spfam{\lambda}{b}$ from
Theorem \ref{thm:Sp2:smooth:closure}
is complete and
asymptotic with rate (at least) $-2$ to a unique~\Sp{2}-invariant closed cone $C$ with $\ell = \frac{c_2}{c_1} \ge 1$. 
In this section we will prove that \emph{every} closed cone with $\ell>1$ arises as the
asymptotic limit of a unique (up to scale) smoothly-closing expander.

Recall from Remark \ref{rmk:scaling_sc} that, up to scale,
$\spfam{\lambda}{b}$ depends on the single scale-invariant parameter
$q=\lambda b^2 \in \Rpos$.
Since the asymptotic cone is invariant under rescaling, we can express
its parameter $\ell=\frac{c_2}{c_1} = \displaystyle\lim_{t \to \infty} \tfrac{y}{x}$
as $\lmap(q)$ for a function $\lmap: \Rpos \to \Rpos$ as in
Definition \ref{def:lmap}.

\begin{theorem}
\label{thm:asymptotic_cones} 
The asymptotic limit map $\lmap$ of smoothly-closing expanders
is a continuous increasing bijection between $(0,\infty)$ and $(1,\infty)$. 
\end{theorem}

Continuity of $\lmap$ %
follows immediately from Theorem \ref{thm:uniform}.
The rest of the theorem is then a consequence of the following four claims,
which are proven in the rest of the section:
\begin{enumerate}
\item $\lmap(q) > 1$ for any $q>0$
\item $\lmap$ is a strictly increasing function of $q$
\item $\lmap \to 1$ as $q\to 0$
\item $\lmap \to \infty$ as $q \to \infty$.
\end{enumerate}

\begin{remark*}
(i) may seem redundant given claim~(ii) and that $\ell \geq 1$, since the image of a
strictly increasing function $(0,\infty) \to [1,\infty)$ cannot contain the
boundary point~1. However, our proof of (ii) relies on first establishing (i).
\end{remark*}

Because of the freedom to rescale, we can either try to  understand 
how $\lmap$ depends on~$b$ for fixed~$\lambda$, or how it depends on $\lambda$
for fixed $b$. In the details of the proofs, the latter perspective turns out
to be more convenient.

\subsection{The torsion-free cone is not the asymptotic cone of any smoothly-closing expander}
\label{ss:tf:exclusion}
In this section we prove claim (i) above; this is equivalent to proving that no smoothly-closing
expander can have the unique torsion-free~\Sp{2}-invariant~\gtwo-cone as its
asymptotic cone.
The proof uses Proposition~\ref{prop:reg} which for large $t$ gives the leading-order terms in $x$ and $y$ 
for any non-steady AC soliton end, up to an error which is at worst $o(t^{-1})$. 

\begin{theorem}
\label{lem:expander:D:neq:0}
The torsion-free \Sp{2}-invariant \gtwo-cone does not arise as the asymptotic cone of 
any smoothly-closing  expander.
Equivalently,  $\ell= \lim_{t \to \infty} \frac{y}{x}>1$ for any smoothly-closing expander.
\end{theorem}
\begin{proof}
Suppose for a contradiction that 
some smoothly-closing expander has asymptotic cone the torsion-free \Sp{2}-invariant \gtwo-cone. Then %
$\ell=1$ and hence by Proposition~\ref{prop:S:limit} it has $S^*=0$. Proposition~\ref{prop:reg} then implies that, up to a translation in $t$, for large $t$ we have
\[
x = \tfrac{1}{2}t + o(t^{-1}), \quad y=\tfrac{1}{2} t+ o(t^{-1}).
\]
Hence $y^2 - x^2 \to 0$ as $t \to \infty$. 
Now recall from %
Theorem \ref{thm:sc:expanders:complete}(i) that $\tau_2$, $S$ and
$\tilde \tau_1$ (or equivalently~$R_1$) are positive for all $t>0$ on any
smoothly-closing expander.  Hence by~\eqref{eq:ODE:tau2'}, $\tau_2$ is a
positive increasing function of $t$.
In particular, there exists $B>0$ (depending on $b$ and $\lambda$) so that $ x \tau_2 > B\,t$ holds for $t$ sufficiently large. 
But since $S = (y^2-x^2) - \frac{3}{2}x \tau_2$ and $y^2-x^2 \to 0$, the (at least) linear growth of $x \tau_2$ 
implies that eventually $S$ becomes negative, contradicting the positivity of $S$.
\end{proof}

\begin{corollary}
\label{cor:S*:positive}
Any smoothly-closing expander has $S^*: = \lim_{t \to \infty}{S(t)} >0$. 
\end{corollary}
\begin{proof}
This is immediate from the previous result 
and the formula for $S^*$ given in Proposition~\ref{prop:S:limit}.
\end{proof}

We also have the following result about the decay rate of complete AC~\Sp{2}-invariant solitons.
\begin{prop}
If a complete~\Sp{2}-invariant soliton is AC with rate $\mu<-2$ then it must in fact be static and 
therefore is the Bryant--Salamon torsion-free~\gtstr~which has $\mu=-4$. 
\end{prop}
\begin{proof}
We already observed in Remark \ref{rmk:scaling_sc} that a smoothly-closing~\Sp{2}-invariant steady soliton must be static.

In the non-steady cases, we noted in Remark \ref{rmk:precise_rate} %
that the solution is AC with rate precisely $-2$ unless 
$\ell=1$.
For smoothly-closing expanders the latter is ruled out by Theorem~\ref{lem:expander:D:neq:0}. 

For shrinkers, we have the (much more) general result that the only complete AC shrinker  asymptotic to a torsion-free cone is the Gaussian shrinker on $\R^7$. Indeed, 
the uniqueness of AC shrinker ends proven in~\cite{HKP} (see Theorem \ref{thm:HKP}) shows that the only (gradient) AC shrinkers asymptotic to any torsion-free cone 
are the Gaussian shrinkers associated with that cone and hence are incomplete unless the cone is flat $\R^7$. 
\end{proof}

\subsection{Smoothly-closing expanders are uniquely determined by their asymptotic cone}
\label{ss:ac:expander:uniqueness}

In this section we prove a ``$g$-for-$g$'' comparison result for local
expanders and use it to prove that $\lmap$ is a nondecreasing function.
Together with our understanding
(from \S\ref{ss:rate-2:AC}) of the second term in the large-$t$ expansions of $x$
and $y$, we can use this comparison result to prove that, in fact,  $\lmap$ is strictly
increasing. %
It is convenient in the argument to interpret this as $\lmap(q)=\lmap(\lambda b^2)$ being an
increasing function of $\lambda$ when the parameter $b$ that determines the volume of the singular orbit is kept fixed.

Recall the definition of $\tilde \tau_2$ from Section~\ref{ss:tilde_tau},
that it satisfies $\tilde \tau_2'= \lambda y^2$, is increasing
and is positive on any smoothly-closing expander. 
\begin{prop}
\label{prop:comp_g_1}
Consider two local expander solutions labelled $\pm$, possibly with different values 
of the dilation constant $\lambda_\pm$, and parametrise them both by $g$.
Suppose that $\lambda_+ \geq \lambda_-$, and that for some $g = g_0$ we have
\[ y_+(g) > y_-(g), \quad
\frac{\tilde \tau_{2+}(g)}{\lambda_+} > \frac{\tilde \tau_{2-}(g)}{\lambda_-}. \]
Then these inequalities hold for all $g > g_0$; indeed, under these conditions
\[
\frac{\tilde \tau_{2+}(g)}{\lambda_+} - \frac{\tilde \tau_{2-}(g)}{\lambda_-}
\]
is a strictly increasing function of $g$.
\end{prop}

\begin{proof}
Note the obvious fact that given $g$ and $y$, then $x$ is determined via $x=g^3/y^2$. In particular,  
$y_+(g)>y_-(g)$ is equivalent to $x_+(g)<x_-(g)$ and therefore $y_+(g)>y_-(g)$ implies $(y_+/x_+)(g) > (y_-/x_-)(g)$.

The fact that (recall also that $g'$ satisfies~\eqref{eq:g:dot}) 
\[ \frac{1}{\lambda}\frac{d\tilde\tau_2}{dg} = \frac{\tilde \tau_2'}{\lambda g'}= 6g^2 \frac{y^2}{x^2 + 2y^2} \]
is an increasing function of $y$ for fixed $g$ implies that
\begin{equation}
\label{eq:tautilde:diff}
\frac{\tilde \tau_{2+}(g)}{\lambda_+} - \frac{\tilde \tau_{2-}(g)}{\lambda_-}  \quad \text{is increasing whenever\ } y_+(g) > y_-(g)~\text{holds.} 
\end{equation}
Hence if the claim were false, then at the smallest $g_1 > g_0$ where the
condition fails, the following must hold
\[
 x_+(g_1) = x_-(g_1)= x(g_1), \quad y_+(g_1) = y_-(g_1)=y(g_1), \quad \frac{\tilde \tau_{2+}(g_1)}{\lambda_+} >
\frac{\tilde \tau_{2-}(g_1)}{\lambda_-}, 
\quad 
\frac{d(y^2_+-y^2_-)}{dg} (g_1)\le 0.
\]
We can rewrite the second equation in~\eqref{eq:tilde:tau1:tau2} as 
\[
\frac{3\tau_2}{\lambda} = \frac{x^2+2y^2}{x^2} \frac{\tilde \tau_2}{\lambda} - 2  \frac{y^2g^3}{x^2}.
\]
It follows from this equation and the conditions that hold at $g=g_1$ 
that 
\[
\frac{\tau_{2+}(g_1)}{\lambda_+} > \frac{\tau_{2-}(g_1)}{\lambda_-}.
\]
Since we assumed $\lambda_+ \geq \lambda_- >0$ this also implies that
$\tau_{2+}(g_1) > \tau_{2-}(g_1)$.
Hence because
\[
\frac{d(y^2)}{dg} = \frac{(y^2)'}{g'} = 6g^2\frac{x+\tau_2}{x^2 + 2y^2}
\]
we find that
\[
\frac{d(y^2_+-y^2_-)}{dg}(g_1) = \frac{6 g_1^2  \left( \tau_{2+}(g_1) - \tau_{2-}(g_1) \right)}{(x^2 + 2y^2)(g_1)}>0
\]
which is a contradiction. 
\end{proof}

\begin{remark}
\label{rmk:comp_lambdas}
In the case of a pair of smoothly-closing expanders $\spfam{\lambda_\pm}{b}$,
with respective dilation constants $\lambda_+> \lambda_-$ but satisfying the
same initial condition $y_\pm(0) = b$, we claim that
the inequalities in Proposition \ref{prop:comp_g_1} hold for small $g>0$
and hence by the previous comparison result they persist for all $g$.
Indeed, according to \eqref{eq:sp2:expansion},  the first few terms for $x$ and $y$ 
are
\[ x = t - \frac{t^3}{54b^2} \left(4\lambda b^2+9\right) + \cdots 
, \qquad
y = b + \frac{t^2}{36b} \left( 2\lambda b^2+9 \right)  + \cdots \] 
Then $G: = g^3 =xy^2 =  b^2 t + O(t^3)$ implies $t = b^{-2}G + O(G^3)$, so that
\[
y = b + \frac{G^2}{36b^5}(2\lambda b^2 + 9) + O(G^4).
\]
Hence $y$ is an increasing function of $\lambda$ for fixed $G$ (or equivalently, for fixed $g$) and $b$.

Since $\tilde \tau_{2\pm}(0)=0$, then having established that $y_+(g) > y_-(g)$ holds for $g>0$ sufficiently small,
it follows from~\eqref{eq:tautilde:diff} that $\frac{\tilde\tau_{2+}(g)}{\lambda_+} - \frac{\tilde\tau_{2-}(g)}{\lambda_-}$
is increasing (and therefore positive) for $g>0$ sufficiently small. 
\end{remark}

Alternatively, we could argue that two smoothly-closing expanders with the same
$\lambda$ but initial conditions $b_+ > b_-$ satisfy the inequalities for
small $g$. Either way, the conclusion that $y_+(g) > y_-(g)$ for all $g$
immediately implies that the limits $\ell_\pm$ of $\frac{y_\pm}{x_\pm}$
satisfy $\ell_+ \geq \ell_-$. However, for proving the strict inequality
$\ell_+ > \ell_-$, it turns out to be more useful to compare solutions with
fixed $b$ and $\lambda_+ > \lambda_-$, using the following easy consequence
of Proposition \ref{prop:reg}.

\begin{lemma}
\label{lem:comp_ends}
Consider two expanders $(x_\pm, y_\pm)$ with dilation constants
$\lambda_{\pm}$,
asymptotic to the same closed cone $(c_1, c_2)$, \ie
\[ \frac{x_\pm}{y_\pm} \to \frac{c_1}{c_2} . \]
If $c_2>c_1$ and $\lambda_+ > \lambda_-$ then
\[ \frac{x_+(t)}{y_+(t)} > \frac{x_-(t)}{y_-(t)} \]
holds for $t$ sufficiently large.  The same inequality also holds if we compare $g$-for-$g$
instead; equivalently
\[ y_+(g) < y_-(g) \]
holds for all $g$ sufficiently large. 
\end{lemma}

\begin{proof}
The assumption that $c_2>c_1$ implies, thanks to~\eqref{eq:s*_from_c}, that $S^*$ is a positive decreasing function of $\lambda>0$. 
The $t$-for-$t$ comparison now follows from \eqref{eq:xy_coeff}. 
The $g$-for-$g$ comparison results follow too since $t = (c_1c_2^2)^{-1/3}g + O(g^{-1})$.
\end{proof}

\begin{remark*}
Actually we can prove the $g$-for-$g$~comparison result in
Lemma \ref{lem:comp_ends} using just a fragment of the argument from
Proposition \ref{prop:reg}. 
To see this observe---thanks to~\eqref{eq:s*_from_c}---that
\[ g^3 \frac{d}{dg} \log \frac{x}{y} = \frac{6Sg^2}{x^2+2y^2}
\to \frac{3(c_2^2-c_1^2)}{\lambda (c_1c_2^2)^{1/3}} >0 \]
as $g \to \infty$. 
Then $\lambda_+ > \lambda_- > 0$ implies that 
\[ \frac{d}{dg} \log \frac{x_+}{y_+} < \frac{d}{dg} \log \frac{x_-}{y_-} \]
holds for sufficiently large $g$.
Since the limits of $\log \frac{x_\pm}{y_\pm}$ as $g \to \infty$ are equal
it follows that
\[ \frac{x_+(g)}{y_+(g)} > \frac{x_-(g)}{y_-(g)}  \]
for $g$ sufficiently large. 

However, the refined asymptotics for $x$ and $y$ established in Proposition~\ref{prop:reg} 
were needed in the proof of Theorem~\ref{lem:expander:D:neq:0}.
\end{remark*}

\begin{theorem}
\label{thm:D*:monotone}
$\lmap$ is a strictly increasing function of $q$. In particular,
two \Sp{2}-invariant smoothly-closing expanders asymptotic to the same cone are rescalings of each other. 
\end{theorem}

\begin{proof}
Consider two smoothly-closing expanders %
with scale-invariant parameters $q_+>q_-$. 
By the scaling invariance 
there is no loss of generality in assuming $q_{\pm} = \lambda_\pm b^2_\pm$
satisfy $b_+ = b_-$ and $\lambda_+ > \lambda_-$.
Proposition~\ref{prop:comp_g_1} and Remark~\ref{rmk:comp_lambdas} together imply that for any $g>0$ we have the $g$-for-$g$ comparison 
result that $y_+(g) > y_-(g)$ for all $g>0$, which immediately implies
$\lmap(q_+) \geq \lmap (q_-)$.

Now suppose for a contradiction that $\lmap(q_-) = \lmap(q_+)$, 
\ie that both expanders are asymptotic to the same closed cone $(c_1,c_2)$.
By Theorem~\ref{lem:expander:D:neq:0} we know that $c_2>c_1$ and hence that $S^*>0$.
Now because both expanders are asymptotic to the same cone with $c_2>c_1$, 
Lemma~\ref{lem:comp_ends} applies and gives the comparison
$y_+(g) < y_-(g)$ for $g$ sufficiently large, which yields a contradiction.
\end{proof}

\subsection{Asymptotic behaviour of \texorpdflmap{} as \texorpdfstring{$q \searrow 0$}{q -> 0}}
\label{ss:lmap:qto0}
To understand the behaviour of \texorpdflmap{} 
 as $q=\lambda b^2 \to 0$ we will
exploit the fact (from Theorem \ref{thm:Sp2:smooth:closure} and the remark following it) that a smoothly-closing \Sp{2}-invariant steady soliton is
actually a trivial soliton, \ie it is torsion free with vanishing soliton vector field $X=0$.
\begin{lemma}
\label{lem:sid*:bto0}
$\lmap(q) \to 1$ as $q \to 0$.
\end{lemma}

\begin{proof}
Because of the scaling invariance it suffices to consider the
asymptotic behaviour of smoothly-closing expanders %
with fixed $b$ and $\lambda \to 0$. Suppose then for a contradiction that 
$\lmap$ is bounded away from $1$. Then, because on every smoothly-closing expander (by the positivity of $S$
proven in 
Theorem \ref{thm:sc:expanders:complete}.i) 
$\frac{y}{x}$ is a decreasing function of $t$, there would exist 
a uniform $\ell_0 > 1$ such that
\[
\tag{*}
\label{eq:*}
\frac{y}{x} \geq \ell_0 \quad \text{for each smoothly-closing expander and all\ } t > 0.
\]

Since the expander ODEs depend continuously (in fact, analytically) on $\lambda$ and the initial conditions are fixed, then as
$\lambda \to 0$ the expanders converge uniformly on compact subsets
to the unique \Sp{2}-invariant \emph{steady} soliton with the same initial
condition.
However, by the second part of Theorem~\ref{thm:Sp2:smooth:closure},
any smoothly-closing $\tu{Sp}_2$-invariant steady soliton is actually a trivial soliton, \ie it is torsion free with vector field $X=0$.
In particular, its asymptotic cone is the torsion-free
\Sp{2}-invariant cone which has $\ell=1$.
Therefore, for some $t_1$ sufficiently large the limiting steady soliton satisfies $\frac{y}{x}(t_1) < \ell_0$. 
Thus because the smoothly-closing expanders converge
to this limit steady soliton uniformly on $[0,t_1]$, we contradict~\eqref{eq:*}.
\end{proof}

\begin{remark*}
One might be tempted to think that one could apply
Theorem \ref{thm:uniform} to deduce this result, 
but there we assumed the limit to be an \emph{expander}, not a steady soliton.
However, the monotonicity of $\frac{y}{x}$ that holds for smoothly-closing expanders makes it easier to
prove uniform convergence in this special case.
\end{remark*}

\subsection{Asymptotic behaviour of \texorpdflmap{} as \texorpdfstring{$q \nearrow \infty$}{q-> oo}}
\label{ss:large:b:asymptotics}

To understand the behaviour of $\lmap$ as $q  \to \infty$ the key ingredient
is the monotone quantity $M/g^3$ introduced in Lemma~\ref{lem:monotone}.
This quantity allows us to relate the behaviour of a smoothly-closing expander at the singular orbit at $t=0$
to its asymptotic behaviour as $t \to \infty$. 

\begin{prop}
\label{prop:b_oo}
For $\lambda > 0$,
the asymptotic limit map $\lmap$ satisfies
\begin{equation}
\label{eq:ell_bound}
\lmap^2 > \tfrac13 q - \half.
\end{equation}
In particular, with $\lambda$ fixed, $\lmap \to \infty$ as $b \to \infty$.
\end{prop}

\begin{proof}
Recall from Lemma~\ref{lem:monotone} that the quantity $M/g^3$, where $M= 3x + \tilde \tau_1$, is strictly decreasing in~$t$. 
Also recall that $\tilde \tau_1$ (and hence $M$) is positive on any smoothly-closing expander
for all $t>0$. 

On $\spfam{\lambda}{b}$,
$Mg^{-3}$ takes the value $3y(0)^{-2} = 3b^{-2}$ at $t = 0$. %
Hence on this expander the inequality
\[ \frac{\tilde\tau_1}{xy^2} < 3 \left(\frac{1}{b^2} - \frac{1}{y^2} \right) \quad \text{holds for all \ } t>0. 
\]
Now split $\tilde \tau_1/xy^2$ into two terms
\[ \frac{\tilde\tau_1}{xy^2} = \lambda\frac{2x^2}{2y^2+x^2} -
\frac{6x\tau_2}{(2y^2+x^2)y^2}. \]
The linear growth of $x$, $y$ and $\tau_2$ implies that the second term is $O(t^{-2})$, 
and hence
\[
\lim_{t \to \infty} \frac{\tilde\tau_1}{xy^2} = \frac{2 \lambda}{ 2\ell ^2 +1},
\]  
where $\ell = \lim_{t \to \infty}  \frac{y}{x}$.
Taking the limit of the previous inequality as $t \to \infty$ implies that 
$\ell$ satisfies 
\[ \frac{2\lambda}{2\ell^2 + 1} < \frac{3}{b^2}, \]
which is equivalent to \eqref{eq:ell_bound}.
\end{proof}

\begin{remark}
\label{rmk:ell_sharp}
There is good reason to believe that the bound \eqref{eq:ell_bound} is sharp
asymptotically, in the sense that
\begin{equation}
\label{eq:ell_sharp}
\frac{\lmap^2}{\lambda b^2} \to \frac{1}{3} \ \ \text{as $\lambda b^2 \to \infty$.}
\end{equation}
Given a solution $(x,y,\tau_2)$ to the~\Sp{2}-invariant soliton
equations \eqref{eq:ODE:tau2'} and any $\epsilon \ge 0$, setting
\[ \hat x_\epsilon := x,  \quad \hat y_\epsilon = \epsilon y,
\quad \hat \tau_{2\epsilon} = \epsilon^2 \tau_2\]
gives a solution to the $\epsilon$-dependent ODE system
\begin{equation}
\label{eq:Sp2:ODE:eps}
(\hat{x}^2)' = 2\hat{x} - 2\frac{\hat{x}^2\hat\tau_2}{\hat y^2} - \epsilon^2\frac{\hat{x}^3}{\hat y^2} ,
\qquad  (\hat y^2)' =  \hat\tau_2 + \epsilon^2 \hat{x}, \qquad
\hat{\tau}_2' = \frac{4\wh{R}_1 \wh{S}}{3\hat{x}(2\hat y^2 + \epsilon^2 \hat{x}^2)},
\end{equation}
where now
\[\wh{R}_{1}:=\lambda \hat{x}\hat y^2-3\hat\tau_2,  \quad \ \wh S:=\hat y^2-\tfrac{3}{2}\hat{x}\hat\tau_2 -\epsilon^2 \hat{x}^2.
\]
For $\epsilon >0$ the ODE system~\eqref{eq:Sp2:ODE:eps} is equivalent to the soliton ODE system~\eqref{eq:ODE:tau2'}.
However, the rescaled system~\eqref{eq:Sp2:ODE:eps} also makes sense when
$\epsilon = 0$ and this limit ODE system (see also~\eqref{eq:Sp2:ODE:limit})
is no longer equivalent to~\eqref{eq:ODE:tau2'}.

The arguments from Section~\ref{s:Sp2:expanders:completeness} can be
adapted to show that there is a unique solution $(\hat x, \hat y, \hat \tau_2)$ %
to this limit ODE system for a given $\lambda>0$ and initial condition $\hat y(0)
= b$, and moreover that $\frac{\hat y}{\hat x}$ has a finite
limit $\wh \ell$ as $t \to \infty$. Now the monotonic quantity
$Mg^{-3}$ used in the proof of Proposition \ref{prop:b_oo} has an analogue
$\frac{\wh M}{\hat x \hat y^2}$ in the limit system (where
$\wh M := 3 \hat x + \frac{\hat x^2}{\hat y^2}\wh R_1$) and this quantity is actually \emph{conserved}, showing
that $\wh \ell$ is exactly $\sqrt{\frac{\lambda b^2}{3}}$.
In fact, we first discovered the conserved quantity in this limit system and this motivated us 
to seek a related monotone quantity in the full~\Sp{2}-invariant soliton ODE system.

Now, similarly to the proof of Theorem \ref{lem:expander:D:neq:0}, it is clear
that for $\lambda$ and $b$ fixed, the unique solution
$(\hat x_\epsilon, \hat y_\epsilon, \hat \tau_{2\epsilon})$
to~\eqref{eq:Sp2:ODE:eps}
with $\hat y_\epsilon(0) = b$ converges uniformly on compact subsets to the
solution $(\hat x, \hat y, \hat \tau_{2})$ of the limit system. However, it is
more complicated here to argue that one gets uniform convergence on all of
$(0,\infty)$, which is what would be needed to deduce \eqref{eq:ell_sharp}.
\end{remark}

\section{Soliton end behaviours}
\label{sec:ends}
In the next two sections we describe all possible end behaviours of non-steady~\Sp{2}-invariant solitons and
explain what implications this has for the dimensions of the space of complete~\Sp{2}-invariant solitons on $\Lambda^2_- \Sph^4$. 
For both expanders and shrinkers we sketch how to construct three different classes of end behaviours by recasting these 
problems as certain singular initial value problems. Two types of end behaviour are common to both expanders and shrinkers;
the other end behaviour is different for shrinkers and for expanders. 

The main result in this section, Theorem~\ref{thm:tri:expanders}, 
proves that for~\Sp{2}-invariant expanders the three claimed types of end behaviour exhaust all possible end behaviours and
thus establishes Theorem~\ref{mthm:both_tri} for ends of~\Sp{2}-invariant expanders. 
In Section~\ref{SS:su3} we will indicate what similarities and differences we expect for the possible end behaviours of~\sunitary{3}-invariant 
expanders. 
(The possible end behaviours of
\sunitary{3}-invariant steady solitons were already completely described in
\cite[Theorem~F]{Haskins:Nordstrom:g2soliton1}.)

The proof of Theorem~\ref{mthm:both_tri} for~\Sp{2}-invariant shrinker ends is substantially more involved than our proof for expanders. 
However, in Section~\ref{ss:shrinker_fi_ends} we show that methods similar to those used in the expander case suffice to understand forward-incomplete shrinker ends. This will leave us free in Section~\ref{sec:shrinker-tri} to concentrate 
on the new ideas, including another reformulation of the shrinker ODEs,
needed in the case of forward-complete shrinker ends to complete the proof of Theorem~\ref{mthm:both_tri} for shrinkers.

\subsection{Rigidity and instability of AC shrinkers versus stability of AC expanders}
\label{subsec:rigid_vs_instab}

Let us begin by emphasising an important difference between the end behaviours of AC expanders
and AC shrinkers, and the consequences that has for the expected dimensions of
complete AC solitons.

We have proved that AC ends of \Sp{2}-invariant expanders are stable.
Indeed, we have provided two different arguments for this:
Proposition \ref{prop:AC:end}, which characterised AC behaviour in terms of the increasing function
$\tilde \tau_1$ being eventually positive, and
Proposition \ref{prop:stable}, which also proved continuous dependence of the
asymptotic cone on the initial condition.

This AC expander end stability on its own is enough to argue that the subset of
complete AC solutions in the 1-parameter family of smoothly-closing local
expanders from Theorem \ref{thm:Sp2:smooth:closure} is open, and thus (if non-empty)
1-dimensional; but of course we have established in Theorem
\ref{thm:sc:expanders:complete} that \emph{all} smoothly-closing expanders are AC.
However, as we will explain in Section~\ref{SS:su3},  this sort of stability argument for AC expanders becomes more important in the case of \sunitary{3}-invariant expanders on~$\Lambda^2_- \CP^2$, because in that setting we expect that not all the smoothly-closing expanders are forward complete.

The stability of AC expander ends contrasts with the main result of
Haskins--Khan--Payne~\cite{HKP}.

\begin{theorem}
\label{thm:HKP}
There is at most one AC (gradient) shrinker end asymptotic to any
closed~\gtwo-cone.
Moreover, every isometry of the asymptotic cone
of an AC shrinker end is inherited by the end itself. 
\end{theorem}

The \emph{existence} of an AC shrinker end asymptotic to an arbitrary
closed~\gtwo-cone is not addressed by this result, but it does imply that any
AC shrinker end asymptotic to an~\Sp{2}-invariant closed cone must itself
be~\Sp{2}-invariant. 
In the next subsection we explain how one can use ODE methods
(more specifically, solving a certain irregular singular initial value
problem) to construct a (unique)
\Sp{2}-invariant AC shrinker end asymptotic to any~\Sp{2}-invariant closed cone.
Since the space of such cones is 1-dimensional,
parametrised by $\ell = \frac{c_2}{c_1}>0$, so is the space of~\Sp{2}-invariant shrinkers with AC ends.

Since the space of all local~\Sp{2}-invariant shrinkers is $2$-dimensional
(corresponding to flow lines in the 3-dimensional phase space consisting of triples $x$, $y$ and $\tau_2$), the condition
that a local shrinker has an AC end is therefore a codimension 1 condition.
In particular, in contrast to the expander case, the AC end condition is nongeneric among all local
\Sp{2}-invariant shrinkers. Furthermore,  a naive parameter count suggests that complete AC
shrinkers should form a $0$-dimensional subset of the $2$-dimensional space of local
shrinkers: the space of smoothly-closing solutions and the space of AC shrinker ends
both have codimension 1. 
For \sunitary{3}-invariant AC shrinkers we will give a similar heuristic in Section~\ref{SS:su3}.

 Numerical simulations, in fact, strongly suggest the following conjecture (that we have not yet been able to prove). 
 \begin{conjecture}
 \label{c:shrinker:unicity}
 The explicit AC shrinker of Example \ref{ex:explicit_shrinker} is the unique complete AC~\Sp{2}-invariant shrinker. 
 \end{conjecture}

\subsection{Constructions of \texorpdfSp{2}-invariant soliton ends: forward-incomplete and AC ends}
\label{subsec:end_sivp}
In the next two subsections 
we consider the problem of constructing solitons with various types of `prescribed end behaviour'.
This subsection deals with forward-incomplete and AC ends; the next one deals with forward-complete non-AC ends. 

In all cases, to construct ends with prescribed asymptotic behaviour we proceed by making a variable change that turns the problem into a
singular initial value problem (SIVP). Depending on what type of prescribed end behaviour 
we are seeking to construct the resulting singular initial value problem can be either regular or irregular. 
A common feature is that the value of the dilation constant tends not to play a key role when the end behaviour 
gives rise to a regular SIVP, but it does when the end behaviour results in an irregular problem. 

Given the
discussion in the previous subsection we would certainly like to show the existence of AC expanders and AC shrinkers asymptotic to
a given closed cone, and to understand from this viewpoint why AC ends of expanders and shrinkers behave differently. 
However, since the resulting singular initial value problem is an irregular one, it is technically easier to first
consider the construction of forward-incomplete solutions (because that leads to a regular singular initial value problem). 

In \cite[\S 6.5]{Haskins:Nordstrom:g2soliton1}, for any value of the dilation constant $\lambda$ we found \Sp{2}-invariant
solitons with finite extinction time $t_*$ as follows.
We change variables by setting $v := \sqrt{t_* - t}$, and make an ansatz of
the form
\begin{equation}
\label{eq:stable:singularity}
x = av\sqrt{1+v\bar x}, \quad y = b \sqrt{2\frac{1+v\bar y}{v}}, \quad
\tau_2 = b^2 \left(\frac{1+v\bar \tau_2 }{v^3} \right)
\end{equation}
for constants $a,b > 0$ and functions $\bar x, \bar y, \bar \tau_2$ of $v$.
The soliton ODEs \eqref{eq:ODE:tau2'} then become a regular singular initial
value problem for $\bar x, \bar y$ and $\bar \tau_2$ with singularity at $v=0$. 
It is explained in\mbox{\cite[Theorem~6.23]{Haskins:Nordstrom:g2soliton1}} that
for each pair of parameters $(a,b)$ there is a unique power series solution,
yielding a unique smooth solution, which is in fact real analytic.
(This is analogous to Theorem \ref{thm:Sp2:smooth:closure} for smoothly-closing
solutions; another similarity between the two results is
that both are independent of the value of the dilation constant $\lambda$.)
The flow lines of the resulting 2-parameter family of solutions thus fill up an
open region of the 3-dimensional phase space, and hence for any value of $\lambda$ this kind of
forward incompleteness is stable (\cf Remark \ref{rmk:incomplete_open}).
(In fact,~\cite[Theorem 6.23]{Haskins:Nordstrom:g2soliton1} shows also the existence of a stable class of forward-incomplete~\sunitary{3}-invariant solitons.)

Note that as $t \to t_*$, the ansatz implies that $g^3= xy^2 \to 2ab^2$ and
$\tilde \tau_1 \to -\frac{3a^2}{2}$ by \eqref{eq:tilde:tau1:tau2}.
The latter is consistent with Proposition \ref{prop:AC:end}: any forward-incomplete 
expander must certainly have $\tilde \tau_1<0$ throughout its lifetime.

\medskip
Next we consider the problem of finding an \Sp{2}-invariant soliton
end asymptotic to a closed cone specified by a solution $(c_1,c_2)$ to the
closed cone equation \eqref{eq:closed_cone}. 
Recall from Proposition \ref{prop:S:limit} that the quantity $S$ plays a
distinguished role for such weakly AC solutions, and, in particular, that if $\lambda \neq 0$ then $S$ converges to
$S^* := \frac{3}{\lambda}(c_2^2 - c_1^2)$ as $t \to \infty$. It is therefore
convenient to consider the soliton system as an ODE for $x,y$ and $S$ rather
than $x,y$ and $\tau_2$. Meanwhile, in view of Proposition~\ref{prop:reg}
it is natural to change variable to $u := t^{-2}$ and take as an ansatz
\[ x = c_1 t\,(1 + u X(u)), \quad y = c_2 t\,(1 + uY(u)) . \]
The soliton system \eqref{eq:ODE:tau2'} then reduces to the following system of ODEs 
\begin{equation}
\label{eq:irreg_ode}
\begin{aligned}
\frac{dX}{du} &= \frac{(-3c_1^2 - 2c_2^2)X + 2c_1^2 Y - 4S}{(12c_1c_2^2)\, u} + R_X \\
\frac{dY}{du} &= \frac{(-c_1^2+2c_2^2)X -4c_2^2Y + 2S}{(12c_1c_2^2)\, u} + R_Y \\
\frac{dS}{du} &= \frac{\lambda}{6}\frac{S - S^*}{u^2} + \frac{R_S}{u}
\end{aligned}
\end{equation}
where $R_X,R_Y$ and $R_S$ are real analytic functions of $u, X, Y$ and $S$.
The presence of the singular term 
\[\frac{\lambda}{6u^2}(S-S^*)
\] in the final equation makes this
an \emph{irregular} singular initial value problem whenever $\lambda \neq 0$
(in the steady case $\lambda = 0$ it becomes a regular singular initial value
problem, but AC steady ends are easier to analyse by exploiting
additional scale-invariance properties of the steady soliton ODEs as
in~\cite[\S 7.3]{Haskins:Nordstrom:g2soliton1}).

For any fixed closed cone $c_1, c_2$ and $\lambda \not=0$ one can prove with standard methods that there exists
a unique formal power series solution to~\eqref{eq:irreg_ode} by recursively solving for
the coefficients, the constant terms of which are given by~\eqref{eq:x:y:larget:exp}.
By a general result of Malgrange \cite[Th\'eor\`eme 7.1]{malgrange1974} there always exists some smooth solution to the ODE with that Taylor expansion at $u = 0$.
Hence for any $\lambda \not=0$ there is at least one $\lambda$-soliton end
asymptotic to any given closed cone. 
However, the question of uniqueness is more subtle.
As discussed previously, AC expander and shrinker ends behave very differently:
by the stability of AC expander ends and dimension counting it is clear that we cannot have uniqueness 
of AC expander ends for every closed cone; 
on the other hand, Theorem~\ref{thm:HKP} implies the uniqueness of a (gradient) shrinker end asymptotic to any closed cone.
The question is then how to understand the uniqueness or (the extent of) failure of uniqueness of non-steady AC ends 
from the viewpoint of the theory of irregular singular initial value problems. 

The dependence on the sign of $\lambda$ is evident already when considering
the prototypical scalar singular IVP 
\[
\frac{dz}{du} = \frac{\lambda z}{u^2}:
\]
for $\lambda < 0$ the unique solution which remains bounded at $u=0$ is
$z \equiv 0$, while for $\lambda > 0$ there are solutions proportional to
$z(u) = \exp(-\lambda u^{-1})$ (demonstrating also that one should not expect
solutions to be real analytic around $u=0$).
Analogously, for the AC end soliton ODEs~\eqref{eq:irreg_ode} we can prove that
\begin{itemize}[left=0em]
\item
for $\lambda < 0$ there is a unique smooth solution for each closed cone $(c_1, c_2)$
\item for $\lambda > 0$ there is
a 1-parameter family of solutions for each closed cone $(c_1, c_2)$, all with the same
power series expansion at $u = 0$; the difference between any two solutions with the same asymptotic cone has order a polynomial times
$\exp(-\frac{\lambda}{6}u^{-1}) = \exp(-\frac{\lambda}{6}t^2)$.
\end{itemize}
We intend to present the details of this analysis elsewhere, as it is not crucial to the main results of this paper
and it is a nontrivial investment to develop the requisite theory 
of irregular singular initial value problems for nonlinear ODE systems.
(Some related theory is developed by Stein--Turner \cite[\S3.6]{stein24}, where it is used as a tool
to understand the asymptotic behaviour of~\gtwo-instantons on ALC~\gtwo-holonomy spaces. 
Their method in this special case was inspired by
discussions with the current authors about our approach to the theory in a more general context.) 
The same theory can also be applied to prove the existence of other types 
of non-AC forward complete non-steady ends as described in the next subsection. 
In the present conical end setting the qualitative conclusions are in any case as described in the previous subsection:
after fixing a scale, \eg by fixing $\lambda= \pm 1$, 
there is a unique
\Sp{2}-invariant AC shrinker end asymptotic to each closed~\Sp{2}-invariant cone, while instead
there is a 1-parameter family of AC expander ends sharing the same asymptotic cone. 
Hence the flow lines of \Sp{2}-invariant AC shrinker ends form a codimension $1$ subset
of the 3-dimensional phase space, while \Sp{2}-invariant AC expanders form an open subset of it. 

\subsection{Forward-complete non-AC soliton ends}
\label{ss:fc_non_ac}
In this subsection we describe how to construct certain forward-complete but non-AC expander and shrinker ends.  

We start with the expander case. 
For \Sp{2}-invariant expanders we have now identified two disjoint open subsets
of the phase space: those leading to AC ends and those leading to forward-incomplete ends with asymptotics given by~\eqref{eq:stable:singularity}.
This raises the question of what sort of end behaviour might occur at the boundary between these two open regions. 
Since this end behaviour can occur as a limit of forward-complete AC ends we would expect to see 
some kind of forward-complete but non-AC end behaviour.

So suppose that we have a forward-complete expander end that is not AC.
Then by Proposition~\ref{prop:AC:end}, necessarily the warping $\frac{y}{x} \to \infty$ monotonically.
Thus, far along the end, $y$ is much bigger than $x$, and so one expects
that the expander ODEs are well approximated by the $\epsilon = 0$ limit of the
rescaled system \eqref{eq:Sp2:ODE:eps}, \ie
\begin{equation}
\label{eq:Sp2:ODE:limit}
(\hat{x}^2)' = 2\hat{x} - 2\frac{\hat{x}^2\hat\tau_2}{\hat y^2} ,
\qquad  (\hat y^2)' =  \hat\tau_2 , \qquad
\hat{\tau}_2' = \frac{2\wh{R}_1 \wh{S}}{3\hat{x}\hat y^2},
\end{equation}
where
\[\wh{R}_{1}:=\lambda \hat{x}\hat y^2-3\hat\tau_2,  \quad \ \wh S:=\hat y^2-\tfrac{3}{2}\hat{x}\hat\tau_2 .
\]
This rescaled  limit system has the following explicit forward-complete expander end solutions (defined for $t>0$) where $\hat x \to 0$, $\frac{\hat{y}}{\hat{x}} \to \infty$ as $t \to \infty$:
\begin{equation}
\label{eq:explicit_limit_sol}
\hat x = \frac{3}{\lambda t},
\quad \hat y = A \sqrt{t}\exp \left(\frac{\lambda t^2}{12}\right), \quad
\hat \tau_2 = A^2 \left(\tfrac{\lambda}{3}t^2 +1 \right) \exp\left(\frac{\lambda}{6} t^2 \right)
\end{equation} 
for $A>0$ constant. (One way to arrive at these explicit solutions is that they are exactly the solutions of~\eqref{eq:Sp2:ODE:limit}
on which the conserved quantity $\frac{\wh M}{\hat x \hat y^2}$ described in Remark \ref{rmk:ell_sharp} vanishes.)

Substituting these functions into the original unscaled system, one finds that
they solve the unscaled system up to an error of order $\exp(-\frac{\lambda}{6}t^2)$, so
it is reasonable to expect to be able to correct to a genuine solution by
adding some bounded correction terms.
Taking such an ansatz  (and changing variables by $u = t^{-2}$
as before) yields another irregular singular initial value problem,
and for each value of the parameter $A>0$ we find a unique~\Sp{2}-invariant forward-complete expander with
end behaviour modelled on \eqref{eq:explicit_limit_sol}. This $1$-parameter family of forward-complete expander ends all has
quadratic-exponential volume growth, with $x \to 0$, and $\tilde \tau_1 \nearrow 0$ as $t \to \infty$. 
(In view of Proposition \ref{prop:AC:end}, one might have expected that points at the boundary of the space of AC expander ends
should have $\tilde \tau_1 \to 0$ as $t \to \infty$.)

\medskip
We now observe that 
there is another well-behaved limit of rescalings of the~\Sp{2}-invariant soliton ODE system; this second rescaling limit ODE system turns out to be crucial in our later analysis of~\Sp{2}-invariant shrinker ends. If we rewrite the system in terms of coordinates 
\begin{equation}
\label{eq:heisenberg:scaling}
\check x = \epsilon^2 x, \quad \check y = \epsilon y \quad \text{and} \quad \check\tau_2 = \epsilon^2 \tau_2,
\end{equation}
then in the limit $\epsilon \to 0$ we obtain the ODE system
\begin{equation}
\label{eq:heisenberg_full}
\check x' = -\frac{\check x \check \tau_2}{\check y^2} - \frac{\check x^2}{2\check y^2},
\quad (\check y^2)' = \check x + \check \tau_2,
\quad \check \tau_2' = - \frac{4\lambda \check y^2(\check x + \tfrac{3}{2} \check \tau_2)}{3\check x}.
\end{equation}

This ODE system is equivalent to the system considered by Fowdar \cite[\S 5.2]{Fowdar:S1:invariant:LF} for Laplacian solitons on $\R \times \iwa^6$ where 
$\iwa^6$ is the Iwasawa manifold, \ie a certain compact $6$-dimensional nilmanifold obtained as a quotient of the complex 3-dimensional Heisenberg group
by a lattice.  (More generally, $M^6$ can be certain $T^2$-bundles over \hk{} $4$-manifolds $B$;  for the Iwasawa manifold one takes~$B=T^4$. 
We also mention in passing that Ball~\cite{Ball} found solutions to~\eqref{eq:heisenberg_full} with~$\lambda=0$.) For~$\lambda < 0$, Fowdar found an explicit solution, that in our coordinates takes the form
\begin{equation}
\label{eq:fowdar}
\check x = 4e^{\mu t}, \quad \check y^2 = \mu^{-1} e^{\mu t},
\quad \check \tau_2 = -3 e^{\mu t}
\end{equation}
for $\mu = \sqrt{\frac{-\lambda}{18}}$. This explicit solution
to~\eqref{eq:heisenberg_full} can also be used as an approximate end solution to the~\Sp{2}-invariant shrinker ODE system
and again this approximate end solution can be corrected to a forward-complete~\Sp{2}-invariant
shrinker end solution with $\frac{y}{x} \to 0$ as $t \to \infty$:
see Proposition~\ref{prop:fowdar_end} for a proof of the existence of this type of end
that does not rely on the theory of irregular singular value problems.

\begin{remark}
\label{rmk:scaling_reduction}
An important observation is that the rescaling limit ODE systems~\eqref{eq:Sp2:ODE:limit} and \eqref{eq:heisenberg_full}
each have an additional symmetry arising from the rescalings being used:
they are invariant under
\[ \hat x \mapsto \hat x, \; \hat y \mapsto \epsilon \hat y, \;
\check \tau_2 \mapsto \epsilon^2 \hat \tau_2 \quad \textrm{and} \quad
\check x \mapsto \epsilon^2 \check x, \; \check y \mapsto \epsilon \check y, 
\; \check \tau_2 \mapsto \epsilon^2 \check \tau_2 \]
respectively. (The effect of the latter on \eqref{eq:fowdar} is the same as
translation in $t$).
Both limit systems can therefore be reduced to an ODE system in two variables given by ratios that are invariant.

In particular, \eqref{eq:heisenberg_full} yields a self-contained ODE system for the invariant ratios $\alpha := \frac{\check \tau_2}{\check x}$, $\beta := \frac{\check y^2}{\check x}$. The analysis of the resulting system \eqref{eq:heisenberg:limit} will be important for understanding the forward-evolution of shrinkers.
Given a solution $(\alpha, \beta)$ of~\eqref{eq:heisenberg:limit} defined on
some $t$-interval, one can then recover $\check x$, $\check y$
and~$\check \tau_2$ (up to the symmetry action) on the same interval by one
further integration, \eg~$(\log{\check x})' = - \frac{1+\alpha}{\beta}$.
The explicit solution \eqref{eq:fowdar} arises naturally from a fixed point
\eqref{eq:heisenberg_fixed} of the reduced system~\eqref{eq:heisenberg:limit}.
\end{remark}

\subsection{Trichotomy for~\texorpdfSp{2}-invariant expander ends}
\label{ss:trichotomy}

To summarise the claims made in the previous two subsections concerning \Sp{2}-invariant
expander ends, one can construct the following: a 2-parameter family of AC
ends; a 1-parameter family of forward-complete ends with quadratic-exponential volume
growth; and a 2-parameter family of forward-incomplete solutions. 

We now prove the expander case of Theorem \ref{mthm:both_tri}: \emph{only} the three types of
end behaviour just described can occur for \Sp{2}-invariant expanders.

\begin{theorem}
\label{thm:tri:expanders}
The forward evolution of an~\Sp{2}-invariant $\lambda$-expander starting at any  regular initial data satisfies exactly one of the following:
\begin{enumerate}[left=0em]
\item
The solution is forward complete with $\frac{y}{x} \to \ell \in (0,\infty)$ as $t \to \infty$. 
Then the solution is weakly asymptotic to the unique closed~\Sp{2}-invariant~\gtwo-cone 
with $c_2/c_1=\ell$, \ie $\frac{x}{t} \to c_1$ and $\frac{y}{t} \to c_2$. 
In this case $M=3x+\tilde \tau_1$ remains positive for all $t$ and $\frac{M}{g^3} \to \frac{2\lambda}{1+2\ell^2}$ as $t \to \infty$ and $\tilde \tau_1$ and $\tilde \tau_2$ are both eventually positive. 
\item
The solution is forward complete with $\frac{y}{x} \to \infty$ monotonically as $t \to \infty$.
Then $\tilde \tau_1$ remains negative while $M = 3x + \tilde \tau_1$ remains
positive for all $t$, and both tend to $0$ as $t \to \infty$. 
Moreover,  as $t \to \infty$,  $tx \to \tfrac{3}{\lambda}$, while (up to
translation in $t$) $g^3e^{-\frac{\lambda}{6}t^2}$ has a finite limit;
in particular, $y$ grows faster than $e^{\frac{\lambda}{12}t^2}$.
\item
The solution has a finite extinction time $t_*$. Then $\tfrac{y}{x} \to \infty$
monotonically, and %
$\tilde \tau_1 \to -\tfrac{3a^2}{2}$ as $t \to t_*$ for some nonzero $a$.
Moreover, 
$x \to 0$, $(x^2)' \to -a^2$ and $g$ has a finite limit as $t \to t_*$.
In particular, $M \to - \frac{3a^2}{2} <0$ as $t \to t_*$.
\end{enumerate}
\end{theorem}
\noindent
Recall from~\eqref{eq:monotonic} that $M/g^3= (3x+ \tilde \tau_1)/g^3$ is decreasing in $t$ on any~\Sp{2}-invariant soliton.

\begin{remark}
\label{rmk:trichotomy:relevance}
In the context of smoothly-closing~\Sp{2}-invariant expanders this trichotomy 
turns out not to come into play because all such expanders are AC. 
However, in the final section we shall indicate why we expect closely
related issues to become pertinent when considering smoothly-closing
\sunitary{3}-invariant expanders on~$\Lambda^2_- \CP^2$. 
\end{remark}

There is an analogous trichotomy for solutions to the rescaling limit ODE system
\eqref{eq:Sp2:ODE:limit}. There the proof is easier due to the
existence of the conserved quantity $\frac{\wh M}{\hat x \hat y^2}$; the
three classes of solutions are then distinguished by the sign of $\wh M$.
Similarly we see that the (eventual) sign of the decreasing quantity $M/g^3$ plays an important role in distinguishing the end
behaviours of~\Sp{2}-invariant expanders. 

\begin{remark}
\label{rmk:incomplete_open}
In particular, Theorem \ref{thm:tri:expanders} entails that forward
incompleteness for expanders is equivalent to $M$ ever becoming negative, which is clearly an open
condition.
\end{remark}

Proposition \ref{prop:AC:end} already characterised  the AC expander ends appearing in (i).
Therefore the proof of Theorem \ref{thm:tri:expanders} is completed by
the next two propositions, which characterise the forward-incomplete
and forward-complete non-AC end solutions respectively.

\begin{prop}
\label{prop:incomplete:asymptotics}
If $M = 3x + \tilde \tau_1$ ever becomes negative on a local expander then
it has a finite extinction time $t_*$. As $t \to t_*$, 
$\tilde \tau_1$ has a strictly negative limit $-\frac{3a^2}{2}$,
$x \to 0$ and $(x^2)' \to -a^2$ for some $a >0$, 
\ie $x = a\sqrt{t_*-t} + o(\sqrt{t_*-t})$.
Furthermore, $g$ has a finite limit, so $y$ is of order $(t_*-t)^{-1/4}$.
\end{prop}

\begin{proof}
Pick any $C>0$ such that $\frac{M}{g^3} < -C$ holds at some instant. Then by
\eqref{eq:monotonic} this inequality persists. 
From
$x = \frac{1}{3} ( -\tilde \tau_1 + \frac{M}{g^3} g^3)$ and~\eqref{eq:monotonic} 
we can rewrite the ODE for $x$ in the form
\begin{equation}
\label{eq:x'M}
x' = - \frac{1}{3} \tilde \tau_1' + \frac{M}{3g^3}(g^3)' + \left(\frac{M}{3g^3}\right)' g^3 = - \frac{1}{3} \lambda x^2 -  \frac{x^2}{y^2}  + \frac{M}{3g^3}(g^3)'.
\end{equation}
Then using  $\frac{M}{g^3} < -C$, together with $g' \geq \frac12$ and $\lambda>0$, eventually the derivative of $x$ is bounded by
\[
x' < \frac{M}{g^3}g' g^2 < -\frac{C}{2}g^2.
\]
Hence $x$ is eventually decreasing and so by Proposition~\ref{prop:x_limit}, $x \to 0$ as we approach the end of its lifetime. 
Moreover,  since $g$ is increasing, $x$ cannot remain positive for infinite time, so the lifetime must be finite, say $t_*$. 
Hence by Lemma~\ref{lem:extinction_x} also $\frac{x}{y} \to 0$ as $t \to t_*$. 
Also since $\tilde \tau_1 < M<-Cg^3$
the limit of the increasing function $\tilde \tau_1$ as $t \to t_*$ is negative,
and we can call it $-\frac{3a^2}{2}$ for some $a > 0$.
Since $x \to 0$ we also have $M \to - \frac{3a^2}{2}$ as $t \to t_*$.

From \eqref{eq:x'M}, using $(g^3)' = \frac{1}{2}x^2+y^2$ we then get
\begin{equation}
\label{eq:x^2:dot}
\frac{d}{dt}x^2 =
-\frac{2\lambda x^3}{3} + \left(\frac{x^2}{3y^2} + \frac{2}{3}\right)M - \frac{2x^3}{y^2}  \to \frac{2}{3} M = - a^2 
\end{equation}
as $t \to t_*$ because $x , \frac{x}{y} \to 0$ and $M \to - \frac{3}{2}a^2$.
Hence  $x = a\sqrt{t_*-t} + o(\sqrt{t_*-t})$.
Therefore we also have the following bound for the derivative of $\log{g}$
\[ (\log g)' =  \frac{x^2 + 2y^2}{6xy^2} < \frac{C}{\sqrt{t_*-t}}, \]
which implies that the increasing function $g$ is bounded and therefore has a finite limit as $t \to t_*$. 
\end{proof}

\begin{remark}
\label{rmk:FT:asymptotics}
Under the hypotheses that $x \to 0$, $\frac{x}{y} \to 0$ and
${M \to - \frac{3}{2}a^2<0}$ as $t$ approaches a finite extinction time $t_*$, 
the final paragraph of the previous proof gives the same finer asymptotics
as~$t \to t_*$ irrespective of the sign of the dilation constant $\lambda$. 
\end{remark}

\begin{prop}
\label{prop:fc:yx:unbounded}
Any expander for which $M$ remains positive throughout its lifetime is forward complete.
If the warping $\frac{y}{x}$ is  unbounded above
as $t \to \infty$, then $M \to 0$ and $\tilde \tau_1 \to 0$ and 
(up to translation in $t$) $\frac{1}{x} - \frac{\lambda}{3}t \to 0$ and
$e^{-\frac{\lambda}{6}t^2}xy^2$ has a finite limit as $t \to \infty$.
\end{prop}

\begin{proof}
Using $(g^3)'=\frac{1}{2}x^2+y^2$, equation~\eqref{eq:x'M} is equivalent to
\[
x' = -\frac{\lambda x^2}{3} + \left(\frac{x}{6y^2} + \frac{1}{3x}\right)M - \frac{x^2}{y^2},
\]
which, in turn, is equivalent to 
\begin{equation}
\label{eq:lin_decay}
\frac{d(-x^{-1} + \frac{\lambda}{3}t)}{dt}
= \left(\frac{1}{6xy^2} + \frac{1}{3x^3}\right)M  - \frac{1}{y^2} .
\end{equation}
Recall from Proposition~\ref{prop:x:over:y:bded} that $y$ is bounded away from $0$ on any expander. 
Therefore if $M$ remains positive then the right-hand side of~\eqref{eq:lin_decay}  is bounded below;  so by integration
$-x^{-1} - \frac{\lambda}{3}t$ is bounded below on any finite interval.
Thus $x$ is bounded away from 0 on any finite interval, and this implies
forward completeness by Lemma \ref{lem:extinction_x}.

Now suppose $M$ remains positive but $\frac{y}{x} \to \infty$ as $t \to \infty$.
Then $\tilde\tau_1$ remains negative throughout, otherwise by Proposition~\ref{prop:xy:bounds} the warping $\frac{y}{x}$ is bounded above. 
$x$ cannot be bounded away from $0$, as otherwise $\tilde \tau_1'=\lambda x^2$ would imply that $\tilde \tau_1$ is eventually positive and 
hence $x\to 0$, by Proposition~\ref{prop:x_limit}. 
Together with
\[ (\log xy^2)' = \frac{\half x^2 + y^2}{xy^2} > \frac{1}{x} \]
this shows that $xy^2$ grows faster than $e^{\mu t}$ for any $\mu>0$, 
and hence so does $y^2=\tfrac{1}{x} xy^2$.

Since $M=3x+\tilde \tau_1$ is positive for all time, while $\tilde \tau_1$ remains negative and $3x \to 0$, 
this implies that $M \to 0$ (and hence also $\tilde \tau_1 \to 0$) as $t \to \infty$ . 
Because $M$ remains positive, \eqref{eq:M'} implies that 
\[
M' > -\frac{3x^2}{y^2}.
\]
Integrating this inequality and using the (super)exponential growth of $y^2$ and the fact that $M 
\to 0$ gives that for any $\mu>0$, 
$M< C e^{-\mu t}$ holds for a suitable constant~$C$.

Now we want to prove that all terms on the right-hand side of~\eqref{eq:lin_decay} are integrable, because
that forces $x^{-1} - \frac{\lambda}{3}t$ to have a finite limit as $t \to \infty$. %

The negative term $-1/y^{2}$ is exponentially decaying and so is certainly integrable. Therefore $-x^{-1} + \frac{\lambda}{3}t$ must be bounded below and so $tx$ is bounded below. The first positive term $M/6xy^2$ is also exponentially decaying and hence integrable. Finally, the lower bound for $tx$ together with 
the exponential decay of $M$ implies that $M/x^3$ is integrable. So the whole right-hand side of \eqref{eq:lin_decay} is integrable as claimed. 

Changing the variable $t$ by a translation, we can assume the limit
of $x^{-1} - \frac{\lambda}{3}t$ as $t \to \infty$ to be 0.
Because its derivative decays exponentially, then so too does
$x^{-1} - \frac{\lambda}{3}t$. Hence
\[ \frac{d}{dt} \log \left(e^{-\frac{\lambda}{6}t^2}xy^2\right)
= -\frac{\lambda}{3}t + \frac{x}{2y^2} + \frac{1}{x} \]
decays exponentially and so is integrable too, and hence $e^{-\frac{\lambda}{6}t^2}xy^2$ has a finite limit.
\end{proof}

\subsection{Forward-incomplete shrinker ends}
\label{ss:shrinker_fi_ends}

We now move on to considering the end behaviours of \Sp{2}-invariant shrinkers.
The ideas needed to understand forward-complete shrinker ends turn out to be qualitatively different, and are discussed in the next section.

The arguments for understanding the forward-incomplete case are more similar to the expander case, but the conclusion is in a sense simpler: a shrinker is forward-incomplete if and only if $x$ is eventually decreasing (while we just saw that the quadratic-exponential growth expander ends have~$x \to 0$ 
as $t \to \infty$ even though they are forward-complete).

Recall from \S\ref{ss:monotone} that for any  soliton $x$ is eventually monotone, and moreover in the non-steady case either $x \to 0$ or $x \to \infty$. In the expander case,  the trichotomy in Theorem \ref{thm:tri:expanders} implies that in case (i) where $x \to \infty$ then the warping $\frac{y}{x}$ is bounded above, whereas in both other cases $x \to 0$ and then the warping $\frac{y}{x} \to \infty$. 
We prove that, in fact, an alternative like this also holds for shrinkers.

\begin{lemma}
\label{lem:shrinker:x:unbounded}
On any shrinker, if $x \to 0$ then $\frac{y}{x} \to \infty$,
while if $x \to \infty$ then $\frac{y}{x}$ is bounded above. %
\end{lemma}

\begin{proof}
If $x \to 0$ then certainly $\frac{y}{x} \to \infty$ since $g^3=x^3 \left(\frac{y}{x}\right)^2$ is increasing.

On the other hand, suppose $x \to \infty$.
We know from Theorem \ref{thm:y:x:eventually:mono} that the warping $\frac{y}{x}$ is eventually monotone, and patently $\frac{y}{x}$ is bounded above if it is eventually decreasing.
We may therefore assume that $\frac{y}{x}$ is eventually increasing, or equivalently, $S$ is eventually negative. 
Now rewrite \eqref{eq:MS} as
\[ \frac{M(x^2+2y^2)}{g^3} = 2\lambda x^2 + \frac{7x^2}{y^2} + 2 +\frac{4S}{y^2} . 
\]
Since $\frac{x}{y}$ is eventually decreasing it is bounded above, say by $D$, and so eventually 
\[2\lambda x^2 + \frac{7x^2}{y^2} + 2 +\frac{4S}{y^2} < 2\lambda x^2 + 7D^2 + 2
\]
(because eventually $S<0$). 
Because $x \to \infty$ this expression and hence also $M$ is eventually negative.
By Lemma \ref{lem:monotone}, there exists $C > 0$ such that eventually $\frac{M}{g^3} < -C$. 
Recall the form of the ODE for $x'$ in terms of $\frac{M}{g^3}$ given in~\eqref{eq:x'M}. From that expression, 
using that $\frac{M}{g^3}<-C$ and estimating $(g^3)' > y^2$ we obtain
\[ x' < -\frac{x^2}{3}\left(C\frac{y^2}{x^2} + \lambda\right) . \]
Since $x$ is eventually increasing (by Corollary \ref{cor:x_monotone}, given that we assumed $x \to \infty$),
$\frac{y^2}{x^2}$ must therefore eventually be bounded above by $\frac{-\lambda}{C}$.
\end{proof}

The next result does not hold for expanders. Indeed, all the non-AC forward-complete expanders that occur in case (ii) 
of the expander trichotomy, Theorem~\ref{thm:tri:expanders}, have $x$ decreasing for all $t$ with $x \to 0$ (and $M\to 0$) as $t \to \infty$. 

\begin{prop}
\label{prop:incomplete_shrinkers}
Any shrinker with $x \to 0$ (equivalently, $x$ is eventually decreasing) is forward incomplete with a finite-time singularity of stable type, \ie 
as $t \to t_*$ the extinction time,  it has the asymptotic behaviour~\eqref{eq:stable:singularity}.
\end{prop}

\begin{proof}
First suppose for a contradiction that the solution is forward complete. $M$ cannot remain positive for all $t$ because then~\eqref{eq:x'M} yields the bound
\[x' > x^2 \left( - \frac{\lambda}{3} - \frac{1}{y^2} \right)\]
for all $t$.  Since $y \to \infty$ as $t \to \infty$, this  would imply that eventually $x'>0$, which contradicts $x \to 0$.
Therefore there exists $C > 0$ such that eventually $\frac{M}{g^3} < -C$. 
Since $-\frac{\lambda}{3}x^2 \to 0$, it is certainly bounded above, say by $D$. 
So eventually~\eqref{eq:x'M} implies
\[
x'< \frac{M}{3g^3} (g^3)' - \frac{\lambda}{3}x^2< -\frac{1}{2} C g^2 +D.
\]
Since $g$ grows at least linearly 
this implies that $x$ reaches $0$ in finite time, contradicting forward completeness. 

So now, together with Lemma~\ref{lem:shrinker:x:unbounded}, we know that the shrinker is forward incomplete with $x \to 0$ 
and $\frac{x}{y} \to 0$ as $t \to t_*$ and this certainly implies that $\tau_2$ is eventually positive (because $S$ is eventually negative and $\frac{x}{y} \to 0$).
Then by~\eqref{eq:tilde:tau1:tau2} $\tilde \tau_1$ has the same sign as $\lambda g^3 - 3\tau_2$, 
which is negative.  
Since $x$ is bounded above, the decreasing eventually negative function $\tilde \tau_1$ is bounded below. 
Hence the limit of $\tilde \tau_1$ as $t \to t_*$ exists and is equal to $-\frac{3}{2}a^2<0$ for some $a>0$ and
$M$ has the same limit (since $x \to 0$). 
The hypotheses required in Remark~\ref{rmk:FT:asymptotics} are therefore satisfied and 
hence we have a stable finite-time singularity at $t=t_*$. 
\end{proof}

Here is one (manifestly open) condition guaranteeing that $x$ is eventually decreasing (equivalently, that $x \to 0$). The open nature of this condition is consistent with our previous observation that our construction of a $2$-parameter family of stable finite-extinction-time soliton ends gives rise to an open subset 
of the soliton phase space (for any value of $\lambda$). 

\begin{lemma}
If at any instant a shrinker satisfies both $x'<0$ and $\frac{M}{g^3}<0$ then both those conditions persist.  
In particular, $x$ is eventually decreasing with $x \to 0$ and hence the shrinker is forward incomplete with a stable finite-time singularity. 
\end{lemma}
\begin{proof}
Negativity of  $\frac{M}{g^3}$ is clearly forward-preserved.  
Our analysis of the critical points of $x$ (in Corollary~\ref{cor:x_monotone}) implies %
$x'' <0$ at any critical point where $\frac{M}{g^3}<0$.  
If there were a later instant at which $x' \ge 0$ held, then  $x'' \ge 0$ would hold at the first instant when $x'=0$, which is a contradiction. 
\end{proof}

\begin{remark*}
An easy consequence of the previous result is that any shrinker for which $x$ attains a local maximum is forward incomplete:
at any local maximum of $x$ we must have $\frac{M}{g^3}\le 0$ and then immediately after the local maximum has occurred both $x'$ and $\frac{M}{g^3}$ 
become negative. 
If a smoothly-closing shrinker has a finite-time stable singularity then clearly $x$ must have a local maximum, since $x \to 0$ both as $t \to 0$ and as $t$ approaches
the extinction time. So the preceding result provides some sort of converse. 
\end{remark*}

\begin{remark*}
If $x'<0$ and $3 + \lambda y^2 \le 0$, then this forces $\frac{M}{g^3}<0$ because
by~\eqref{eq:x'M} 
\[
x' + \frac{x^2}{3y^2}(3+\lambda y^2) = \frac{M}{3g^3} (g^3)'.
\]
In fact, one can also show that the joint conditions $3 + \lambda y^2<0$ and $x'<0$ also persist.  
\end{remark*}

\section{Forward-complete \texorpdfSp{2}-invariant shrinkers and the shrinker limit system}
\label{sec:shrinker-tri}
In this section we prove the shrinker case of Theorem \ref{mthm:both_tri}: a trichotomy result for \Sp{2}-invariant shrinker ends, analogous to Theorem \ref{thm:tri:expanders} for expander ends. The nature of the proof is however completely different from in the expander case; 
a starring role will be played by the rescaling limit ODE system introduced in~\eqref{eq:heisenberg_full}. 

\begin{theorem}
\label{thm:tri:shrinkers}
The forward evolution of an~\Sp{2}-invariant shrinker with dilation constant $\lambda<0$ starting at any regular initial data satisfies exactly one of the following:

\begin{enumerate}[left=0em]
\item
The solution is forward complete and $\frac{y}{x} \to \ell \in (0,\infty)$ as $t \to \infty$. 
Then the solution is weakly asymptotic to the unique closed~\Sp{2}-invariant~\gtwo-cone 
with $c_2/c_1=\ell$, \ie $\frac{x}{t} \to c_1$ and~$\frac{y}{t} \to c_2$. 
In this case 
$\tilde \tau_1$ and $\tilde \tau_2$ are both eventually negative and
$\frac{M}{g^3} \to \frac{2\lambda}{1+2\ell^2}\in (2\lambda,0)$ as~$t \to \infty$.
\item
The solution is forward complete, but eventually $\frac{y}{x} \to 0$ monotonically.
Then $\frac{M}{g^3} \to 2 \lambda <0$ and~$\frac{\tilde \tau_2}{g^3} \to 0$ as $t \to \infty$. 
Up to translation in $t$ there is a unique such forward-complete shrinker end and
it satisfies $x \to \infty$, $\frac{\tau_2}{x} \to - \frac{3}{4}$ and $\frac{y^4}{x^2} \to - \frac{9}{8\lambda}$ as $t \to \infty$. 
In particular, it has exponential volume growth. 
\item
The solution has a finite extinction time $t_*$. Then eventually $\tfrac{y}{x} \to \infty$
monotonically, and $\tilde \tau_1 \to -\tfrac{3a^2}{2}$ as $t \to t_*$ for some nonzero $a$.
Moreover, 
$x \to 0$, $(x^2)' \to -a^2$ and $g$ has a finite limit as $t \to t_*$.
In particular, $M \to - \frac{3a^2}{2} <0$ as $t \to t_*$.
\end{enumerate}
\end{theorem}

\begin{proof}[Proof, assuming Theorem \ref{thm:udhav_end}]%
To distinguish these three types of end behaviour first recall from Corollary \ref{cor:x_monotone} and Proposition \ref{prop:x_limit} that either $x \to 0$ or $x \to \infty$. In the former case, Proposition \ref{prop:incomplete_shrinkers} already implies that we are in case (iii).

On the other hand, if $x \to \infty$, then Lemma \ref{lem:shrinker:x:unbounded} implies that $\frac{x}{y}$ is bounded away from 0, while Theorem \ref{thm:y:x:eventually:mono} implies that it is eventually monotonic.
If $\frac{x}{y}$ is bounded above, it must therefore converge, and then by Lemma \ref{lem:yx:bounded} we are in case~(i). To complete the proof of Theorem \ref{thm:tri:shrinkers}, it therefore remains only to prove that if eventually $x$ and $\frac{x}{y} \to \infty$ monotonically, then we are in case~(ii). We will accomplish
this in Theorem \ref{thm:udhav_end}.
\end{proof}

Most of this section is devoted to the proof of Theorem \ref{thm:udhav_end}. 
The main idea is that when $x \to \infty$, the end behaviour should be well
approximated by solutions to the limit system~\eqref{eq:heisenberg_full}.
In \S\ref{ss:fe:shrinker:limit} we analyse the end behaviour in that simpler
system, and in \S\ref{ss:shrinker:full} we adapt those arguments to
prove Theorem \ref{thm:udhav_end}. 

\begin{remark*}
Note that (unlike in the expander ends trichotomy) the decreasing quantity $\frac{M}{g^3}$ is eventually negative in all three cases. 
In fact, in all three cases $\frac{M}{g^3}$ has a finite negative limit.
(Of course this also implies that $\tilde \tau_1$ is also eventually negative on any shrinker.)
Instead, in the shrinker setting the three cases are distinguished by the eventual behaviour of $M$ itself: if $M$ remains bounded 
then the shrinker is forward incomplete; if $M$ is unbounded below then the shrinker is forward complete and 
$m:=\displaystyle\lim_{t \to \infty}\tfrac{M}{g^3}$ exists with $m \in [2\lambda,0)$; the end is AC if and only if $m>2 \lambda$ 
and non-AC if $m= 2 \lambda$. 
\end{remark*}

\subsection{Shrinker limit system comparison}
To implement the idea %
that~\Sp{2}-invariant shrinker ends with $x$ and $\frac{x}{y}$ sufficiently
large should be well approximated by solutions to the %
limit system~\eqref{eq:heisenberg_full}, we need to introduce 
a new set of variables adapted to this regime. 
If we change variables to
\[ \alpha := \frac{\tau_2}{x}, \quad \beta := \frac{y^2}{x}>0,
\quad \gamma := \frac{1}{x} >0\]
then the~\Sp{2}-invariant soliton ODE system \eqref{eq:ODE:tau2'} becomes
\begin{subequations}
\label{eq:heisenberg}
\begin{align}
\alpha' &= -2 (\lambda\beta-3\alpha\gamma)\left(\frac{\alpha+1}{1+2\beta\gamma} - \frac{1}{3}\right) + \frac{\alpha(1+2\alpha)}{2\beta} - \alpha\gamma, \\
\beta' &= 2\left(\alpha+\tfrac{3}{4}\right) - \beta\gamma, \\
\gamma' &= \frac{(2\alpha+1)\gamma}{2\beta} - \gamma^2.
\end{align}
\end{subequations}
This ODE system extends continuously to $\gamma = 0$. The set $\gamma = 0$ is moreover invariant, and the corresponding subsystem
\begin{subequations}
\label{eq:heisenberg:limit}
\begin{align}
\label{eq:heisenberg:limit:alpha}
\alpha' &= -2\lambda\beta(\alpha+\tfrac{2}{3}) + \frac{\alpha(1+2\alpha)}{2\beta}, \\
\beta' &= 2\left(\alpha+\tfrac{3}{4}\right)
\end{align}
\end{subequations}
coincides with that described in Remark \ref{rmk:scaling_reduction} for
invariant ratios in the rescaling limit system~\eqref{eq:heisenberg_full}.
Our main interest in~\eqref{eq:heisenberg:limit}  
now is  as a tool for understanding 
~\Sp{2}-invariant shrinker ends with~${\frac{x}{y} \to \infty}$;
we will therefore refer to \eqref{eq:heisenberg:limit} as the
\emph{shrinker limit system}.
The shrinker limit system is easier to analyse than the full
system \eqref{eq:heisenberg} simply because the global structure of planar ODE
systems is so
much more constrained than $3$-dimensional ODE systems.

\begin{remark*}
In passing, we note that the subsystem~\eqref{eq:heisenberg:limit} makes sense for any value of
$\lambda$ and, via \eqref{eq:heisenberg_full}, its solutions describe closed Laplacian
$\lambda$-solitons on $\R \times \iwa^6$ where $\iwa^6$ is the Iwasawa $6$-manifold mentioned earlier. 
Thus there is independent geometric motivation to study %
solutions to \eqref{eq:heisenberg:limit} for any $\lambda$.
Because our current interest in this system is as a tool to understand~\Sp{2}-invariant shrinker ends
we stick to the case $\lambda<0$; 
results we have established on the behaviour of solutions
to~\eqref{eq:heisenberg:limit} when $\lambda>0$ may appear elsewhere. 
Ball~\cite{Ball} studied solutions with $\lambda=0$; because of the additional scaling invariance
of steady solitons it turns out to be easier to treat. 

We mention in passing that there is also a more systematic way to understand the appearance of both rescaling limit systems described in Remark~\ref{rmk:scaling_reduction}
in terms of Lie algebra contractions of the Lie algebra $\lsp{2}$. The shrinker limit system~\eqref{eq:heisenberg_full} is related to a $10$-dimensional Lie algebra that 
is a semi-direct product of $\mathfrak{u}(2)$ with the $3$-dimensional complex Heisenberg Lie algebra; 
the  limit system~\eqref{eq:Sp2:ODE:limit} is instead related to the $10$-dimensional Euclidean algebra $\mathfrak{e}_4$. Both these $10$-dimensional 
Lie algebras appear as Lie algebra contractions of $\lsp{2} \simeq \lorth{5}$. 
Since this plays no direct role in the current paper we again defer presentation of further details to another venue. 
One might expect such Lie algebra contractions and the associated limiting ODE systems to play an important 
role in other cohomogeneity-one problems. 
\end{remark*}

\begin{remark}
\label{rmk:limit_lifetime}
Clearly any solution to the subsystem \eqref{eq:heisenberg:limit} can be extended (both forwards and backwards in $t$) so long as
$\alpha$ remains bounded and $\beta$ is bounded away from 0. In other words, a
solution with finite lifetime must have $\alpha \to \pm\infty$ or $\beta \to 0$
at the extinction time. By Remark \ref{rmk:scaling_reduction} there is a
corresponding 1-parameter family of solutions
to~\eqref{eq:heisenberg_full} with the same lifetime.
\end{remark}

For $\lambda < 0$, the shrinker limit system \eqref{eq:heisenberg:limit} has
a unique fixed point, located at
\begin{equation}
\label{eq:heisenberg_fixed}
\alpha = -\frac{3}{4}, \quad \beta = \beta_* := \sqrt{-\frac{9}{8\lambda}}.
\end{equation}
The linearisation of~\eqref{eq:heisenberg:limit} at this fixed point is given by the matrix
\[ \begin{pmatrix} \tfrac{5}{6} \sqrt{-2\lambda} & \tfrac{\lambda}{3} \\ 2 & 0
\end{pmatrix} \]
which has a pair of complex conjugate eigenvalues $\dfrac{5\sqrt{-2\lambda} \pm \sqrt{46\lambda}}{12}$.
Since these eigenvalues have positive real part, the fixed point is an unstable spiral point, with
a 1-parameter family of flow lines emanating from it. 

One easy conclusion from viewing the shrinker limit system as a subsystem of
the full shrinker ODEs is that a forward-complete non-AC end solution of the
type in Theorem \ref{thm:tri:shrinkers}(ii) really exists.

\begin{prop}
\label{prop:fowdar_end}
There is a unique (up to translation in $t$) non-AC forward-complete shrinker end
solution $(x, y, \tau_2)$ with $\frac{y^2}{x} \to \beta_*$, 
$\frac{\tau_2}{x} \to -\tfrac34$ and $x \to \infty$.
\end{prop}

\begin{proof}
In terms of the variables $(\alpha, \beta, \gamma)$, such a shrinker
end corresponds to a solution (with $\gamma > 0$) of \eqref{eq:heisenberg}
converging to the fixed point $(-\tfrac34, \beta_*, 0)$.
Compared with the fixed point in the limit system, the linearisation of \eqref{eq:heisenberg} 
has an additional eigenvalue $-\frac{1}{4\beta_*}$.
Since this eigenvalue is negative, while the other two have positive real parts
as before, this fixed point has a 1-dimensional stable manifold. It is transverse
to the plane $\gamma = 0$ (since that is the tangent plane to the unstable
manifold), so one half of it is a unique flowline with $\gamma > 0$ asymptotic to
the fixed point.
\end{proof}

It is not difficult to see that $(-\tfrac34, \beta_*, 0)$ is, in fact,
the only fixed point of the full shrinker system~\eqref{eq:heisenberg}.

\subsection{Forward-evolution in the limit system}
\label{ss:fe:shrinker:limit}
As a step towards proving Theorem \ref{thm:tri:shrinkers}, we first solve the
easier problem of determining all end behaviours in the shrinker limit system
\eqref{eq:heisenberg:limit}.

\begin{theorem}
\label{thm:limit_tri}
For $\lambda<0$, there exist  (up to time translation) exactly two forward-complete solutions to
\eqref{eq:heisenberg:limit}:
\begin{enumerate}
\item A unique solution, defined for all $t \in \R$,  with $(\alpha, \beta) \to (-\frac34, \beta_*)$ as $t \to -\infty$ and $\beta \to \infty$, $\alpha \to -\frac23$
(in fact, $\beta(\alpha + \frac23)$ remains bounded) as $t \to \infty$.
\item The constant solution $\alpha = -\frac34$, $\beta = \beta_*$ corresponding to the unique fixed point~\eqref{eq:heisenberg_fixed}.
\end{enumerate}
Any other solution has finite forward lifetime, and $\alpha \to \infty$ at its
extinction time.
\end{theorem}

When we move on to studying the full system \eqref{eq:heisenberg}, where we
will consider $\gamma$ to be small but non-zero, we will use the actual
statement of this theorem only as a heuristic. In fact, the only statement
from this subsection that we will explicitly refer to later is
Lemma \ref{lem:heisenberg_trap}(iii). 

However, 
a number of statements we prove here in the case $\gamma=0$ also hold when $\gamma$ is small.
The advantage is that the proofs of these statements are easier to follow when $\gamma = 0$,
and hence they serve as a models/warm-ups for the arguments needed to prove Theorem~\ref{thm:tri:shrinkers}.
Those parts of the proof of Theorem \ref{thm:limit_tri} that are not relevant for the proof of
Theorem \ref{thm:tri:shrinkers} will be outlined more briefly. 

\begin{remark}
Any solution of~\eqref{eq:heisenberg_full} satisfies
\[
\left(\log{\frac{\check y}{\check x}}\right) ' = \frac{3}{2\beta} \left(\alpha + \frac{2}{3} \right) = \beta\left(\alpha+\frac{2}{3}\right)  \frac{3}{2\beta^2}.
\]
The fact that in (i) above,  $\alpha \to - \frac{2}{3}$, $\beta \to \infty$ with $\beta(\alpha+\frac23)$ being bounded therefore implies (since $\beta' \to \frac{1}{6}$ and so $\beta$ grows linearly in $t$) that the
corresponding solutions to~\eqref{eq:heisenberg_full} have
$\frac{\check y}{\check x}$ converging to some $\ell \in (0,\infty)$
as $t \to \infty$.
They are all related by the 
scaling symmetry described in Remark~\ref{rmk:scaling_reduction}
(which changes $\ell$).
(i) can therefore be interpreted as the essentially unique AC end solution
to~\eqref{eq:heisenberg_full}.
This AC end solution is also complete backwards in $t$, where it
has a cusp-type end modelled on~\eqref{eq:fowdar} as $t \to -\infty$. 
It admits two distinct geometric interpretations.
\begin{itemize}[left=0em]
\item It defines a complete gradient Laplacian shrinker on $\R \times \iwa^6$ (where $\iwa^6$ denotes the Iwasawa $6$-manifold mentioned in Section~\ref{ss:fc_non_ac})
with two different types of end: a finite-volume `cusp-type' end as $t \to -\infty$ and a Euclidean volume growth end as $t \to +\infty$. 
\item
For each $\ell \in (0, \infty)$, we can form a broken flowline in the system
\eqref{eq:heisenberg_full} by joining the
solution from Proposition \ref{prop:fowdar_end}
to the AC solution scaled so that
$\frac{\check y}{\check x} \to \ell \in (0,\infty)$.
On the other hand, we explained in \S\ref{subsec:end_sivp} that there should
exist a unique AC end solution to the original shrinker system
\eqref{eq:ODE:tau2'} for each fixed $\ell$. When $\ell$ is large, a large
portion of this solution should be approximated by the broken flowline.
\end{itemize}
\end{remark}

\begin{figure}
\begin{center}
\includegraphics[width=17cm]{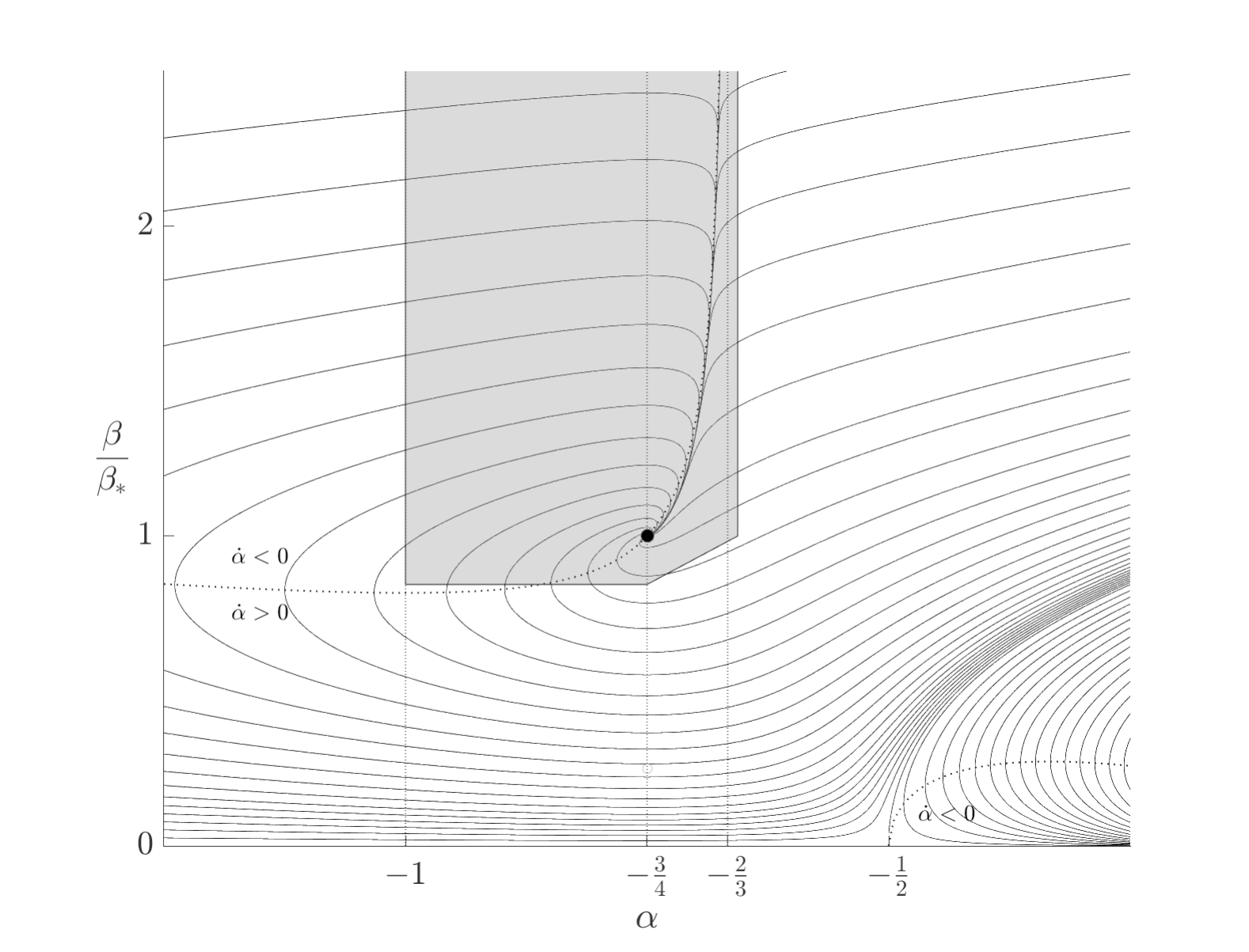}
\caption{Phase diagram of the shrinker limit system~\eqref{eq:heisenberg:limit}}
\label{fig:heisenberg}
\end{center}
\end{figure}

The statement of Theorem \ref{thm:limit_tri} is tied up with the following
picture of the global dynamics of the limit
system \eqref{eq:heisenberg:limit}, which is immediately apparent from its phase diagram (see Figure \ref{fig:heisenberg}). The AC solution in (i)
acts as a barrier: nearby initial conditions on one side of it (to the right of it in our figure) lead to solutions where immediately
$\alpha \to \infty$ monotonically; nearby initial conditions on the other side of the AC solution instead
lead to solutions which first `take a detour around the fixed point', \ie 
there is a regime where %
$\alpha$ and $\beta$ both decrease
for some time, 
before eventually $\alpha \to \infty$ monotonically.
 
To rigorously capture the key aspects of this long-term behaviour, our main tool
is an unbounded polygonal region $R$---shown shaded in Figure \ref{fig:heisenberg}---that contains the fixed point, does not allow
reentry, and contains no closed orbits (besides the fixed point).
There is a certain amount of slack in choosing this region $R$; 
before giving a precise definition of $R$ let us first sketch its key features.

The curve in the $(\alpha,\beta)$-half-plane (a half-plane because
$\beta \ge 0$) where $\alpha'=0$ 
is smooth, but consists of two components. 
The region where $\alpha'>0$ is connected, but there are two components 
where $\alpha'<0$, one bounded and the other unbounded.
We focus on the unbounded component, because  the
fixed point lies on its boundary.
For large $\beta$ the maximum
value where $\alpha' \leq 0$ is close to $-\frac23$, 
and it approaches $-\frac23$ as $\beta \to \infty$. 
Then $\alpha' > 0$ when $\alpha = -\frac23$ (for any value of $\beta$), and an infinite portion of the ray
$\alpha=-\frac23$ is a good choice for one edge of $R$.
With an eye toward making the argument work for
small but non-zero $\gamma$, we instead take one edge to be (an infinite portion of) the ray at a slightly larger value of $\alpha$, namely at
$\alpha = -\frac{21}{32}$.
On the other hand, the minimum of $\beta$ on the unbounded component of $\alpha'\le0$ occurs at
$\alpha = -1$, where $\beta = \sqrt{-\frac{3}{4\lambda}} = \sqrt{\frac23}\beta_*$; 
as another edge for $R$ we take the infinite portion of the ray with $\alpha = -1$ that starts a little
above this value of $\beta$.

Choosing the bottom part of the boundary of $R$ is a little more fiddly,
and requires some control on $\alpha'$ relative to $\beta'$ in a suitable region.
Sufficient control for our later purposes is established in the following lemma. 
\begin{lemma}
For $\beta < \beta_*$ and $-\frac{3}{4} < \alpha < -\tfrac{21}{32}$ any solution $(\alpha,\beta)$ of the shrinker limit system satisfies
\begin{equation}
\label{eq:slope}
(8\beta_*\alpha-5\beta)' > 0 .
\end{equation}
\end{lemma}

\begin{proof}
First note that $\beta < \beta_*$ implies
\begin{align*}
\frac{d}{d\beta}\left(-2\lambda\beta(\alpha+\tfrac{2}{3}) +
\frac{\alpha(1+2\alpha)}{2\beta}\right) &= -2\lambda(\alpha+\tfrac{2}{3}) -
\frac{\alpha(1+2\alpha)}{2\beta^2} \\ &<
-2\lambda(\alpha+ \tfrac{2}{3}) + \frac{4\lambda}{9}\alpha(1+2\alpha)
= -\frac{2\lambda}{9}(-4\alpha^2+7\alpha+6).
\end{align*}
Because $-4\alpha^2+7\alpha+6$ is increasing for all negative $\alpha$ its value on the interval
$[-\tfrac{2}{3}, -\tfrac{21}{32}]$ is bounded above by its value at
$-\tfrac{21}{32}$, which is $-\frac{81}{256}$. Thus $\frac{d\alpha'}{d\beta} < 0$ in this region, so $\alpha'$ is bounded below
 by the value at $\beta = \beta_*$, \ie $\alpha'$ at the point $(\alpha,\beta)$ is greater than $\alpha'$ at the point $(\alpha,\beta_*)$, and the latter
is equal to
$\beta_*^{-1}(\alpha^2 + \frac{11}{4}\alpha + \frac{3}{2})
= \beta_*^{-1}(\alpha+2)(\alpha+ \frac{3}{4}) > 0$. %
Thus
\[ (8\beta_*\alpha-5\beta)' \ge 
8(\alpha+2)(\alpha+ \tfrac{3}{4}) - 10(\alpha+\tfrac{3}{4}) =
8(\alpha + \tfrac{3}{4})^2 \geq 0 . \qedhere \]
\end{proof}

We now give the precise definition of the set $R$. 
Let
\begin{equation}
\label{eq:R:def}
R := \left\{(\alpha,\beta) :
\; -1 \leq \alpha \leq -\tfrac{21}{32}, \; \beta \geq \tfrac{27}{32}\beta_*,
\; 3\beta - 5\beta_*\alpha \geq \tfrac{201}{32}\beta_* \right\}. 
\end{equation}
Geometrically, $R$ is a closed unbounded convex polygonal domain with two parallel half-infinite edges and two finite edges, 
with corners at $\left(-1, \frac{27}{32}\beta_*\right)$,
$\left(-\frac{3}{4}, \frac{27}{32}\beta_*\right)$ and
$\left(-\frac{21}{32}, \beta_*\right)$.

\begin{remark}
\label{rmk:R:eternal}
Note that $\beta$ is bounded away from $0$ and $\alpha$ is bounded in $R$. Hence it follows from Remark~\ref{rmk:limit_lifetime} 
that no solution of~\eqref{eq:heisenberg:limit} can become extinct (either forward or backward in $t$) while it remains in $R$. 
\end{remark}

\begin{lemma}
\label{lem:heisenberg_trap}
In the limit system~\eqref{eq:heisenberg:limit},
the region $R$ has the following properties.
\begin{enumerate}
\item
Going forwards in time, $R$ cannot be entered from the outside. 
\item
$R$ contains no nontrivial closed orbits. 
\item
The only solution that remains in $R$ for all future time and has $\beta$
bounded above is the constant solution at the fixed point
$(-\frac{3}{4}, \beta_*)$.
\end{enumerate}
\end{lemma}

\begin{proof}
We first prove (i). 
First we deal with the two edges on which $\alpha$ is constant. 
On the edge $\alpha = -1$,  $\beta > \frac{27}{32}\beta_*$ we have
\[ \alpha' =
\frac{2 \lambda \beta}{3}+ \frac{1}{2\beta}.
\]
For $\beta>0$ this is a strictly decreasing function of $\beta$ with image $\R$. 
Its value at $\beta = \frac{27}{32} \beta_*$ is $- \frac{139}{4096 \beta_*}$. 
Hence on the edge $\alpha = -1$,  $\beta > \frac{27}{32}\beta_*$, $\alpha' \le - \frac{139}{4096 \beta_*}<0$, 
so the flow points out of~$R$.
Along the edge with $\alpha = -\frac{21}{32}$, (for any value of $\beta$) both terms in the expression for $\alpha'$ are positive
because $\alpha \in (-\frac{2}{3}, -\frac{1}{2})$.
So along this edge the flow again points out of~$R$.

On the edge with constant $\beta$, \ie $-1 \leq \alpha \leq -\frac{3}{4}$, $\beta = \frac{27}{32}\beta_*$
the fact that $\alpha \leq -\frac{3}{4}$ ensures $\beta' < 0$, except at the endpoint $\alpha=-\frac{3}{4}$ where $\beta'=0$, but $\alpha'>0$.
So again the flow points out of $R$ along this edge. 

Finally, because the remaining edge $-\frac{3}{4} \leq \alpha \leq -\frac{21}{32}$,
$3\beta = (5\alpha +\frac{201}{32})\beta_*$ is contained in the region
where \eqref{eq:slope} holds, along this edge we get
\[
(5\beta_*\alpha - 3\beta)'  = \frac{3}{5} (8 \beta_* \alpha - 5 \beta)' + \frac{1}{5} \beta_* \alpha' \ge \frac{24}{5} \left(\alpha+\tfrac{3}{4}\right)^2 +  \frac{1}{5} \beta_* \alpha ' >0\,
\]
because $\alpha'$ is strictly positive on it, proving (i).

For (ii) we use the Bendixson--Dulac criterion to rule out nontrivial closed orbits. The vector field
\[ V:=
\left(-2\lambda\beta(\alpha+\tfrac{2}{3}) + \frac{\alpha(1+2\alpha)}{2\beta}, \;
2\!\left(\alpha+\tfrac{3}{4}\right) \right) \]
whose flow defines the shrinker limit system has
\[ \Div V = \frac{1+4\alpha-4\lambda\beta^2}{2\beta} . \]
Every point in $R$  satisfies $\alpha \ge -1$ and
$-\lambda \beta^2 \ge -\lambda \frac{27^2}{32^2} \beta_*^2
= \frac{3^8}{2^{13}}>\frac{4}{5}$. 
The divergence of $V$ is certainly positive when $\alpha \ge -1$ and $-\lambda\beta^2 > \frac{4}{5}$,
and so it is positive on~$R$.
Hence since $R$ is simply connected, the Divergence Theorem rules out the existence of nontrivial (\ie not fixed points) closed orbits contained in $R$.

By Remark~\ref{rmk:R:eternal}, a solution that remains in $R$ must have infinite lifetime.
(iii) now follows from the Poincar\'e--Bendixson theorem:
if $(\alpha,\beta)$ stays in a bounded region in $R$ then its limit set
must be either a closed orbit (which is impossible by~(ii)) or the unique fixed
point. In the latter case the solution must be constant since the fixed point is
strictly unstable.
\end{proof}

\begin{lemma}
\label{lem:incomplete_limit}
Any solution of~\eqref{eq:heisenberg:limit} that starts in
$T := (-\infty, -\frac{21}{32}) \times (0, \infty) \setminus R$
reaches $\alpha = -\frac{21}{32}$ in finite time.
\end{lemma}

\begin{proof}
We first establish that $-\alpha$ and $\beta$ are bounded above. To obtain an upper bound $\beta_0$ on $\beta$, just note that at any point in $T$ with
$\beta > \beta_*$ we have $\alpha < -1$, and hence $\beta' < - \frac{1}{2}$. 
Given an upper bound $\beta_0$ for $\beta$, by
\eqref{eq:heisenberg:limit:alpha} there is
certainly $C > 0$ such that $\alpha' > 0$ whenever 
$\alpha < -C$ and $\beta < \beta_0$.
Thus $\alpha$ cannot keep decreasing beyond $-C$ and so $\alpha$ is bounded below.

Next we prove that $\frac{\alpha^2}{\beta}$ is bounded above, by showing that whenever it is sufficiently large its derivative is negative.
We have
\[ \frac{d}{dt}\frac{\alpha^2}{\beta} = \frac{2\alpha\beta\alpha'-\alpha^2\beta'}{\beta^2}
= \frac{-4\lambda\alpha\beta^2(\alpha + \tfrac{2}{3}) + \alpha^2 + 2\alpha^3 - \alpha^2(2\alpha + \tfrac{3}{2})}{\beta^2}
= -4\lambda \alpha(\alpha+\tfrac{2}{3}) - \frac{\alpha^2}{2\beta^2} \]
and hence when $\alpha<0$
\begin{equation}
\label{eq:alpha2/beta}
\frac{d}{dt} \frac{\alpha^2}{\beta} < - \frac{\alpha^2}{\beta} \left( \frac{1}{2\beta} + 4 \lambda \beta \right).
\end{equation}
Because of the bound $\alpha^2<C^2$, we have
$\frac{1}{\beta} > \frac{1}{C^2} \frac{\alpha^2}{\beta}$, 
which implies that $\frac{d}{dt} \frac{\alpha^2}{\beta} <0$ whenever $\frac{\alpha^2}{\beta}$ is sufficiently large. 
Hence $\frac{\alpha^2}{\beta}$ is bounded above for any solution that remains in $T$. %

Since $-\alpha$ is bounded above, it follows that $\beta$ is bounded away
from 0. In particular, by Remark~\ref{rmk:limit_lifetime} the solution cannot
become extinct while it stays in $T$.  To complete the proof,
suppose for a contradiction that $(\alpha,\beta)$ stays in $T$ for infinite
time. Then $\alpha$ cannot remain smaller than  $-\frac{4}{5}$ for all time,
since then $\beta' < 2 (- \frac{4}{5}+\frac{3}{4}) =- \frac{1}{10}$ and so the
positive quantity $\beta$ would become $0$ in finite time, violating forward
completeness.  But also $\alpha' > 0$ is bounded away from zero for all points
in $T$ with $\alpha \ge  -\frac{4}{5}$, so once $\alpha$ reaches $-\frac{4}{5}$
it then increases to $-\frac{21}{32}$ in finite time.
\end{proof}

\begin{lemma}
\label{lem:limit_boundS}
For any solution of~\eqref{eq:heisenberg:limit} that stays in $R$ throughout its lifetime, the quantity
$\beta(-\frac32 \alpha - 1)$ remains bounded.
\end{lemma}

\begin{proof}
Let $f : = \beta \left(-1 - \frac{3}{2}\alpha\right)$. By our assumption that the solution remains in $R$,
$\beta$ is bounded away from zero and $\alpha$ is bounded, and hence
\begin{equation}
\label{eq:push_alpha_limit}
\alpha' = \frac{4\lambda}{3} f + \frac{\alpha + 2\alpha^2}{2\beta}
< \lambda f
\end{equation}
whenever $f$ is large enough,
while $\alpha \geq -1$ and $\beta > \beta_*$ give
\[ \beta' = 2\alpha + \frac32 \geq -\frac12 = \frac{4\lambda}{9} \beta_*^2 > \frac{4\lambda}{9} \beta^2 .\]
Therefore whenever $f$ is sufficiently large and $\beta > \beta_*$
\[ f' = \frac{\beta'}{\beta}f - \frac{3}{2}\beta \alpha'
> - \frac{19\lambda}{18}f\beta>0. \]
Hence if $f$ is unbounded then eventually it is increasing monotonically.
Then by~\eqref{eq:push_alpha_limit} $\alpha'$ is negative and bounded away
from~0, hence the solution would reach $\alpha=-1$ in finite time and then exit $R$,
which is a contradiction. 
\end{proof}

\begin{proof}[Outline proof of Theorem \ref{thm:limit_tri}]
Lemma \ref{lem:heisenberg_trap}(i) implies that any solution either remains in
$R$ for all $t$ or eventually leaves $R$ and never returns to it.  
In the latter case, Lemma \ref{lem:incomplete_limit} implies that
eventually $\alpha \geq -\frac{21}{32}$.
From there one can see that $\alpha$ must eventually become positive. After that happens,
$\alpha + \frac{4\lambda}{9} \beta^2$ is increasing, which leads to
$\frac{d}{dt} \frac{1}{\beta} < \frac{8\lambda}{9}$.
Therefore $\beta \to \infty$ and $\alpha \to \infty$ in finite time.

It remains to show that the AC solution (i) exists, is unique and that it is the only solution other
than (ii) that stays in~$R$. We first observe that by Lemma \ref{lem:heisenberg_trap}(iii), any
solution other than (ii) that stays in $R$ has $\beta$ unbounded above.
When $\beta$ is big, Lemma~ \ref{lem:limit_boundS} implies that $\alpha$ is close to
$-\frac23$, so certainly $\beta' > 0$. Therefore we must, in fact, have $\beta \to \infty$ and using Lemma~ \ref{lem:limit_boundS} again this implies that
$\alpha \to -\frac23$.

The existence and uniqueness of such an end solution can be deduced from 
the same sort of analysis of irregular singular initial value problems
as described for AC shrinker ends in~\S\ref{subsec:end_sivp}. 
In fact, here the set-up ends up being rather easier because it can be rewritten as an ODE
for a single scalar variable.
We then need to understand the backward-evolution of this end solution.

Since Lemma \ref{lem:heisenberg_trap}(i) forbids reentry to $R$ forward in time,
the solution remains in $R$ as we evolve backwards in $t$.
In particular, by Remark \ref{rmk:limit_lifetime} the solution cannot become
backward incomplete.
We can also prove that $\beta$ remains bounded. Then if $(\alpha, \beta)$ failed to
converge to the unique fixed $(-\frac34, \beta_*)$ as $t \to -\infty$,
the Poincar\'e--Bendixson theorem would yield a non-trivial limit
cycle, but that is ruled out by Lemma \ref{lem:heisenberg_trap}(ii).
\end{proof}

\subsection{Returning to the full system}
\label{ss:shrinker:full}

For any~\Sp{2}-invariant shrinker we know from Lemma~\ref{lem:shrinker:x:unbounded} 
that if $x \to \infty$ then $\frac{x}{y}$ is bounded below, 
but we do not know that $\frac{x}{y}$ is bounded above. 
We would next like to say that once $x$ and $\frac{x}{y}$ are both sufficiently big, or
equivalently $\gamma$ and $\beta\gamma$ are both sufficiently small, then the full system~\eqref{eq:heisenberg} is well
approximated by the limit system~\eqref{eq:heisenberg:limit}, and thereby conclude that reentry to the region $R$ still
remains impossible.

\begin{corollary}
\label{cor:trap}
For any shrinker $(\alpha, \beta,\gamma)$ such that $\frac{x}{y}>8$ 
is increasing for all future time, either eventually $(\alpha,\beta)$
must remain in $R$ or else eventually it remains in
$T := (-\infty, -\frac{21}{32}) \times (0, \infty) \setminus R$.
\end{corollary}

\begin{proof}
Since $\frac{x}{y}$ increasing implies $x' \geq \frac{1}{3}$ we must
have $x \to \infty$ monotonically (by Proposition \ref{prop:x_limit}, or more simply because
if $x$ were bounded then $\frac{x}{y} \to \infty$ only by $y \to 0$, which
would violate $xy^2$ increasing). Hence $\gamma$ decreases to $0$ monotonically. 

We now argue that once $\gamma$ is small enough, it is impossible for $(\alpha,\beta)$ to enter $R$ from the outside.
It remains impossible to enter $R$ across the edge
with $\beta = \frac{27}{32}\beta_*$ because $\beta' < -\frac{27}{32} \beta_* \gamma<0$ when~$\gamma > 0$.

The inequality preventing entry across the edge with
$3\beta - 5\beta_*\alpha \geq \frac{201}{32}\beta_*$ in the limit system case was not sharp.
Because this edge is compact, the inequality still holds for sufficiently
small $\gamma$.

For $\alpha = -1$, the coefficient of
$\gamma$ in the expression for $\alpha'$ is simply $3$. So
$\alpha' \le  -\frac{139}{4096\beta_*} + 3\gamma < 0$ for any
$\beta \ge  \frac{27}{32}\beta_*$ and any sufficiently small $\gamma$.

Finally, for $\alpha \ge -\frac{21}{32}$ we get
\[ S = x^2\left(\frac{y^2}{x^2} - 1 - \frac{3}{2}\alpha\right)
\le x^2\left(\frac{y^2}{x^2} - \frac{1}{64}\right) , \]
so requiring $\frac{x}{y} > 8$ and increasing forces
$\alpha < -\frac{21}{32}$.
So we cannot enter along the edge where $\alpha = - \frac{21}{32}$ either.

Since a solution cannot reenter $R$, eventually either $(\alpha,\beta)$ remains in $R$ or it remains outside~$R$. 
\end{proof}

\begin{prop}
\label{prop:trap}
Any shrinker $(\alpha, \beta,\gamma)$ such that $\frac{x}{y}>8$ 
is increasing for all future time is forward complete, and eventually $(\alpha,\beta)$
must remain in $R$.
\end{prop}

\begin{proof}
To prove forward completeness, it suffices to show that $-\frac{\alpha}{\beta}$ is bounded. For then
$(\log{\frac{x}{y}})' = \frac{S}{g^3} = \gamma - \frac{1}{\beta} - \frac{3\alpha}{2\beta} <  \gamma - \frac{3 \alpha}{2 \beta}$ is bounded above (because $\gamma$ and $-\frac{\alpha}{\beta}$ are). 
Hence the increasing function $\frac{x}{y}$ remains bounded above on any finite interval, which rules out extinction by our forward-incompleteness criteria.

By Corollary \ref{cor:trap}, $(\alpha,\beta)$ eventually stays either in $R$ or in $T$.
In the former case, $-\frac{\alpha}{\beta}$ is bounded above just by the definition of $R$.
So it remains to prove that $-\frac{\alpha}{\beta}$ remains bounded if the solution remains in $T$. 
Since $-\alpha > \frac{21}{32}$, we can equally well prove that $\frac{\alpha^2}{\beta}$ is bounded above.

The rest of the proof follows the proof of Lemma \ref{lem:incomplete_limit}.
The arguments that $-\alpha$ and $\beta$ are bounded above carry over
without any problem.

Adapting the estimate \eqref{eq:alpha2/beta}
for the derivative of $\frac{\alpha^2}{\beta}$ to take into account that $\gamma$ is non-zero (but small) takes a bit more work.
Here \eqref{eq:heisenberg} yields the more complicated equation
\[
\frac{d}{dt}\frac{\alpha^2}{\beta} =
\frac{1}{\beta^2} \left( -4\alpha\beta (\lambda\beta-3\alpha \gamma) \left(\frac{3\alpha+2-2 \beta \gamma}{3(1+2 \beta \gamma)}\right)+ 
\alpha^2\left( - \tfrac{1}{2} -  \beta \gamma\right) \right).
\]
However, because various terms have definite signs we can still derive from this a differential inequality that is very close to~\eqref{eq:alpha2/beta}.
To see this, first note that
because $\alpha<0$, $\lambda<0$ and $3 \alpha+2 - 2 \beta \gamma = - \frac{2S}{x^2}<0$ (and of course $\gamma$ and $\beta$ are positive) we obtain the upper bound
\begin{align*}
\frac{d}{dt}\frac{\alpha^2}{\beta} & < - \frac{\alpha ^2}{2\beta^2}  + \frac{1}{\beta^2} \left( -4\alpha\beta^2 \lambda \left(\frac{3\alpha+2-2 \beta \gamma}{3(1+2 \beta \gamma)}\right)\right)\\
\intertext{which (because the second term is positive) is bounded by}
& < - \frac{\alpha ^2}{2\beta^2}  - \frac{4\lambda}{3\beta^2} \alpha \beta^2 (3 \alpha+2 - 2 \beta \gamma) < - \frac{\alpha ^2}{2\beta^2}  - 4\lambda \alpha ( \alpha- \tfrac{2}{3} \beta \gamma)\\
& = - \frac{\alpha^2}{\beta} \left( \frac{1}{2\beta} + 4 \lambda \beta - \frac{8\lambda}{3\alpha} \beta^2 \gamma \right).
\end{align*}
Compared to ~\eqref{eq:alpha2/beta} the only difference is the additional (negative) term $- \frac{8\lambda}{3\alpha} \beta^2 \gamma$. 
However, this additional term is harmless. Indeed, 
since $\beta \gamma = \frac{y^2}{x^2}$ is decreasing, $-\frac{1}{\alpha} < \frac{32}{21}$ and we have already established that $\beta$ is bounded, the term
$ -\frac{8\lambda}{3\alpha} \beta^2 \gamma$ is certainly bounded. Hence
since $\frac{\alpha^2}{\beta}$ large implies~$\frac{1}{\beta}$ large, so again
$\frac{d}{dt} \frac{\alpha^2}{\beta}<0$ whenever $\frac{\alpha^2}{\beta}$ is sufficiently large. So $\frac{\alpha^2}{\beta}$ is bounded above as claimed, finishing the proof of forward completeness.

To complete the proof of the proposition, it remains to rule out the
possibility that $(\alpha,\beta)$ stays in $T$ for infinite time, which we can
do by the same contradiction argument as in Lemma \ref{lem:incomplete_limit}.
\end{proof}

We will need the following further control on the behaviour of shrinkers that
remain in $R$, adapting Lemma \ref{lem:limit_boundS} from the limit system.

\begin{lemma}
\label{lem:boundS}
A shrinker with $\frac{x}{y} > 8$ and $x \to \infty$ such that $(\alpha,\beta)$ stays in $R$ must have $\frac{Sy^2}{x^3} = \frac{S}{x^2}\beta$ bounded above.
\end{lemma}

\begin{proof}
We follow the argument from Lemma \ref{lem:limit_boundS}, but now need additional control on the new terms involving $\gamma$ that arise. 
Let $f : = \frac{Sy^2}{x^3} = \frac{S}{x^2} \beta = \beta \left(-1 - \frac{3}{2}\alpha +\gamma \beta \right)>0$. Then we can express the first term in the derivative of $\alpha$ from~\eqref{eq:heisenberg} as
\[ -2(\lambda\beta - 3\alpha\gamma)\left(\frac{2 + 3\alpha - 2\beta\gamma}{3(1+2\beta\gamma)}\right)
=\left(\frac{\frac{4\lambda\beta}{3} - 4\alpha\gamma}{1+ 2\beta\gamma} \right) \frac{S}{x^2}
= \left(\frac{\frac{4\lambda}{3} - \frac{4\alpha\gamma}{\beta}}{1+ 2\beta\gamma}\right) f . \]
Since $\beta\gamma = \frac{y^2}{x^2} < \frac{1}{64}$, $\beta > \beta_*$ and $-1 \leq \alpha \leq -\frac{21}{32}$, certainly
\begin{equation}
\label{eq:push_alpha}
 \alpha' = \left(\frac{\frac{4\lambda}{3} - \frac{4\alpha\gamma}{\beta}}{1+ 2\beta\gamma} \right) f + \frac{\alpha(1+2\alpha)}{2\beta} - \alpha\gamma< \lambda f
 \end{equation}
for $f$ large enough and $\gamma$ small enough (because $\alpha$ bounded,  $\beta$ bounded below away from $0$ and $\gamma \to 0$ certainly imply that 
the second and third terms in the previous equation are bounded above). 
Meanwhile
\[ (\beta\gamma)' = -2\beta\gamma \frac{S}{xy^2} = - \frac{2\gamma f}{\beta} 
\]
and so
\[ f' =\left( \frac{S}{x^2} \beta \right)' 
= \frac{\beta'}{\beta}f + \beta\left((\beta\gamma)' - \frac{3}{2}\alpha'\right)
= f \beta \left( - \frac{3\alpha'}{2f} - \frac{2\gamma}{\beta} + \frac{\beta'}{\beta^2} \right).
\]
Since $\alpha \ge -1$ in $R$, for any $\beta> \beta_*$
we have the lower bound 
\[
\frac{\beta'}{\beta^2} \ge - \frac{1}{2\beta^2} - \frac{\gamma}{\beta} > - \frac{1}{2 \beta_*^2} - \frac{\gamma}{\beta}= \frac{4\lambda}{9}-\frac{\gamma}{\beta}.
\]
Hence when $f$ is large enough so that~\eqref{eq:push_alpha} holds and $\beta>\beta_*$ then also
\[
f' > f \beta \left( - \frac{3\lambda }{2} + \frac{4\lambda}{9} - \frac{3 \gamma}{\beta} \right) = f \left( -\frac{19\lambda\beta}{18} - 3\gamma \right).
\]
In particular, $f'$ is positive for $f$ large enough and $\gamma$ small enough.
Hence if $f$ is unbounded then eventually it is increasing monotonically.
Then by~\eqref{eq:push_alpha} $\alpha'$ is negative and bounded away from~0,
but this is impossible since the solution has infinite forward-lifetime and satisfies $\alpha \geq -1$.
\end{proof}

\begin{theorem}
\label{thm:udhav_end}
Any shrinker such that $\frac{x}{y} \to \infty$ is forward complete, with
$\alpha \to -\frac{3}{4}$, $\beta \to \beta_*$, $\gamma \to 0$ as $t \to \infty$. 
\end{theorem}

\begin{proof}
Since by
Theorem~\ref{thm:y:x:eventually:mono}, $\frac{x}{y}$ is eventually monotonic, 
Proposition \ref{prop:trap} implies forward completeness,
and also that $(\alpha,\beta)$ eventually remains in $R$.
Then Lemma \ref{lem:boundS} implies that $\frac{S}{x^2} < C\frac{1}{\beta}$ for some $C>0$, or equivalently, that 
\[
\alpha> \frac{2}{3} \left(-1  - \frac{C}{\beta} + \frac{y^2}{x^2} \right) > \frac{2}{3} \left(-1  - \frac{C}{\beta} \right).
\]

Now suppose for a contradiction that $\beta$ is unbounded above.
Whenever $\beta$ is large enough, the previous inequality
implies that $\alpha$ cannot be too much less than $-\frac{2}{3}$, say $\alpha > -\frac{17}{24}$.
(Here the value $-\frac{17}{24}$ is just a convenient point in the interval $(-\frac{3}{4},-\frac{2}{3})$). 
At any such point $\beta' > \frac{1}{12} - \frac{y^2}{x^2}$, so if $\beta$ is unbounded and $\frac{y}{x} \to 0$ then $\beta$ must
eventually be linearly increasing.
But then $f= \frac{S}{x^2} \beta>0$ bounded above and $\beta$ increasing linearly implies that $\frac{f}{\beta^2}$ is integrable and hence
\[ \frac{d}{dt}\log \frac{x}{y} = \frac{S}{xy^2} = \frac{f}{\beta^2}\]
implies that $\frac{x}{y}$ is bounded above as $t \to \infty$, which is a contradiction. 
Thus $\beta$ has an upper bound $\beta_0$.

Now suppose for a contradiction that $(\alpha, \beta)$ does not converge to $p_0 = (-\frac{3}{4}, \beta_*)$. Since $(\alpha, \beta)$ remains in the closed  bounded region $R_0 := \{(\alpha,\beta) \in R : \beta \le  \beta_0\}$, there must then exist some accumulation point
$q  \in R_0 \setminus \{p_0\}$,
\ie there is a sequence of times $t_i \to \infty$ such that $(\alpha(t_i), \beta(t_i)) \to q$.

Now consider the solution $(\bar\alpha,\bar\beta)$ to the limit system with initial value $q$ at time $t = 0$.
By Lemma \ref{lem:heisenberg_trap} that solution must leave $R_0$ in finite time, \ie $(\bar\alpha(s), \bar\beta(s)) \not \in R_0$ for some $s > 0$. Then, 
by continuous dependence on initial conditions, the $(\alpha,\beta)$-projection of any solution to the full system with initial condition sufficiently close to $(q,0)$ must also leave $R_0$ by time $s$. In particular, for some sufficiently large $i$, the original forward-complete solution must have $(\alpha(t_i+s), \beta(t_i + s)) \not \in R_0$, giving the desired contradiction.
\end{proof}

\section{\texorpdfsunitary{3}-invariant solitons}
\label{SS:su3}

The main goals of this section are
to explain Remark~\ref{rmk:trichotomy:relevance} about~\sunitary{3}-invariant expanders on~$\Lambda^2_- \CP^2$,
and some of our other expectations about~\sunitary{3}-invariant solitons more generally. 
However, actually proving these expectations remains challenging. Partly the added difficulty
compared with the \Sp{2}-invariant case is a matter of the~\sunitary{3}-invariant soliton ODE system
involving 5 rather than 3 variables; this means we cannot get as much
mileage out of tracking a single function like $\frac{y}{x}$
in \S\ref{sec:warping}. Moreover, we do not currently know a good analogue
in the \sunitary{3}-invariant case of the monotonic quantity from
Lemma \ref{lem:monotone} that played a key role in a number of proofs.

\subsection{\texorpdfsunitary{3}-invariant \texorpdfgtstr s and solitons}
\label{subsec:su3}

First we need to recall from \cite[\S4]{Haskins:Nordstrom:g2soliton1} the main
features of the~\sunitary{3}-invariant setup and explain its relation to
the~\Sp{2}-invariant solitons considered in this paper. (We refer the reader to that paper for further details as needed.)

In the setting of \Sp{2}-invariant~\gtstr s with principal
orbit~$\CP^3$, the variables $x$ and $y$
we used could be interpreted geometrically as the scale of the fibre $\Sph^2$ and
base $\Sph^4$ respectively 
of the homogeneous twistor fibration $\Sph^2 \to \CP^3 \to \Sph^4$. 
Any complete~\Sp{2}-invariant closed~\gtstr~has a unique singular orbit $\Sph^4 = \Sp{2}/\unitary{1} \times \Sp{1}$ 
and hence must be defined on $\Lambda^2_-\Sph^4$. The coefficient $x$,  controlling the $\Sph^2$ fibre size, must vanish  at a certain rate as we approach the zero section 
and the limiting value of $y$ determines the size of the zero section $\Sph^4 \subset \Lambda^2_-\Sph^4$. 

In the~\sunitary{3}-invariant setting, the principal orbit is instead
the flag manifold $F_{1,2} = \sunitary{3}/\T^2$, \ie the quotient
of~\sunitary{3} by its diagonal subgroup $\T^2$. $F_{1,2}$ admits three
distinct homogeneous $\Sph^2$-fibrations over $\CP^2$ corresponding to three
different~\unitary{2}-subgroups of~\sunitary{3}, each containing~$\T^2$.
The closed \sunitary{3}-invariant \gtstr s on $I \times \sunitary{3}/\T^2$
such that the coordinate field $\pd{}{t}$ on the interval $I$ factor has unit length can be parametrised by triples
$f_1, f_2, f_3 : I \to \R_+$ (corresponding to the scales 
of the three different $\Sph^2$ fibres) such that 
\begin{equation}
\label{eq:su3closed}
\frac{d}{dt} (f_1 f_2 f_3) = \frac{1}{2} (f_1^2 +f_2^2 +f_3^2).
\end{equation}
To extend smoothly to a singular orbit
one of the $f_i$ must vanish at the corresponding boundary
point of $I$, while the other two must take the same non-zero value $f_j = f_k = b > 0$.
Then the singular orbit is a $\CP^2$ (with volume proportional to $b^4$),
and the \gtstr{} is defined on a neighbourhood of the zero section in
$\Lambda^2_- \CP^2$. The choice of which $f_i \to 0$ on the singular orbit $\CP^2$ corresponds to three
topologically different ways to fill in $I \times \sunitary{3}/\T^2$ by
a $\CP^2$, detected for instance, by the image of $H_4(\CP^2) \cong \Z$ in
$H_4(\sunitary{3}/\T^2) \cong \Z^2$.

If $f_i=f_j$ for some $i \neq j$ then the \gtstr{} on $I \times F_{1,2}$ 
admits an additional free isometric involution $\sigma_k$ corresponding
to the fibrewise antipodal map of one of the three $\Sph^2$-fibrations
over $\CP^2$.
If the \gtstr{} extends smoothly to (one of the three versions of)
$\Lambda^2_- \CP^2$ then $\sigma_k$ corresponds to multiplying the
fibres of $\Lambda^2_- \CP^2$ by $-1$ (while the other two involutions do not extend
to the zero section).

\begin{remark}
Because the involutions $\sigma_k$ are orientation-reversing, they cannot literally preserve any~\gtstr. Rather, for an \sunitary{3}-invariant \gtstr{}
$\varphi$ with $f_i = f_j$ we have $\sigma_k^*\varphi = -\varphi$.
For brevity though we will nevertheless call these \gtstr s ``$\sunitary{3} \times \Z_2$-invariant'' (or ``$\sunitary{3} \times \langle \sigma_k\rangle$-invariant'' if we want to emphasise which of the three involutions acts isometrically).
\end{remark}

If $f_1=f_2=f_3$ then the three resulting involutions $\sigma_1, \sigma_2, \sigma_3$
generate an action by
the symmetric group~$S_3$. Then the closure condition \eqref{eq:su3closed}
implies $f_i = \frac{t}{2}$, which defines the unique
\sunitary{3}-invariant torsion-free \gtwo-cone.
However, the Bryant--Salamon ~\sunitary{3}-invariant complete AC
torsion-free~\gtstr s on $\Lambda^2_-\CP^2$ 
that are asymptotic to this torsion-free \gtwo-cone have only an additional
$\Z_2$ symmetry. Thus its asymptotic cone has an isometry of order 3 that does not extend (even as a diffeomorphism) to the zero section.

\begin{remark}
\label{rmk:transition}
One can interpret the discussion above as showing the existence of three topologically different
AC \gtwo-manifolds all asymptotic to the unique~\sunitary{3}-invariant torsion-free~\gtwo-cone. 
This can be seen as analogous to the existence of two topologically distinct small
resolutions of the complex 3-dimensional ordinary double point singularity
(called the conifold in the physics literature), related by a `flop'.
Atiyah--Witten~\cite[\S 2.3]{Atiyah:Witten:Mtheory} physically interpret
the relationship between the three AC $G_2$ holonomy spaces sharing the same
asymptotic cone as more similar to the `conifold transition' between a small
resolution and a deformation of the conifold.
\end{remark}

\medskip
The soliton ODEs for \sunitary{3}-invariant \gtstr s on
$I \times \sunitary{3}/\T^2$ can be formulated as a first-order ODE system for
$f_1, f_2, f_3$ and two further real variables encoding the torsion $2$-form  of the closed~\gtstr.
 The phase space $\mathcal{P}$ for~\sunitary{3}-invariant solitons is therefore $5$-dimensional, and hence for any fixed
$\lambda$ the space of local~\sunitary{3}-invariant solitons is $4$-dimensional
($\lambda\not= 0$ fixes the scale, but in the steady case we should really
think of the family as 3-dimensional up to scale).

The involutions $\sigma_k$ act on the 5-dimensional phase space, each fixing
a 3-dimensional submanifold that is invariant under the soliton ODE system.
The soliton ODEs on these $3$-dimensional invariant manifolds reduce to
the ODEs \eqref{eq:ODE:tau2'} for the \Sp{2}-invariant soliton system
if we set $y=f_i=f_j$ and $x=f_k$ (note, for instance, that the closure
condition \eqref{eq:su3closed} specialises to \eqref{eq:closure} in this
situation). Thus solutions to \eqref{eq:ODE:tau2'} have two geometric
interpretations: 
either $\sunitary{3}\times \Z_2$-invariant solitons on $I \times  \sunitary{3}/\T^2$ (in three different ways, depending on which $f_k$ we identify with $x$)
or~\Sp{2}-invariant solitons on $I \times \CP^3$. 

For each fixed dilation constant $\lambda$, there is a 2-parameter
family of smoothly-closing~\sunitary{3}-invariant solitons defined near the zero section
of $\Lambda^2_- \CP^2$. 
This family can be parametrised by a pair $(b,c)$ where $b>0$ and $c\in \R$; 
the parameter $b$ controls the size of the singular orbit $\CP^2$ and
the involution that multiplies the fibres of $\Lambda^2_- \CP^2$ by $-1$
acts as $(b,c) \mapsto (b,-c)$.
Following the pattern above, we can identify this with a family
of solutions $\scfam{\lambda}{b,c}{k}$ on $(0,\epsilon) \times \sunitary{3}/\T^2$
in three different ways, depending on which $f_k \to 0$ as $t \to 0$
(while $f_i$ and $f_j \to b$).

\begin{remark}
\label{rmk:scaling_su3}
Analogously to Remark \ref{rmk:scaling_sc}, for any $\mu>0$ the smoothly-closing soliton
$\scfam{\mu^{-2}\lambda}{\mu b, \mu c}{k}$ is a rescaling of
$\scfam{\lambda}{b,c}{k}$.
We can therefore think of the smoothly-closing solutions as a 2-parameter
family up to scale, with scale-invariant parameters $\lambda b^2$ and
$\frac{c}{b}$; for steady solitons they form a $1$-parameter family with scale-invariant parameter $\frac{c}{b}$. 
\end{remark}

The conditions for an~\Sp{2}-invariant soliton to close smoothly
on $\Sph^4$ turn out to be the same as for the corresponding
$\sunitary{3} \times \Z_2$-invariant soliton to close smoothly on $\CP^2$.
In particular, the $1$-dimensional subfamily $\scfam{\lambda}{b,0}{k}$,
which is invariant under the involution $\sigma_k$, corresponds to the
$1$-parameter family of smoothly-closing~\Sp{2}-invariant solitons
$\spfam{\lambda}{b}$ from Theorem \ref{thm:Sp2:smooth:closure}.

\begin{example}
\label{ex:explicit_su3shrinker}
The explicit~\Sp{2}-invariant shrinker solution from
Example \ref{ex:explicit_shrinker} 
gives rise to a complete~$\sunitary{3}\times \Z_2$-invariant AC shrinker on
$\Lambda^2_-\CP^2$.
Its asymptotic cone is closed but not torsion-free;
unlike in the torsion-free case this cone is invariant only under $\Z_2$ and not all
of~$S_3$. One should therefore think of there being three distinct closed
conical \gtstr s on $\Rpos \times \sunitary{3}/T^2$ that are the asymptotic limits
of three explicit AC shrinkers, that close smoothly on three topologically different
$\CP^2$s. In view of Remark \ref{rmk:explicit_shrinker}, up to scale they
correspond to the smoothly-closing solutions $\scfam{-1}{\frac32,0}{k}$,
where $\scfam{-1}{\frac32,0}{1}$ is defined by
\[ f_1 = t, \quad f_2 = f_3 = \tfrac12 \sqrt{9 + t^2} . \]
\end{example}

\begin{example}
\label{ex:SU3:Z2:expanders}
For $\lambda>0$, it follows from Theorem \ref{thm:sc:expanders:complete} that
$\scfam{\lambda}{b,0}{k}$ defines a complete~$\sunitary{3}\times\Z_2$-invariant AC expander on $\Lambda^2_-\CP^2$ for every $b > 0$.
Moreover, every complete~$\sunitary{3}\times\Z_2$-invariant expander on $\Lambda^2_-\CP^2$ is actually AC. 
\end{example}

\subsection{Completeness of \texorpdfsunitary{3}-invariant solitons}

We would like to understand more generally, for each $\lambda$, which
elements of $\scfam{\lambda}{b,c}{k}$ extend to complete solitons
on~$\Lambda^2_-\CP^2$.
In the steady case $\lambda=0$, two of the present authors gave a definitive answer to this completeness question
in~\cite[Theorem~G]{Haskins:Nordstrom:g2soliton1}.

\begin{theorem}
\label{thm:steadies}
The smoothly-closing soliton $\scfam{0}{b,c}{k}$ is complete and AC (to the
torsion-free cone, with rate $-1$)
for $\frac{1}{b} \abs{c} < \frac{3}{\sqrt{2}}$,
and is forward-incomplete for $\frac{1}{b} \abs{c} > \frac{3}{\sqrt{2}}$.
The borderline case of~$\frac{1}{b} \abs{c} = \frac{3}{\sqrt{2}}$ defines a
complete steady soliton (unique up to scale) with exponential volume growth.
\end{theorem}

\begin{remark}
The steady soliton $\scfam{0}{1, \tfrac{3}{\sqrt{2}} }{1}$ can be
written explicitly as 
\[
f_1 = 2 \sinh{\tfrac{t}{2}}, \quad f_2 = \sqrt{1 + e^t} \quad  f_3  = \sqrt{1 + e^{-t}}.
\]
In particular, $f_3$ is bounded as $t \to \infty$, while $f_1$ and $f_2$
grow exponentially with $\frac{f_1}{f_2} \to 1$. Informally we could say that
the end is ``approximately invariant'' under the involution $\sigma_3$.
By \cite[Theorem~F]{Haskins:Nordstrom:g2soliton1}, any forward-complete
non-AC steady end  has such an approximate $\Z_2$ symmetry.
However, for the complete non-AC steady solitons this symmetry does not extend
to a diffeomorphism of $\Lambda^2_- \CP^2$
(for~$\scfam{0}{1, \tfrac{3}{\sqrt{2}}}{1}$ it is $\sigma_1$
that acts as $-1$ on the fibres of $\Lambda^2_- \CP^2$).
\end{remark}

As in the~\Sp{2}-invariant case it turns out to be instructive to try to construct~\sunitary{3}-invariant AC expander and shrinker ends asymptotic to closed~\sunitary{3}-invariant cones. 
If we set $f_1=c_1 t$, $f_2 = c_2 t$, $f_3 = c_3 t$ for some positive triple $(c_1,c_2,c_3) \in \R^3_+$ then the associated $3$-form $\varphi$ 
is conical with vertex at $t=0$ and the closure condition \eqref{eq:su3closed}
on $\varphi$ becomes
\begin{equation}
\label{eq:su3closed_cone}
6 c_1 c_2 c_3 = c_1^2+c_2^2+c_3^2.
\end{equation}
It follows that each homothety class of positive triples $[c_1,c_2,c_3]$ contains a unique closed cone, 
so the space of~\sunitary{3}-invariant closed cones is $2$-dimensional and moreover is diffeomorphic to $\Sph^2 \cap \R^3_+$.
The unique torsion-free cone has $c_1=c_2=c_3=\frac{1}{2}$. 

Like in the~\Sp{2}-invariant case, the problem of constructing~\sunitary{3}-invariant AC expander and AC shrinker ends asymptotic to a given closed~\sunitary{3}-invariant cone
can be recast as an irregular singular initial value ODE problem, whose qualitative properties depend strongly on the sign of~$\lambda$; 
the properties of such AC ends lead to heuristics
for the dimensions of the spaces of complete AC~\sunitary{3}-invariant solitons analogous to those
in \S\ref{subsec:end_sivp}.

For any fixed $\lambda<0$, we are able to prove the existence of a unique~\sunitary{3}-invariant AC shrinker end
asymptotic to any given closed~\sunitary{3}-invariant cone. In particular, the space of AC shrinker ends has codimension $2$ 
within the $4$-dimensional space of all local~\sunitary{3}-invariant shrinkers. 
Thus the expected dimension of the space of complete~\sunitary{3}-invariant AC shrinkers is again zero, 
being the intersection of a pair of codimension $2$ subsets within the $4$-dimensional space of all local~\sunitary{3}-invariant shrinkers. 

\begin{conjecture}
\label{conj:finite}
For fixed $\lambda < 0$, there are finitely many
complete \sunitary{3}-invariant AC shrinkers.
\end{conjecture}

The set of complete AC shrinkers is certainly non-empty because of the
explicit~$\sunitary{3}\times \Z_2$-invariant AC shrinkers on $\Lambda^2_-\CP^2$ already mentioned in Example~\ref{ex:explicit_su3shrinker}; at present, 
we do not know if there are any further complete~\sunitary{3}-invariant AC shrinkers that do not possess such an additional isometric involution.

On the other hand,  for any fixed $\lambda>0$, we can prove the existence of a $2$-parameter family 
of~\sunitary{3}-invariant AC expander ends asymptotic to any given closed~\sunitary{3}-invariant cone. As in the~\Sp{2}-invariant case 
any two AC expander ends asymptotic to the same cone share the same Taylor polynomial, but differ by exponentially-small terms. 

\begin{remark}
In the case of a $\sunitary{3} \times \Z_2$-invariant cone, say
with $c_1 = c_2$, the subfamily of $\sunitary{3} \times \Z_2$-invariant expander
AC end solutions given in Example~\ref{ex:SU3:Z2:expanders} corresponds to the 1-parameter family discussed in Section~\ref{subsec:end_sivp}.
Thus most members of the 2-parameter family of \sunitary{3}-invariant end solutions
with that asymptotic cone are not invariant under the involution. 
This demonstrates that discrete isometries of the asymptotic cone
need not be inherited by AC expander ends, in contrast to what Theorem \ref{thm:HKP} tells us about (gradient) AC shrinker ends.
This leaves open the question of whether or not continuous isometries 
of the asymptotic cone are inherited by AC expander ends. 
(In Ricci flow, Chodosh~\cite{Chodosh} proved that AC expanders on $\R^n$ with positive curvature are $\orth{n}$-invariant 
and therefore must be Bryant expanders~\cite{Bryant:expanders}. 
Also it is known in several cases that AC torsion-free~\gtstr s inherit the continuous symmetries of their asymptotic cones~\cite{Karigiannias:Lotay:conicalG2}.)
\end{remark}

Since the set of \sunitary{3}-invariant closed cones is also 2-dimensional,
all in all we have a 4-parameter family of \sunitary{3}-invariant AC expander
end solutions, suggesting that these form an open subset in the space of
all~\sunitary{3}-invariant local expanders 
(as we have proved is true in the~\Sp{2}-invariant setting). 
As mentioned in Remark \ref{rmk:lending}, the argument in Proposition
\ref{prop:stable} can be generalised to the case of~\sunitary{3}-invariant expanders, and this establishes  that the AC end property is indeed stable in that
setting too. In particular, the existence of the $1$-parameter family of  $\sunitary{3}\times \Z_2$-invariant  complete AC expanders given in Example~\ref{ex:SU3:Z2:expanders}, 
together with the stability of AC expander ends, implies the existence of a
$2$-parameter family of complete AC \sunitary{3}-invariant expanders
(and the generic member of this family will not possess any additional
isometric involution). However, this argument does not address whether all
smoothly-closing~\sunitary{3}-invariant expanders are complete, or whether 
a complete~\sunitary{3}-invariant expander is necessarily AC.

To discuss the possible end behaviours of~\sunitary{3}-invariant expanders more generally, it is helpful to first note that
we can adapt the other two end constructions from~\S\ref{subsec:end_sivp}
to the \sunitary{3}-invariant setting. Indeed, 
for any fixed $\lambda$, \cite[Theorem 6.23]{Haskins:Nordstrom:g2soliton1}
provides a 4-parameter family of forward-incomplete end solutions analogous
to those in (iii) (which contains a corresponding 2-parameter family
with extra $\Z_2$-invariance). In particular, this type of forward incompleteness is
still stable (for any value of $\lambda$). 

Similarly, for a fixed $\lambda > 0$, we obtain a 3-parameter family of~\sunitary{3}-invariant 
forward-complete expander ends with quadratic-exponential volume growth, analogous to the ones
in (ii). More precisely, for each $A > 0$, setting $f_1 = \hat x$,
$f_2 = f_3 = \hat y$ from \eqref{eq:explicit_limit_sol} 
gives an approximate end solution, and then there is a 2-parameter
family of genuine end solutions that are a bounded distance away from this model. 
Note that all solutions of this type have $\frac{f_2}{f_3} \to 1$ as $t \to \infty$, so they all have an ``approximate $\Z_2$ symmetry'' on the end, but for each $A > 0$ there is only one solution with $f_2 = f_3$.

It is thus plausible that there should be a similar trichotomy for
\sunitary{3}-invariant expander ends as for \Sp{2}-invariant ones.
However, unlike for the \Sp{2}-invariant case, we expect that the non-AC
end behaviours actually arise even for smoothly-closing solutions.

\begin{conjecture}
\label{conj:transition}
Fix $\lambda > 0$. Then $\scfam{\lambda}{b,c}{i}$ is forward complete and AC if
$|c| \leq \frac{3}{\sqrt{2}}b$.
For each $k > \frac{3}{\sqrt{2}}$, there is a unique $\mu > 0$ such that
$\scfam{\lambda}{b, kb}{i}$ is forward complete and AC for $b > \mu$,
incomplete for $b < \mu$, and complete non-AC for $b = \mu$.
\end{conjecture}

Numerical calculations give strong evidence of the existence of some transition like this between AC and incomplete
forward-evolution of the smoothly-closing expanders. We also have a heuristic based on a comparison with the behaviour of
smoothly-closing~\sunitary{3}-invariant  steady solitons (stated earlier in Theorem
\ref{thm:steadies}).
On the one hand, for each fixed $k$, the~\sunitary{3}-invariant version of Proposition \ref{prop:stable} should imply
AC end behaviour of the smoothly-closing $\lambda$-expander with parameters
$(b, k b)$ whenever $b$ is sufficiently large.

On the other hand, for any $\mu>0$, Remark \ref{rmk:scaling_su3} means that 
the smoothly-closing $\lambda$-soliton $\scfam{\lambda}{\mu b, \mu c}{i}$ is a
rescaling of $\scfam{\mu^2 \lambda}{b, c}{i}$.
Similarly to the argument used in the proof of Lemma \ref{lem:sid*:bto0},
we now compare with a steady soliton as $\mu \to 0$.
By Theorem \ref{thm:steadies}, $\scfam{0}{b, c}{i}$ is forward incomplete when
$k = \frac{|c|}{b} > \frac{3}{\sqrt{2}}$.
Because of the stability of this type of finite extinction behaviour, 
for fixed $\lambda > 0$ and any fixed $(b,c)$ with
$\abs{c} > \frac{3}{\sqrt{2}} b$, 
we expect that
$\scfam{\lambda}{\mu b, \mu c}{i}$
has finite extinction time for $\mu$ sufficiently small.

\subsection{Asymptotic cones of AC \texorpdfsunitary{3}-invariant expanders}
\label{subsec:su3cones}

The final question we wish to discuss is which 
cones arise as the asymptotic limit of a complete~\sunitary{3}-invariant AC expander; 
of particular interest to us is whether the asymptotic cones of the explicit
$\sunitary{3} \times \Z_2$-invariant AC shrinkers from
Example \ref{ex:explicit_su3shrinker} are realised this way.

Recall that the~\sunitary{3}-invariant closed cones are parametrised by positive triples $c_1, c_2, c_3$ satisfying
the scale-normalisation \eqref{eq:su3closed_cone} imposed by the
closure condition. 
To make the previous question precise, we should restrict attention to complete expanders
that extend smoothly to a $\CP^2$ orbit in one of the three possible ways,
\ie we should prescribe which $f_i$ goes to 0 at the special orbit; 
changing that prescription will permute the triples $(c_1,c_2,c_3)$ of the asymptotic cones realised.
For concreteness, we make the choice that $f_1 \to 0$ as $t \to 0$. 

\begin{conjecture}
\label{conjecture:SU3:hit:cones}
For $\lambda > 0$, the closed cone with parameters $c_1, c_2, c_3$ appears
as the asymptotic limit of an AC member of the family of smoothly-closing
expanders $\scfam{\lambda}{b,c}{1}$ if and only if $c_1 < \max(c_2, c_3)$, \ie if $c_1$
is not the largest of the $c_i$.
\end{conjecture}

Once again some of our evidence for this conjecture comes from numerical
calculation. Another heuristic reason to believe that the (open) set of cones that
are realised as asymptotic limits of complete AC expanders is bounded by $c_1 = c_2$ and $c_1 = c_3$ is that
the quadratic-exponential growth expander end solutions that we expect to appear
as we approach the boundary of the space of AC solutions all have an approximate $\Z_2$
symmetry at their end.

\smallskip
Since each ray in the positive octant contains a unique triple $c_1, c_2, c_3$
satisfying \eqref{eq:su3closed_cone}, by radially projecting from the set of closed cones
to the plane $c_1 + c_2 + c_3 = 1$ we can visualise the set of closed cones
as the interior of an equilateral triangle, see Figure \ref{fig:su3cones}.
This radial projection map preserves the ordering of the $c_i$ and preserves
equality of any pair $c_i=c_j$. 
Points close to an edge of this triangle correspond to cones where
one of the $c_i$ is much smaller than the other two.

The centre of the triangle corresponds to the torsion-free cone
with $c_1 = c_2 = c_3 = \frac12$. The three lines of symmetry (where two of the
$c_i$ are equal) correspond to cones that are invariant under one of the three
involutions $\sigma_k$. The shaded region, bounded by the segments
$c_1 = c_2 > c_3$ and $c_1 = c_3 > c_2$, is what we conjecture arises as
asymptotic cones of the AC members of the family of smoothly-closing expanders $\scfam{\lambda}{b,c}{1}$.
The bold segment $c_2 = c_3 > c_1$ is what we know arises as the
limits of the $\sunitary{3} \times \langle \sigma_1\rangle$-invariant
expanders $\scfam{\lambda}{b,0}{1}$.

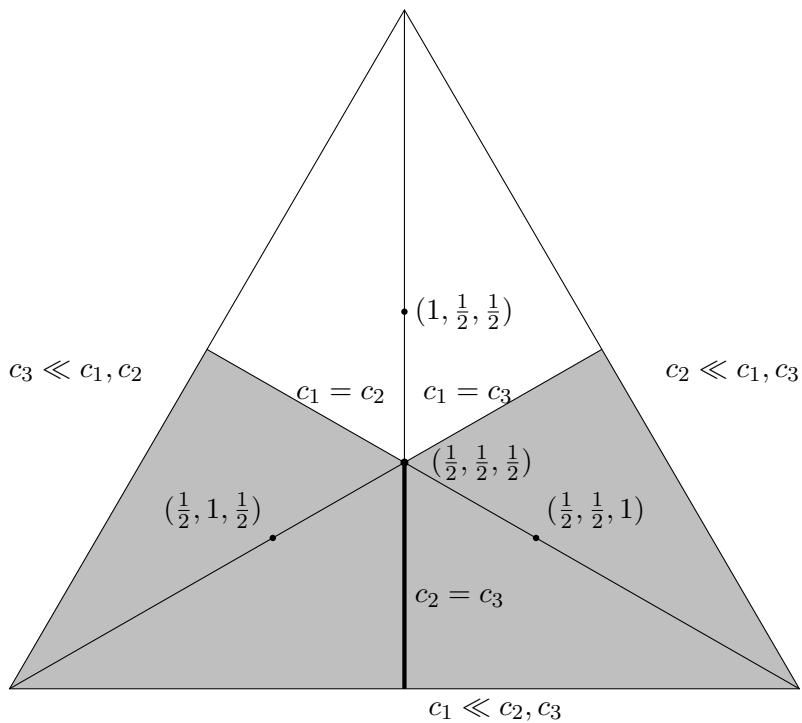
\begin{figure}
    \begin{tikzpicture}[scale=6]
      \fill[color=lightgray] (0,0) -- (0.433,0.25) -- (0.866,-0.5) -- (-0.866,-0.5) -- (-0.433,0.25) -- cycle ;
      \fill (0,0) circle(0.25pt) node[right] {$\;\;(\frac12, \frac12, \frac12)$};
      \fill (0,0.3333) circle(0.2pt) node[right] {$(1, \frac12, \frac12)$};
      \fill (-0.2886,-0.1667) circle(0.2pt) node[above left] {$(\frac12, 1, \frac12)$};
      \fill (0.2886,-0.1667) circle(0.2pt) node[above right] {$(\frac12, \frac12, 1)$};
      \draw (0.866,-0.5) -- (0,1) -- (-0.866,-0.5) -- cycle ;
      \draw (0,-0.5) -- (0,1); 
      \draw[ultra thick] (0,0) -- (0,-0.5); 
      \draw (0.866,-0.5) -- (-0.433,0.25); 
      \draw (-0.866,-0.5) -- (0.433,0.25); 
       \node[below] at (0.2,-0.5) {$c_1 \ll c_2, c_3$}; 
       \node[right] at (0.55,0.2) {$c_2 \ll c_1, c_3$}; 
       \node[left] at (-0.55,0.2) {$c_3 \ll c_1, c_2$}; 
       \node[right] at (0,-0.3) {$c_2 = c_3$}; 
       \node[right] at (-0.26,0.15) {$c_1 = c_2$}; 
       \node[left] at (0.26,0.15) {$c_1 = c_3$}; 
    \end{tikzpicture}
  \caption{Closed \sunitary{3}-invariant cones}
\label{fig:su3cones}
\end{figure}

The closed cone with parameters $(1,\frac12, \frac12)$ is the asymptotic cone
of the explicit $\sunitary{3} \times \langle \sigma_1\rangle$-invariant AC
shrinker $\scfam{-1}{\frac32,0}{1}$ from Example \ref{ex:explicit_su3shrinker}.
According to our conjecture this cone is \emph{not} the asymptotic cone of any
of the AC expanders in the family $\scfam{\lambda}{b,c}{1}$ ($\lambda > 0$),
\ie there is no complete AC expander on ``the same''
$\Lambda^2_-\CP^2$ with this asymptotic cone.

However, the point $(\frac12, \frac12, 1)$, which corresponds to the asymptotic
cone of the  $\sunitary{3} \times \langle \sigma_3\rangle$-invariant shrinker
$\scfam{-1}{\frac32,0}{3}$, \emph{is} in the conjectured region. So too is
$(\frac12, 1,  \frac12)$, the asymptotic cone of the
$\sunitary{3} \times \langle \sigma_2\rangle$-invariant shrinker
$\scfam{-1}{\frac32,0}{2}$.

\begin{remark*}
More generally, if Conjectures~\ref{c:shrinker:unicity} and~\ref{conjecture:SU3:hit:cones} are both true then for \emph{any}~\sunitary{3}-invariant  complete AC shrinker on $\Lambda^2_-\CP^2$ there would be exactly two~\sunitary{3}-invariant~complete AC expanders whose cones match those of that shrinker.
In fact, we not need the full strength of Conjecture~\ref{c:shrinker:unicity} to draw this conclusion, only the weaker result that
no~\Sp{2}-invariant complete AC shrinker $(x,y)$ on $\Lambda^2_-\Sph^4$ has $\ell = \lim_{t \to \infty} \frac{y}{x}>1$
(whereas Conjecture~\ref{c:shrinker:unicity} asserts that only $\ell = \frac{1}{2}$ is possible). 

Moreover, if there were any~\Sp{2}-invariant complete AC shrinker $(x,y)$ on $\Lambda^2_-\Sph^4$ with $\ell>1$, then 
that would give rise to an~\Sp{2}-invariant weak solution to Laplacian flow: it is the AC shrinker for $t<0$, its asymptotic cone at $t=0$ 
and for $t>0$ it is the unique (up to scale)~\Sp{2}-invariant complete AC expander with the given value of $\ell>1$ 
(the existence and uniqueness of such an AC expander being guaranteed by Theorem~\ref{mthm:asymptotic:limit}).
\end{remark*}

So if the Laplacian flow were to develop a conical singularity modelled on the
explicit AC shrinker on~$\Lambda^2_-\CP^2$, then there are \emph{prima facie} two ways to try to
smoothly continue the flow, namely by gluing in either of these two complete AC expanders (defined on the other two versions of~$\Lambda^2_-\CP^2$). 
In particular, this suggests that to be able to smoothly continue the flow after this conical singularity has developed 
Laplacian flow must necessarily undergo a~\gtwo-conifold transition as in Remark \ref{rmk:transition}.

If true, this would also tell us that the behaviour of AC expanders in Laplacian flow is different from in K\"ahler-Ricci flow. In the latter setting, 
a complete AC K\"ahler expander $E$ coming out of a K\"ahler cone $C$ is unique (up to pullback by biholomorphisms)
and $E$ must be the unique minimal model of $C$ (which here is assumed to be smooth)~\cite[Corollary B]{Conlon:Deruelle:Sun:Kahler:solitons}.
On the other hand, in general Riemannian Ricci flow, much more severe non-uniqueness of AC expanders was recently proven to 
occur, even in the cohomogeneity-one setting, and moreover this non-uniqueness 
can occur without any topology change~\cite{Angenent:Knopf}. 

\bibliographystyle{amslink}
\bibliography{g2soliton}

\end{document}